\documentclass{mod}
\usepackage{url}
\usepackage{setspace}
\usepackage{scrextend}
\usepackage{cite}
\usepackage[shortlabels]{enumitem}
\usepackage{hyperref}
\usepackage{caption}
\usepackage{subcaption}

\usepackage{amsmath}
\usepackage{amsthm}
\usepackage{amssymb}

\usepackage{mathtools}
\usepackage{mathrsfs}

\usepackage[all]{xy}
\usepackage[toc,page]{appendix}
\usepackage{titlesec}
\usepackage{upgreek}

\usepackage{tocloft}

\usepackage{tocbasic}
\DeclareTOCStyleEntry[
beforeskip=.1em plus 1pt,
entryformat=\normalfont ,
pagenumberformat=\bfseries
]{tocline}{section}
\usepackage{etoolbox}
\patchcmd{\thebibliography}{%
	\section*{\refname}\@mkboth{\MakeUppercase\refname}{\MakeUppercase\refname}}{%
	\section*{\refname}}{}{}

\usepackage{ stmaryrd }

\newcounter{mainthm}

\newcounter{mainconj}

\newtheorem{thm}{Theorem}[section]
\newtheorem{observation}[thm]{\textit{\underline{Observation}}}

\theoremstyle{plain}
\newtheorem{lem}[thm]{Lemma}
\newtheorem{prop}[thm]{Proposition}
\newtheorem{cor}[thm]{Corollary}
\newtheorem{defn-thm}[thm]{Definition-Theorem}

\newtheorem{defn-lem}[thm]{Definition-Lemma}

\newtheorem{conjecture}
[thm]{Conjecture}

\newtheorem{defn}[thm]{Definition}
\newtheorem{convention}
[thm]{Convention}

\theoremstyle{definition}

\newtheoremstyle{rmk}
{5pt}
{5pt}
{}
{}
{\itshape}
{}
{.5em}
{}

\newtheorem{rmk}[thm]{Remark}

\newtheoremstyle{note}
{8pt}
{5pt}
{\itshape}
{10pt}
{\bfseries}
{}
{.5em}
{}

\theoremstyle{note}

\setlist[description]{font=
	\normalfont
	\itshape
	\space}

\usepackage{titletoc}

\titlecontents{section}
[0.72em]
{\scriptsize\bfseries\vspace{3pt}}%
{\thecontentslabel.\enspace}
{}
{\titlerule*[0.38pc]{.}\contentspage}%

\titlecontents{subsection}
[0em]
{\scriptsize \vspace{0pt}}%
{\qquad \quad \thecontentslabel.\enspace}
{}
{\titlerule*[0.5pc]{.}\contentspage}%

\titlecontents{subsubsection}
[0em]
{\footnotesize \vspace{1pt}}%
{\qquad \qquad \thecontentslabel.\enspace}
{}
{\titlerule*[0.5pc]{.}\contentspage}%

\titleformat{\subsubsection}[runin]{
	\bfseries
	\itshape\normalsize}{\thesubsubsection \ }{0em}{}[\mbox{ . } ]
\titlespacing{\subsubsection}
{0pt}
{2.25ex plus 1ex minus .2ex}
{0pt}

\newcommand{\mathds}[1]{\text{\usefont{U}{dsrom}{m}{n}#1}}
\newcommand{\one}{\mathds {1}}
\newcommand{\trop}{ {\mathfrak{trop}} }
\DeclareMathOperator{\val}{\mathsf{v}}

\begin{document}
	\setlength{\parindent}{15pt}	\setlength{\parskip}{0em}

	\title{Family Floer SYZ conjecture for $A_n$ singularity}
	\author{Hang Yuan}
	\date{}
	\begin{abstract} {\sc Abstract:}  
		We resolve a mathematically precise SYZ conjecture for $A_n$ singularity by building a quantum-corrected T-duality between two singular torus fibrations related to the K\"ahler geometry of the $A_n$-smoothing and the Berkovich geometry of the $A_n$-resolution, respectively.
		Our approach involves novel computations that embody a non-archimedean version of the partition of unity, and it confirms the strategy that patching verified local singularity models brings global SYZ conjecture solutions (like K3 surfaces) within reach. There is also explicit extra evidence concerning the collision of singular fibers and braid group actions.
		On one hand, we address the central challenge of matching SYZ singular loci identified by Joyce \cite{Joyce_Singularity}.
		In reality, we construct not merely an isolated SYZ mirror fibration partner, but a parameter-dependent one that always keeps the matching singular loci plus integral affine structure, even when the collision of singular fibers occurs.
		On the other hand, our SYZ result displays a visible tie, regardless of the parameter choice, between the $(A_n)$-configuration of Lagrangian spheres occurred as vanishing cycles in the $A_n$-smoothing and the exceptional locus of rational $(-2)$-curves in the $A_n$-resolution, which aligns with the celebrated works of Khovanov, Seidel, and Thomas \cite{Seidel_Khovanov2002quivers,Seidel_Thomas_2001braid,thomas2000mirror_braid,seidel1999lagrangian}.
	\end{abstract}
	\maketitle

		\tableofcontents

	\section{Introduction}

The 1996 Strominger-Yau-Zaslow (SYZ) conjecture \cite{SYZ} posits that for a pair of "mirror" Calabi-Yau manifolds, there exist "dual" special Lagrangian torus fibrations with congruent singular loci over a shared base.
A rigorous mathematical formulation of SYZ conjecture remains elusive, with the crucial step of a resolution being to accurately state the conjecture itself.

Recent major advancements in SYZ research address the existence of special Lagrangian fibrations, as studied by Y. Li's results \cite{Li2019syz,Li2020metric} and the tropicalization of SYZ picture within algebraic geometry, as explored by Gross-Hacking-Keel \cite{GHK15} and Gross-Siebert \cite{GS11,GS_canonical_wall}.
Meanwhile, the ultimate objective of mirror symmetry aims to foster a bilateral understanding and forge mathematically cogent bridges between two disparate geometric universes.
Kontsevich's homological mirror symmetry \cite{KonICM} closely aligns with this goal.
However, evidence supporting the dualistic aspects of the SYZ conjecture---especially regarding mirror fibration duality with quantum corrections and singular fibers, rather than just the identification of mirror spaces---appears to be relatively limited at present; some related evidence from differential geometry can be found in \cite{collins2020syz}.
As it is often believed that examples are to mathematics what experiments are to physics.
This paper aims to provide further exemplification to justify the correctness of the SYZ approach to mirror symmetry.

A key novelty here lies in properly addressing the SYZ fibration duality with singular fibers. Our
strategy should be also related to the classical ideas of Auroux, Chan, Lau, and Leung \cite{AuTDual,CLL12} and the recent works of Bardwell-Evans, Cheung, Hong, and Lin \cite{Cheung_Lin_some_example, bardwell2021scattering} for the scattering diagram pictures.




Mirror symmetry for the $A_n$ singularity is not a new concept, but the depth of understanding might vary. While there is a known computation matching for the HMS aspect (cf. \cite{Pomerleano2011curved_I,Chan_An_Tduality}), the geometric logic behind the mirror correspondence remains elusive. In particular, the mirror connection between the two braid group actions on both sides is not well-understood.
The work of Abouzaid-Auroux-Katzarkov \cite{AAK_blowup_toric} offers a T-duality view on this mirror space identification, but it lacks fibration duality, mirror singular fibers, or collision of singular points.
Mirror symmetry should also extend beyond just hyperk\"ahler rotation; the almost toric Lagrangian fibration on the $A_n$-resolution does not show compelling connections to the one on the $A_n$-smoothing, as neither the collision of singular points nor the braid group action is rightly discernible.

In this paper, we study a new version of the SYZ conjecture that demonstrates a duality between two torus fibrations on the $A_n$-smoothing and $A_n$-resolution within K\"ahler and Berkovich geometry respectively and shows convincing geometric phenomenon for the braid group action and the collision of singular points. An upgrade to categorical results will be addressed somewhere else.

%

\subsection{Main result}
For SYZ fibration duality, the main challenge lies in coherently addressing and explicitly representing the data of singular fibers and the affiliated quantum correction holomorphic disks. Let's briefly explain the story.
One would typically begin with a fibration $\pi: X\to B$, where the general fiber is a Lagrangian torus and the discriminant locus $\Delta\subset B$ with $B_0=B-\Delta$. The initial "dual" of $\pi_0=\pi|_{B_0}$ should be described as the dual torus fibration $f_0:  \mathscr Y_0 \cong R^1\pi_{0*}(U(1))\to B_0$. A true "dual" of $\pi$ may be a compactification or extension of $f_0$ to some $f$ as below (see Gross's introduction in \cite{Gross_topo_MS}):
\[
{
\xymatrix{
	X \ar[d]_{\pi} & X_0\ar[l] \ar[d]_{\pi_0}   \ar@{.}@/_1.15pc/[rr]_{\small \text{` T-duality '}}    & & \mathscr Y_0\ar[d]^{f_0} \ar[r]  & \mathscr Y \ar[d]^{f} \\
	B & B_0 \ar[l] \ar@{.}@/^1.15pc/[rr] & &  B_0 \ar[r] & B
}
}
\]

An appropriate dual fibration $f_0$ should preserve the integral affine structure on $B_0$ induced inherently by the Lagrangian torus fibration $\pi_0$. By the action-angle coordinates, $\pi_0$ is locally modeled on the logarithm map $\mathrm{Log}: (\mathbb C^*)^n\to \mathbb R^n$, which sends $z_j$ to $\log|z_j|$. Choosing an atlas $(U_i\to V_i)$ of the integral affine structure for a fine open covering $(U_i)$ of $B_0$ allows us to view $\pi_0:X_0\to B_0$ as an assembly of the local pieces $\mathrm{Log}^{-1}(V_i)\to V_i$.
Meanwhile, an analogous map, the tropicalization map, $\trop: (\Bbbk^*)^n \to \mathbb R^n$, exists in Berkovich geometry (e.g. \cite{EKL,NA_nonarchimedean_SYZ,KSAffine}). This sends $z_j$ in $\Bbbk$ to $-\log|z_j|_{\Bbbk}$, using the norm of a non-archimedean field $\Bbbk$.
As a \textit{toy model}, we simply claim that $((\Bbbk^*)^n, \mathfrak{trop})$ is SYZ mirror to $( (\mathbb C^*)^n,\mathrm{Log})$; then, we seek for a reasonable globalization of this local SYZ picture.


If a natural gluing of the local models $\trop^{-1}(V_i)$ were to take place, an affinoid torus fibration $f_0$ in Berkovich geometry would emerge, preserving the integral affine structure on $B_0$ automatically. Then, $f_0$ would be a plausible candidate for the SYZ dual fibration.
Further, we would like the gluing process to be systematic, exhibiting a certain level of inherent data, rather than being random.

As demonstrated in the author's thesis \cite{Yuan_I_FamilyFloer},
the quantum-correcting Maslov-0 holomorphic disks for $\pi_0$ within $X$ offer a unique canonical algorithm for gluing the local fibrations $\trop^{-1}(V_i)\to V_i$ in the category of non-archimedean analytic spaces, drawing inspiration from the pioneering ideas of Fukaya and Tu \cite{FuBerkeley,Tu,FuFamily}.
Let us call the resulting $f_0$ the \textit{canonical dual affinoid torus fibration} of $\pi_0$.
To achieve this, the Floer-theoretic basis requires the selection of the Novikov field $\Lambda=\mathbb C((T^{\mathbb R}))$ as the ground field, replacing the standard topological fibration $R^1\pi_{0*}(U(1))$ with $R^1\pi_{0*}(U_\Lambda)$ where $U_\Lambda$ is the unit circle in $\Lambda$. This is due to Gromov's compactness, a basic principle in symplectic geometry, ensures convergence solely over $\Lambda$, not $\mathbb C$. This morally validates the use of Berkovich geometry.



\begin{conjecture}
	\label{conjecture_our_SYZ}
	Given a Calabi-Yau manifold $X$,
	\begin{itemize}[topsep=1pt, itemsep=1pt, parsep=0pt]
		\item[(a)]  there exists a Lagrangian fibration $\pi:X\to B$ onto a topological manifold $B$ such that the $\pi$-fibers are graded Lagrangians with respect to a holomorphic volume form $\Omega$;
		\item[(b)]  there exists a Berkovich analytic space $\mathscr Y$ over the Novikov field $\Lambda=\mathbb C((T^{\mathbb R}))$ together with a tropically continuous fibration $f : \mathscr Y \to B$ onto the same base $B$;
	\end{itemize}
 
	satisfying the following properties
 
	\begin{itemize}[topsep=1pt, itemsep=0pt, parsep=2pt]
		\item[(i)]  $\pi$ and $f$ have the same singular locus skeleton $\Delta$ in $B$;
		\item[(ii)]  $\pi_0=\pi|_{B_0}$ and $f_0= f|_{B_0}$ induce the same integral affine structures on $B_0=B\setminus \Delta$;
		\item[(iii)] $f_0$ is isomorphic to the canonical dual affinoid torus fibration associated to $\pi_0$.
	\end{itemize}
\end{conjecture}

\begin{defn}
	\label{SYZ_mirror_defn}
If $(X,\pi)$ and $(\mathcal Y, f)$ satisfy Conjecture \ref{conjecture_our_SYZ}, then we declare $\mathscr Y$ as \textit{SYZ mirror} to $X$, or more precisely $(\mathscr Y, f)$ as \textit{SYZ mirror} to $(X,\pi)$. If $\mathscr Y$ embeds in the Berkovich analytification of an algebraic variety $Y$ with equal dimensions, then we also say $Y$ as \textit{SYZ mirror} to $X$.
\end{defn}

The conjecture is accurately stated and supported by multiple examples in \cite{Yuan_I_FamilyFloer,Yuan_conifold}.
While item (iii) may require specialized Floer-theoretic machinery, the two other items (i) (ii) are already non-trivial to establish, despite relying on traditional concepts known for over a decade.

%

\begin{thm}
	\label{Main_thm_oversimplified}
	$A_n$-resolution is SYZ mirror to $A_n$-smoothing.
\end{thm}

The above main theorem is further explained as follows.
Let $\Bbbk$ be an algebraically closed field. By $A_n$-singularity, we refer to the singular variety 
$
Z=Z(\Bbbk)=\operatorname{Spec} \Bbbk[u,v,z]/(uv-z^{n+1}) = \Bbbk^2/\mathbb Z_{n+1}
$.

On the A-side, we study K\"ahler and symplectic geometry over $\Bbbk=\mathbb C$. 
Take a monic polynomial $h(z)=(z-a_0)\cdots (z-a_n)$ with no multiple zeros and $h(0)\neq 0$.
Note that $\mathcal P = \{a_0, \dots, a_n\}$ represents a parameter point in the configuration space $\mathscr C=\mathrm{Conf}_{n+1}(\mathbb C)$.
Define $\bar{X} \subset \mathbb{C}^3$ by $uv = h(z)$, and let $X$ be the complement of $z=0$. We call this an \textit{$A_n$-smoothing} of $Z$.
With the standard form $\omega$ from $\mathbb C^3$, we define a special Lagrangian fibration 
$\pi:X \to \mathbb{R}\times \mathbb{R}_{>0}$ by $(u,v,z) = (\frac{1}{2}(|u|^2-|v|^2),|z|)$ following Goldstein \cite{goldstein2001calibrated} and Gross \cite{Gross_ex}.
We can view the projection $p$ to the $z$-plane as a Lefschetz fibration on $\bar X$ with $n + 1$ critical values located at $\mathcal P$.
Beware that the singular locus $\Delta=\Delta_{\mathcal P}$ depends on the choice of $\mathcal P$; for a generic choice, the $\Delta$ consists of $n+1$ \textit{focus-focus singular points}.

On the B-side, we study algebraic and Berkovich geometry over the Novikov field $\Bbbk= \Lambda = \mathbb C((T^{\mathbb R}))$.  
The minimal resolution of $Z$, referred to as the \textit{$A_n$-resolution}, is a non-affine toric surface $Y_\Sigma$ associated with the fan $\Sigma$ generated by the $n+2$ rays $(0,1),(1,1),\dots, (n+1,1)$ in $\mathbb Z^2$ (Figure \ref{figure_fan}).   
Let us define $Y=Y_\Sigma^*$ as the complement of the divisor $t=1$ for the evident toric morphism $1+y\equiv t:Y_\Sigma\to\Bbbk$.

Now, the aim of Theorem \ref{Main_thm_oversimplified} is to prove that $Y$ is SYZ mirror to $X$ in the sense of Definition \ref{SYZ_mirror_defn}. 
For a broader audience, let's directly present the solution to Theorem \ref{Main_thm_oversimplified}, without diving into the details of its verification for the moment.

\begin{figure}
	\captionsetup{font=footnotesize}
	\centering 
\begin{subfigure}{0.4\textwidth}
	\centering
	\includegraphics[width=7cm]{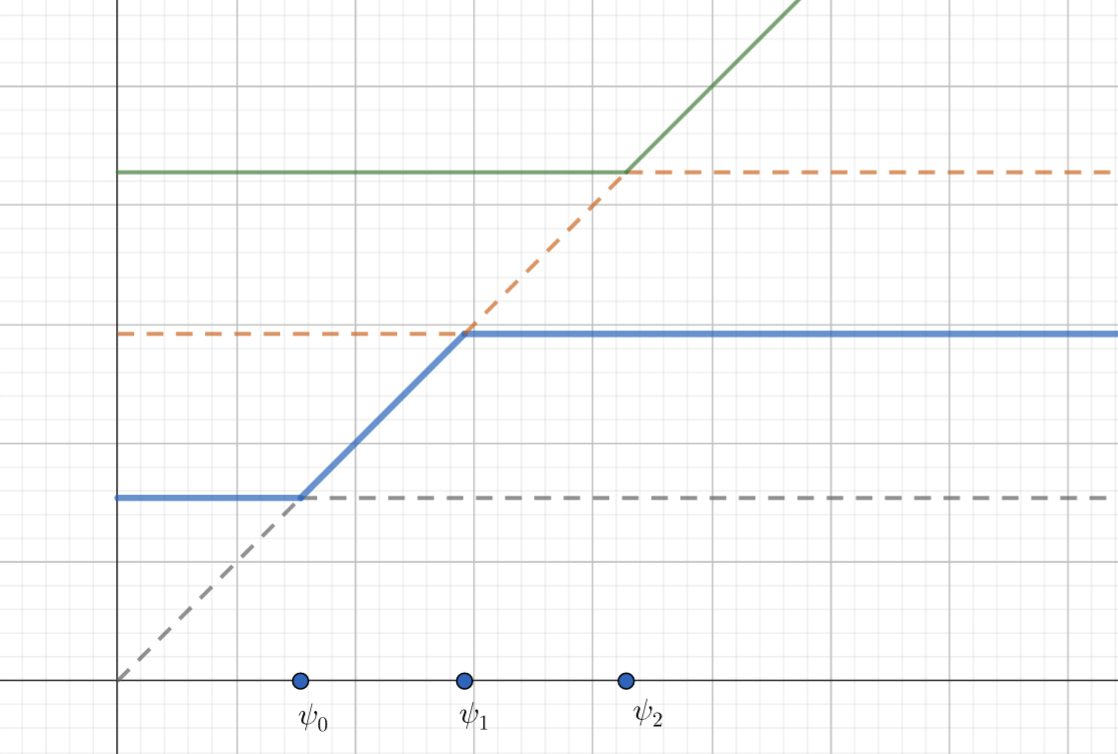}
	\caption{\tiny 
		}
	\label{figure_order_statistic}
\end{subfigure}
\hfill
\begin{subfigure}{0.44\textwidth}
	\centering
	\includegraphics[width=8cm]{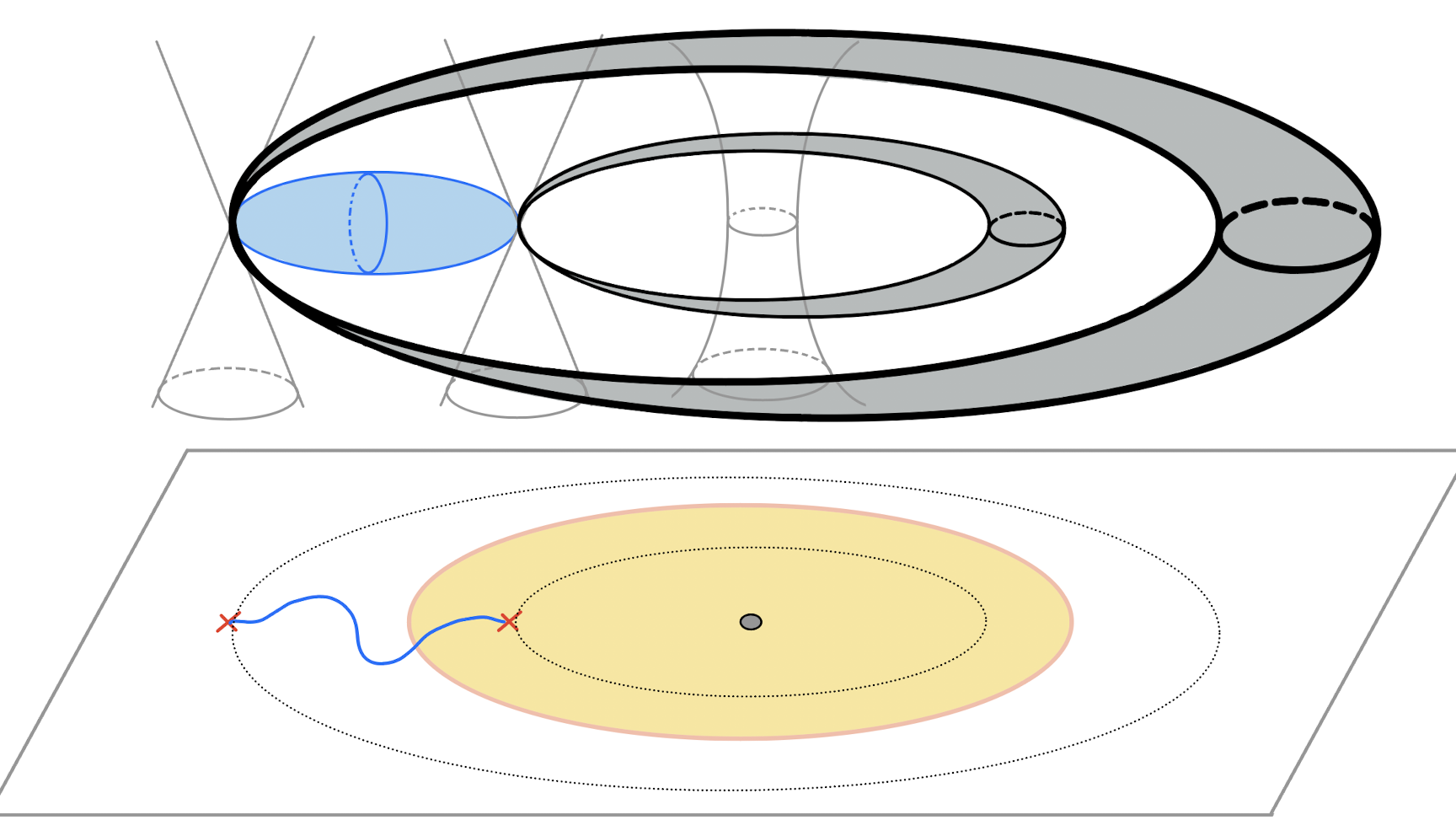}
	\caption{\tiny   }
	\label{figure_psi(s,r)}
\end{subfigure}
\caption{\scriptsize \textbf{(a)}: The 1st, 2nd, 3rd, and 4th \textit{order statistics} of the sample $\{x, \psi_0, \psi_1, \psi_2\}$ in the real variable $x$, ordered from bottom to top.
\textbf{(b)}: The area of the orange disk, $\psi(s,r)$, with radius $r$ for the reduced K\"ahler form at phase $s$, is smooth everywhere except at $s=0$ and $r=|a_j|$.
The two black pinched spheres are singular $\pi$-fibers over $(0,|a_{j-1}|)$ and $(0,|a_j|)$ in $B$. The blue one is a Lagrangian sphere.
\textit{When we apply the Dehn twist along the Lagrangian sphere in (b), we switch $\psi_1$ and $\psi_2$ in (a).}
}
\end{figure}

\begin{proof}[Proof of Theorem \ref{Main_thm_oversimplified} in a nutshell]

The first component of the Lagrangian fibration $\pi$ serves as the moment map for the $S^1$-action $(u,v,z)\mapsto (e^{it}u,e^{-it}v,z)$. The reduced space $\bar{X}_{red,s}=\mu^{-1}(s)/S^1$ of $\bar{X}$ can be identified with the complex plane $(\mathbb{C}_z,\omega_{red,s})$, equipped with the reduced K\"ahler form, via the projection map $p$.
Denote the $\omega_{{red}, s}$-symplectic area of the disk in $\mathbb C_z$ centered at the origin with radius $r$ by $\psi(s, r)$.
A subtle point is that $\psi=\psi_{\mathcal P}$ relies on $\mathcal P=\{a_k\}$ and is non-smooth at $(0,|a_k|)$ for all $k$. Note also that $\psi$ increases with $r$.
Moreover, away from the non-smooth points, the assignment $(s,r)\mapsto (s, \psi(s, r))$ gives a local coordinate chart for the integral affine structure on $B_0$.
Using homogeneous coordinates $[x_0:\cdots :x_{n+1}]$ in $Y_\Sigma$, 
we define a tropically continuous fibration 
$F=(F_0,F_1,\dots, F_{n+1}, \val(y)):Y^{\mathrm{an}} \to \mathbb R^{n+3}$
by setting
\[
F_k=\Big\{
\sum_{j=0}^{n+1} \  (j-k) \cdot \val(x_j)  +k\min\{0,\val(y)\}, 
\quad   \psi(\val(y),|a_0|) , \ \ \psi(\val(y),|a_1|) , \ \ \cdots  \  \ \psi(\val(y),|a_n|)
\Big\}_{[k]}
\]
where $0\leqslant k\leqslant n+1$ and $\val(\bullet)$ represents the non-archimedean valuation. The notation $\{\cdots\}_{[k]}$ denotes the $(k+1)$-th \textit{order statistic}, referring to the $(k+1)$-th smallest value of a sample of $n+2$ real numbers. It generalizes the $\min$ / $\max$ functions. See Figure \ref{figure_order_statistic}, Figure \ref{figure_order_statistics_degenerate}.

Define a topological embedding $j=(j_0,j_1,\dots, j_{n+1}, s): B=\mathbb R\times \mathbb R_{>0}\to\mathbb R^{n+3}$ by setting 
\[
j_k(s,r)=\Big\{
\psi(s,r), \psi(s, |a_0|), \psi(s,|a_1|), \dots, \psi(s,|a_n|) 
\Big\}_{[k]}
\]
where $0\leqslant k\leqslant n+1$. Define an analytic open domain $\mathscr Y=\{ |\prod_{j=0}^{n+1} x_j^j | <1 \}$ for the non-archimedean norm on the Novikov field $\Lambda$. Then, we can verify that $j(B)=F(\mathscr Y)$.
Now, we can define
\begin{equation}
	\label{f_eq_intro}
	f=f_{\mathcal P}=j^{-1}\circ F :  \mathscr Y \to B
\end{equation}
No matter the choice of $\mathcal P=\{a_0,\dots, a_n\}$, we can always confirm that the singular locus $\Delta=\Delta_{\mathcal P}$ of $f$ aligns exactly with that of $\pi$; we can always examine that $f_0:=f|_{B_0}$ induces an integral affine structure that is precisely identical to that of $\pi$. (The full details will be given in the body of this article.)
\end{proof}

The Lagrangian fibration $\pi=\pi_{\mathcal P}$, its singular locus $\Delta=\Delta_{\mathcal P}$, and its induced integral affine structure {\textit{all}} depend on the choice of $\mathcal P=\{a_0,\dots, a_n\}$ (cf. Figure \ref{figure_collision_singular_fiber}). For instance, if the norms $|a_k|$'s are pairwise distinct, the locus $\Delta$ has $n+1$ discrete focus-focus singularities. If all the norms are equal, the locus $\Delta$ just has a single singular point. Remarkably, no matter the choice of $\mathcal P=\{a_0,\dots, a_n\}$ in the configuration space $\mathscr C$, our formula (\ref{f_eq_intro}) for $f=f_{\mathcal P}$ \textit{{always}} delivers the desired solution, namely, satisfying all the conditions (i) (ii) (iii) in Conjecture \ref{conjecture_our_SYZ}.
In particular, as $\mathcal P$ moves within the configuration space, we may in principle witness the dynamic process of $f=f_{\mathcal P}$ that links the braid group action to the collision-and-scattering behavior of the singular locus $\Delta=\Delta_{\mathcal P}$ (see \S \ref{s_miscellaneous_topics}).


Here we adopt the strategy of Kontsevich and Soibelman \cite[\S 8]{KSAffine}: instead of directly finding $f$, we search for some Berkovich-continuous map $F:Y\to \mathbb R^N$ for a large integer $N$ such that the image of $F$ can be identified with $B$ via some map $j$.
The order statistic functions are designed to locally imitate the min/max functions and globally capture the integral affine structures (see Figure \ref{figure_order_statistic}).
Note that it is generally hard to find $f:Y\to B$ with the properties (i) and (ii) in Definition \ref{SYZ_mirror_defn}.
This question even totally makes sense without any Floer theory. But intriguingly, we explicitly identify such a parameter-dependent solution through a Floer-theoretic approach regarding item (iii).
Hence, our methodology should promise to be of value, and actually, (ii) is only a necessary but insufficient condition of the more fundamental (iii).

Note that matching integral affine structures goes beyond merely asserting consistent monodromy around each singularity. It needs precisely aligning two \textit{atlases} of integral affine coordinate systems on both base manifolds.
For instance, the \textit{singular} integral affine structure on $B=\mathbb R\times \mathbb R_{>0}$ subtly relies on $\mathcal P=\{a_0,\dots, a_n\}$ and the given symplectic form.
Accordingly, even with the discovery of a correct formula as (\ref{f_eq_intro}), we must confess that the subsequent verification of integral affine structures and singular loci as claimed below (\ref{f_eq_intro}) can be unfortunately a quite patience-demanding task.
Therefore, we also present more readily verifiable evidence, ensuring that a moderate investment of time suffices to witness some interesting phenomenon (see Observation \ref{observation_A} and Figure \ref{figure_observation}).
Concurrently, the preciseness of this geometric phenomenon should more or less serve as a harbinger for more profound investigations into mirror symmetry, unifying an array of mathematical domains, including Lagrangian Floer theory, homotopy theory for $A_\infty$ structures, non-archimedean analysis, and Berkovich geometry.

%



\vspace{1.5em}
\textbf{Acknowledgment . }
The author thanks the hospitality of the Simons Center for Geometry and Physics during a visit at Stony Brook in April-May 2023, where part of this work was completed. The author is also grateful to the organizers of Concluding Conference of the Simons Collaboration on Homological Mirror Symmetry, as well as Mohammed Abouzaid, Denis Auroux, Kenji Fukaya, Ludmil Katzarkov, Tony Pantev, Daniel Pomerleano, and Paul Seidel for valuable in-person conversations in the conference. The author also thanks Kwok Wai Chan, Paul Hacking, Mingyuan Hu, Yusuke Kawamoto, Siu-Cheong Lau, Wenyuan Li, Yu-Shen Lin, Mark McLean, Chenyang Xu, Tony Yue Yu, Eric Zaslow, and Shizhuo Zhang for useful discussions at various stages of this work.

\section{Topological wall-crossing and atlas of integral affine structure}

\subsection{Lagrangian fibration}
\label{ss_Lag_fib}
	Consider
$ \bar X =\{(u,v,z)\in \mathbb C^3 \mid uv=h(z)\}$
	where $h(z)=\prod_{k=0}^n (z-a_k)$.
Initially, we assume the norms $|a_k|$'s are pairwise distinct such that
$0<|a_0|<|a_1|<\dots <|a_n|<\infty$.
For the divisor $\mathscr D=\{z=0\}$ in $\bar X$, we define 
\[
X=\bar X\setminus \mathscr D \ .
\]

Introduce the following divisors:
	\begin{equation}
		\label{divisor_Duv_k_eq}
		\begin{aligned}
			\textstyle 	D_u^k=\{ u=0 , v\in\mathbb C,  z=a_k \}  \qquad D_u= \bigcup_{k=0}^n D_u^k \\
			\textstyle D_v^k = \{ u\in\mathbb C, v=0, z=a_k\} 	\qquad D_v=\bigcup_{k=0}^n D_v^k 
		\end{aligned}
	\end{equation}

We equip $X$ and $\bar X$ with the K\"ahler form obtained by restricting the standard K\"ahler form on $\mathbb C^3$.
Note that a smooth affine variety is a Stein manifold \cite[Example 2.8]{mclean2009lefschetz}.
There is a natural Hamiltonian $S^1$-action on $ \bar X$ given by
	\begin{equation}
		\label{S1_action_eq}
		e^{it} \cdot (u,v, z) = (e^{it} u, e^{-it}v, z)
	\end{equation}
	The associated moment map $\mu: \bar X \to\mathbb R$ is given by
	$
	\mu(u,v,z)=\tfrac{1}{2} (|u|^2-|v|^2) 
	$.
	The fixed points of the $S^1$-action are given by the $n+1$ points $p_k = (0, 0, a_k)$ for $0\leqslant k\leqslant n$.
	
	For any $s\in \mathbb R$, let
	$
	\bar X_{red,s}:= \mu^{-1}(s)/S^1
	$
	be the reduced space associated to the moment map $\mu$. Then, $\bar X_{red,s}$ is smooth and diffeomorphic to $\mathbb C$ via the projection 
	$p:(u,v, z)\mapsto z$. One may regard $p$ as a Lefschetz fibration as illustrated in Figure \ref{figure_psi(s,r)}.
	The reduced symplectic form is denoted by $\omega_{red,s}$.
	We can show that
	\begin{equation}
		\label{pi_Lagrantian_fibration_eq}
		\pi:X \to \mathbb R \times \mathbb R_{> 0} ,\qquad (u,v,z)\mapsto (\tfrac{1}{2}(|u|^2-|v|^2), \ |z| )
	\end{equation}
	is a special Lagrangian fibration whose discriminant locus is
	\[
	\Delta=\{(0,|a_0|), (0,|a_1|), \dots, (0,|a_n|) \}
	\]
	consisting of $n+1$ singular points.
	Observe that the images of the divisors $D_u^k$ and $D_v^k$ in (\ref{divisor_Duv_k_eq}) under the fibration map $\pi$ are given by the dashed half lines as illustrated in Figure \ref{figure_milnor_fiber}:
	\begin{equation}
		\label{pi_of_divisors}
		\pi(D_u^k)= \mathbb R_{\leqslant 0}\times \{|a_k|\} \qquad \pi(D_v^k)=\mathbb R_{\geqslant 0} \times \{|a_k|\}
	\end{equation}

	Denote by $L_q:=\pi^{-1}(q)$ the Lagrangian fiber over $q=(s,r)$.
	Set $B=\mathbb R\times \mathbb R_{>0}$ and $B_0=B\setminus \Delta$, and we define
	\begin{equation}
		\label{pi_0_smooth_part_eq}
		\pi_0:=\pi|_{B_0}: X_0\equiv \pi^{-1}(B_0)\to B_0
	\end{equation}

\begin{figure}
	\centering
		\captionsetup{font=footnotesize}
	\begin{subfigure}{0.4\textwidth}
		\includegraphics[width=8cm]{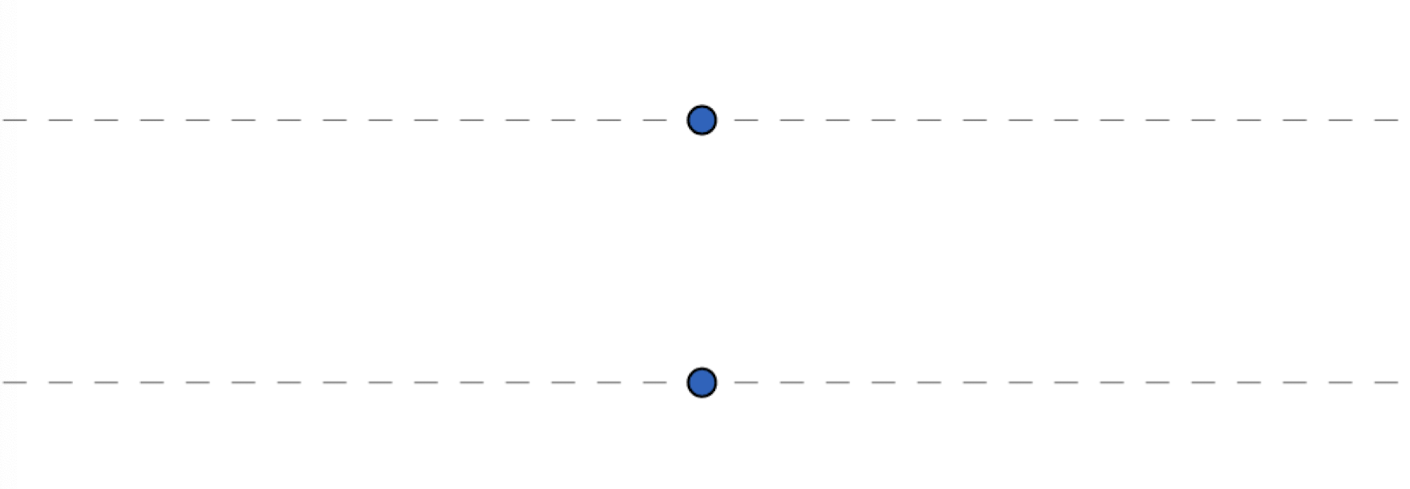}
		\caption{\scriptsize Base space near two singular points}
		\label{figure_milnor_fiber}
	\end{subfigure}
	\hfill
	\begin{subfigure}{0.44\textwidth}
		\includegraphics[width=8cm]{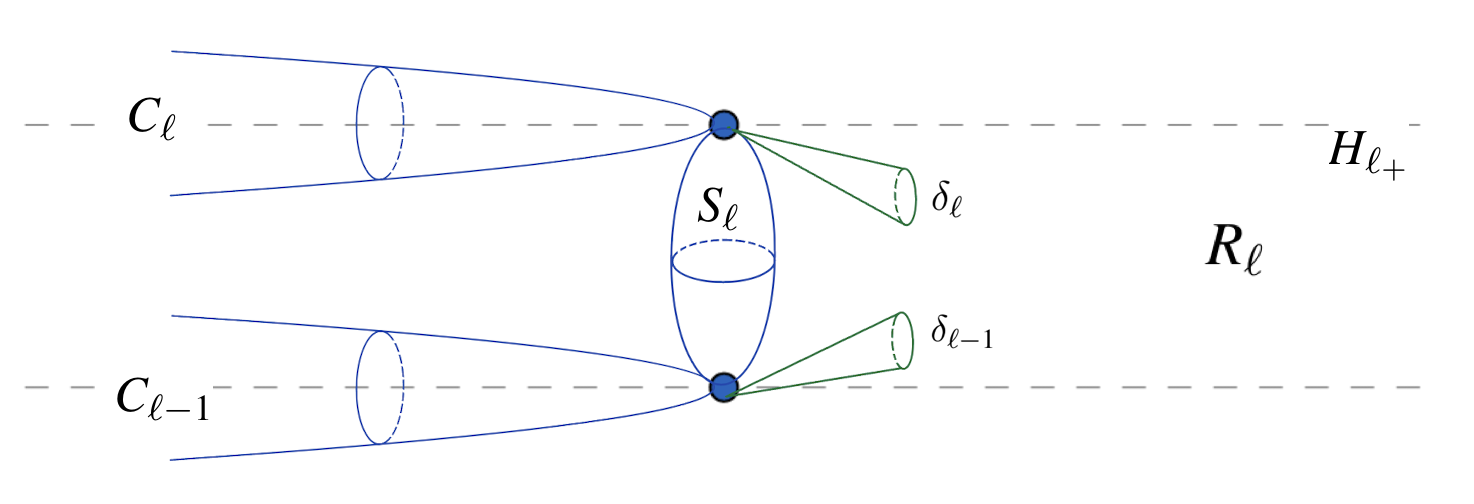}
		\caption{\scriptsize  Several related topological classes}
		\label{figure_milnor_fiber_sphere}
	\end{subfigure}
\caption{}
\end{figure}

\subsection{Compactification for $A_n$ smoothing}
Following J. D. Evans in \cite[Section 7]{evans2011symplectic}, we review and study the compactification of $\bar X$ as follows.
The blow-up $\mathcal V$ of $\mathbb {CP}^2$ at the $(n+1)$ points $[a_k :0 : 1]$ ($k=0,1,\dots n$) can be regarded as the subvariety of $M=\mathbb {CP}^2\times \prod_{k=0}^n \mathbb {CP}^1_k$ defined by the equations
\begin{equation}
	\label{V_eq}
	\zeta_k y=x-a_k w
\end{equation}
for $k=0,1,\dots n$
in the coordinates
$[x:y:w]\in\mathbb {CP}^2$ and $\zeta_k=[\zeta_k': \zeta_k''] \in\mathbb {CP}^1$
on $M$.

\begin{itemize}
	\item For $0\leqslant k\leqslant n$, we denote by $C_k$ the exceptional sphere of the blow-up; namely, $C_k$ is given by $[x:y:w]=[a_k:0:1]$, $\zeta_j=\infty$ for $j\neq k$, and arbitrary $\zeta_k\in\mathbb {CP}^1$.
	\item  The map on $M$ sending to $[x:w]$, restricted on $\mathcal V$, gives rise to a pencil of curves
	$
	P_t =\{  ([x:y:w], \zeta_k ) \in \mathcal V \mid  t=[x:w]\} 
	$
	parameterized by $t\in\mathbb {CP}^1$. We define
	$C_{n+1}=P_\infty$. It is the proper transform of the projective line $[x:y:0]$ in $\mathbb {CP}^2$. Namely, it is a sphere in $\mathcal V$ given by $w=0$ and $\zeta_k=[x:y]$ for all $k$.
	\item We define $C_{n+2}$ to be the sphere in $\mathcal V$ given by $y=0$, all $\zeta_k=\infty$, and arbitrary $[x:w] \in \mathbb {CP}^1$.
\end{itemize}

The following result is due to J. D. Evans \cite[Lemma 7.1]{evans2011symplectic}.

\begin{lem}
	\label{biholomorphic_lem}
	$\mathcal V\setminus (C_{n+1}\cup C_{n+2})$ is biholomorphic to $\bar X$, and $\mathcal V\setminus \bigcup_{k=0}^{n+2} C_k$ is biholomorphic to $\mathbb C^*\times \mathbb C$.
\end{lem}

Roughly speaking, we can define
\[
\Phi: \bar X \to \mathcal V \qquad (u,v, z) \mapsto \left([z:u: 1 ], \  [z-a_0: u], \dots,  [z-a_n : u] \right) \ .
\]
Here since $uv=h(z)=\prod_k (z-a_k)$ for $(u,v,z)\in \bar X$, we actually identify $[z-a_k : u]$ with $[v(z-a_k) : h(z)]= [v : \prod_{j\neq k} (z-a_j)]$ whenever $u=0$; in particular, since all the $a_k$'s are distinct, at least one of these coordinates is not $\infty=[1:0]$.
Thus, the map $\Phi$ misses $C_{n+2}$. It also misses $C_{n+1}$ by definition, and one can finally check $\Phi:\bar X\to \mathcal V\setminus (C_{n+1}\cup C_{n+2})$ is onto.
For $0\leqslant k\leqslant n$, the preimage $\Phi^{-1}(C_k)$ is exactly given by the divisor $D_u^k$ in (\ref{divisor_Duv_k_eq}).
Then, $\mathcal V\setminus \bigcup_{k=0}^{n+2} C_k$ can be identified via $\Phi$ with $\bar X\setminus D_u$, and the latter can be further identified with $\mathbb C^*_u \times \mathbb C_z$ via $v= u^{-1} h(z)$.
Besides, we observe that the $\Phi$-image of the divisor $D^k_v$ in (\ref{divisor_Duv_k_eq}) is given by $C'_k\setminus (C_{n+1}\cup C_{n+2})$ where $C'_k$ is a sphere given by $x=a_k w$, $\zeta_j=[ (a_k-a_j)w : y]$, and $\zeta_k=[0:1]$ for arbitrary $[y:w]\in\mathbb {CP}^1$.
Note that the second Betti number of $\mathcal V$ is $n+2$ (cf. \cite[Theorem 7.31]{voisin2002hodge}), and the second singular homology $H_2(\mathcal V; \mathbb Z)$ can be generated by $[C_0],[C_1],\dots, [C_{n+2}]$ under the constraint
$\textstyle 
[C_{n+2}]=[C_{n+1}]-\sum_{k=0}^n[C_k]$ \cite[Page 73-74]{evans2011symplectic}.

Let ``$\cdot$'' denote the intersection form
$H_2(\mathcal V;\mathbb Z)\times H_2(\mathcal V;\mathbb Z) \to \mathbb Z$; see e.g. \cite[\S 3.2]{scorpan2022wild}. 
Then, it is routine to check (cf. Figure \ref{figure_milnor_fiber_sphere})
\[
\begin{cases}
	C_j\cdot C_k=0   & 0\leqslant j, k\leqslant n, \ j\neq k \\
	C_k\cdot C_k = -1  ,\, \, \, \,
	C_k\cdot C_{n+1}=0 , \, \, \, \,
	C_k\cdot C_{n+2}=1 & 0\leqslant k\leqslant n \\
	C_{n+1}\cdot C_{n+2} =1 
\end{cases} \ .
\]
Besides,
\[
\begin{cases}
	C'_k \cdot C_j=0 &  0\leqslant j,k\leqslant n, \  j\neq k \\
	C'_k \cdot C_k=1, \,\,\,\, C'_k\cdot C_{n+1}=1 , \,\,\,\, C'_k\cdot C_{n+2}=0 & 0\leqslant k\leqslant n 
\end{cases}
\]
where we note that $C'_j\cap C_k=\varnothing$ whenever $j\neq k$ and that $C'_k\cap C_k$ is the single point whose $\Phi$-preimage is exactly the $k$-th fixed point $p_k=(0,0,a_k)$ of the $S^1$-action (\ref{S1_action_eq}).
It follows that $[C'_k]=[C_{n+1}]-[C_k]$ in $H_2(\mathcal V;\mathbb Z)$.
Next, we aim to find $H_2(\bar X ; \mathbb Z)\cong H_2(\mathcal V\setminus (C_{n+1}\cup C_{n+2}); \mathbb Z)$. 
We introduce $S_\ell = C_{\ell-1} - C_\ell $ in $\mathcal V$ for $1\leqslant \ell \leqslant n$ (cf. Figure \ref{figure_milnor_fiber_sphere}). Then, $S_\ell \cdot C_{n+1}=S_\ell \cdot C_{n+2}=0$.
By standard algebraic topology, one can check that (see also \cite[p215]{Seidel2006biased})
\begin{equation}
	\label{H2bar X_eq}
H_2(\bar X;\mathbb Z)\cong \pi_2(\bar X) \cong  \mathbb Z\{ S_1,\dots, S_n\}
\end{equation}

From now on, we will always take the identification from the biholomorphic map $\Phi: \bar X\cong \mathcal V\setminus (C_{n+1}\cup C_{n+2})$.
Recall $\Phi^{-1}(C_k)=D_u^k$ and $\Phi^{-1}(C'_k)=D_v^k$ for $0\leqslant k\leqslant n$ and the divisors in (\ref{divisor_Duv_k_eq}).
Then, we observe that for $1\leqslant \ell\leqslant n$,
\begin{equation}
	\label{S_ell_intersection_number}
	S_\ell \cdot  D_u^k =
	\begin{cases}
		1 & \text{if} \  k=\ell \\
		-1 & \text{if} \ k=\ell-1 \\
		0 & \text{if} \ k\neq \ell-1, \ell
	\end{cases}
	\qquad
	\text{and}
	\qquad
		S_\ell \cdot  D_v^k =
	\begin{cases}
		-1 & \text{if} \  k=\ell \\
		1 & \text{if} \ k=\ell-1 \\
		0 & \text{if} \ k\neq \ell-1, \ell
	\end{cases}
\end{equation}

\subsection{Topological local systems}
We consider the following local systems over $B_0$:
	\begin{equation}
		\label{local_system_H_2_H_1_eq}
		\mathscr R_1:=R^1\pi_*(\mathbb Z)\equiv \bigcup_{q\in B_0} \pi_1(L_q), \qquad  \mathscr R_2 :=\mathscr R_2(\bar X):=\bigcup_{q\in B_0} \pi_2( \bar X,L_q)
	\end{equation}
	Abusing the notations, the fibers $\pi_1(L_q)$ and $\pi_2(\bar X,L_q)$ of $\mathscr R_1$ and $\mathscr R_2$ actually denote the corresponding images of the Hurewicz maps in the (relative) \textit{homology} groups $H_1(L_q)$ and $H_2(\bar X,L_q)$ respectively rather than the homotopy groups. These notations attempt to avoid using $H_1(L_q)$ and $H^1(L_q)$ in the same time, and we apologize for the possibly ambiguous notations.
	By (\ref{H2bar X_eq}), there is a natural action of $\pi_2(X)\equiv H_2(X;\mathbb Z)$ on $\mathscr R_2$ as well.
	Note that $\mathscr R_2$ relies on the choice of the partial compactification space $\bar X$ of $X$.
	
	\begin{rmk}
		We must understand the monodromy behavior of the above local systems across multiple walls, while most of existing literature only study the disks in $\mathscr R_2$ across a single wall.
		Although the latter is sufficient for mirror \textit{space} identifications (cf. \cite{Chan_An_Tduality}), we must understand the former monodromy data for mirror \textit{fibration} realizations.
		This justifies why we need to provide many additional details later on.
	\end{rmk}
	
	The local systems $\mathscr R_1$ and $\mathscr R_2$ can be determined by their trivializations over a covering of $B_0$ by contractible open subsets. There are of course many different such coverings of $B_0$.
	Notice that the orbits of the $S^1$-action (\ref{S1_action_eq}) naturally give rise to a global section of $\mathscr R_1$, denoted by
\begin{equation}
	\label{sigma_R_1}
\sigma \in \Gamma(B_0;\mathscr R_1) .
\end{equation}

For $0\leqslant k\leqslant n$, we define
\[
H_{k+}=\mathbb R_{>0}\times  \{|a_k|\}  \qquad H_{k-}=\mathbb R_{<0}\times  \{|a_k|\}
\]
and define the open regions
	\begin{align*}
		R_0 &= \mathbb R\times (  0 ,  \  |a_0|)     \\
		R_\ell &= \mathbb R\times (|a_{\ell-1}|, \ |a_\ell|)  \qquad \quad (1\leqslant \ell\leqslant n)  \\
		R_{n+1} &= \mathbb R\times (|a_n|,  \infty ) 
	\end{align*}
in $B_0$.
By slightly thickening these open regions, we find a covering $\{U_0,\dots, U_{n+1}\}$ of $B_0$ by $n+2$ contractible open subsets where
\begin{equation}
	\label{U_k_eq}
	\begin{aligned}
	U_0 &=R_0 \cup \mathscr N_{0+} \cup \mathscr N_{0-} \\
	U_\ell &= R_\ell  \cup \mathscr N_{(\ell-1)-}\cup \mathscr N_{(\ell-1)+} \cup \mathscr N_{\ell-} \cup \mathscr N_{\ell +} \qquad (1\leqslant \ell \leqslant n) \\
	U_{n+1} &= R_{n+1}  \cup \mathscr N_{n+}\cup\mathscr N_{n-}  
\end{aligned}
\end{equation}
where we let $\mathscr N_{k\pm}$ denote a sufficiently small neighborhood of $H_{k\pm}$ in $B_0$.

\subsubsection{Preferred disks I}
The above covering of $B_0$ is convenient to study the wall-crossing phenomenon.
But, it is also useful to introduce the following covering $\{ \mathcal B_{k\pm}:  0\leqslant k\leqslant n\}$ of $B_0$:
\begin{align*}
	\begin{aligned}
		\mathcal B_{k+} = R_k\sqcup R_{k+1} \sqcup H_{k+} \\
		\mathcal B_{k-} =  R_k\sqcup R_{k+1} \sqcup H_{k-} 
	\end{aligned}
	\qquad (0\leqslant k\leqslant n)
\end{align*}
For each $0\leqslant k\leqslant n$, we further set
\[
\mathcal B_k=\mathcal B_{k+}\cup \mathcal B_{k-} \ .
\]
Although it is not contractible, there is a natural section (cf. Figure \ref{figure_milnor_fiber_sphere})
\begin{equation}
	\label{delta_k_eq}
\delta_k\in \Gamma(\mathcal B_k; \mathscr R_2)
\end{equation}
of $\mathscr R_2$ over $\mathcal B_k$ such that $\partial\delta_k$ coincides with the section $\sigma$ of $S^1$-orbits in (\ref{sigma_R_1}).
In reality, for any $q\in \mathcal B_k$, we take a path $q(t)$ from $q$ to the singular point $(0, |a_k|)$ that avoids $H_{(k-1)\pm}$ and $H_{(k+1)\pm}$. Note that $\sigma$ degenerates at the singular point, then we define $\delta_k(q)$ to be the disk obtained by the union of $\sigma(q(t))$.

\begin{rmk}
	\label{Maslov_0_explicit_rmk}
	When $q=(s, |a_k|) \in H_{k+}$ for $s>0$, the class $\delta_k(q)\in \pi_2(\bar X,L_q)$ can be represented by an explicit holomorphic disk $\zeta\mapsto (\zeta \sqrt{2s}, 0 , |a_k| )$ for $\zeta\in\mathbb D$. Similarly, when $q=(s,|a_k|)\in H_{k-}$ for $s<0$, the class $\delta_k(q)$ can be represented by an anti-holomorphic disk $\zeta\mapsto (0, \bar \zeta \sqrt{2s}, |a_k|)$.
\end{rmk}

\begin{prop}
	\label{wall_place_prop}
	The Lagrangian torus fiber $L_q$ over a smooth point $q=(s,r)\in B_0$ bounds a nonconstant Maslov-0 holomorphic disk in $X=\bar X\setminus \mathscr  D$ if and only if $q$ is contained in $H_{k+}\cup H_{k-}$ for some $0\leqslant k\leqslant n$. Therefore, we call $H_{k\pm}$ the \textbf{\textit{walls}}.
\end{prop}

\begin{proof}[Sketch of proof]
	This is standard (see \cite{AuTDual}, \cite[Proposition 3.1]{Chan_An_Tduality}, and \cite[Lemma 5.2]{CLL12}). The `if' part has been justified by the above explicit constructions in Remark \ref{Maslov_0_explicit_rmk}. For the `only if' part, suppose $\varphi:(\mathbb D,\partial\mathbb D)\to (X,L_q)$ is such a nontrivial holomorphic disk. It avoids the divisor $\mathscr D=\{z=0\}$ by \cite[Lemma 3.1]{AuTDual}. We then have a holomorphic map $z\circ \varphi:\mathbb D \to \mathbb C^*$ with $|z\circ \varphi|_{\partial\mathbb D}|$ constant. Thus, the maximal principal implies that $z\circ \varphi$ is constant. It follows that $uv\circ \varphi$ is constant, and $\varphi$ can be nontrivial only if $uv\circ \varphi=0$. Hence, $z\circ \varphi$ must be one of the roots $a_k$'s of $h(z)$.
\end{proof}

For $0<\ell\leqslant n$, we aim to study how the $\delta_\ell$ differs from $\delta_{\ell-1}$, as sections of the local system $\mathscr R_2$, on the overlap open subset 
$\mathcal B_\ell\cap \mathcal B_{\ell-1} \equiv R_\ell$.

Since we only consider topological classes, we may slightly thicken it to be $\tilde R_\ell=R_\ell \cup \mathscr N_{(\ell-1)+} \cup \mathscr N_{\ell+}$ adding the small neighborhoods $\mathscr N_{(\ell-1)+}$ and $\mathscr N_{\ell+}$ of $H_{(\ell-1)+}$ and $H_{\ell+}$.
The intersection numbers $\delta_i\cdot D_v^j$ for $i,j=\ell-1,\ell$ are \textit{not} well-defined, while $\delta_i\cdot D_u^j$ are still well-defined (cf. \cite{CLL12}). In fact, it is routine to check that
$\delta_{\ell-1}\cdot D_u^{\ell-1}=\delta_\ell \cdot D_u^\ell=1$, $\delta_{\ell-1}\cdot D_u^j=0$ for $j\neq \ell-1$, and $\delta_\ell\cdot D_u^j=0$ for $j\neq \ell$.
Now, let $q$ be an arbitrary point in $R_\ell$. Since both $\partial\delta_\ell(q)$ and $\partial\delta_{\ell-1}(q)$ agree with the $\sigma(q)$, it follows that
$\delta_\ell-\delta_{\ell-1}$ can be viewed as a class in $\pi_2(X)$. And, the above discussion of the intersection numbers infers that $(\delta_\ell-\delta_{\ell-1})\cdot D_u^\ell=1$, $(\delta_\ell-\delta_{\ell-1})\cdot D_u^{\ell-1} =-1$, and $(\delta_\ell-\delta_{\ell-1}) \cdot D_u^j=0$ for $j\neq \ell-1,\ell$.
According to (\ref{S_ell_intersection_number}), we finally conclude that (cf. Figure \ref{figure_milnor_fiber_sphere})
\begin{equation}
	\label{delta_ell_ell-1_relation}
	\delta_\ell-\delta_{\ell-1}=S_\ell  \qquad \text{over} \ R_\ell
\end{equation}
for $1\leqslant \ell\leqslant n$.

Observe that all these $\delta_k$'s avoid the divisor $\mathscr D$, the $\partial\delta_k$'s coincide with the $S^1$-orbits, and $\omega$ is exact on $X=\bar X\setminus \mathscr  D$. Thus, an explicit calculation deduces that the moment map $\mu$ can be identified with the symplectic areas of these $\delta_k$'s (up to the scalar $\frac{1}{2\pi}$). 
Moreover, the symplectic area of the class $S_\ell$ is actually zero. In fact, the topological class $S_\ell$ can be represented by a Lagrangian sphere (see e.g. \cite[\S 7.3]{evans_2021_book}, \cite[\S 7]{Seidel_Maydanskiy_2010lefschetz}, \cite[\S 6c]{Seidel_Khovanov2002quivers}). In other words, we have
\begin{equation}
	\label{moment_map_restrict_V_ell_eq}
	\mu|_{\mathcal B_k}= \frac{1}{2\pi} \int_{\delta_k} \omega =E(\delta_k)
\end{equation}
for $0\leqslant k\leqslant n$. In practice, we just keep in mind that whenever $q=(s,r)\in \mathcal B_k$, we have $s=E(\delta_k(q))$.

\subsubsection{Preferred disks II}
\label{sss_preferred_disks_beta}

For each $0\leqslant k\leqslant n$, we fix a representative point $q_k=(0,r_k)$ in $U_k$ where
\begin{equation}
	\label{r_k_eq}
0<r_0<|a_0|<\cdots < |a_{n-1}|<r_{n}<|a_n|<r_{n+1} <\infty
\end{equation}
From now on, we always let $\zeta$ denote a complex variable in the closed unit disk $\mathbb D=\{z\in\mathbb C\mid |z|\leqslant 1\}$.

	\begin{itemize}
		\item Let $g_0(\zeta)$ be a square root of the nonvanishing holomorphic function 
		\[
		\zeta\mapsto \prod_{k=0}^n (r_0 \zeta-a_k) \equiv h(r_0\zeta) \  .
		\] 
		Then, 
		\[
		\zeta\mapsto (g_0(\zeta),  g_0(\zeta), r_0\zeta) 
		\]
		is a holomorphic disk in $\bar X$ bounded by $L_{q_0}$. Its topological class gives a section of $\mathscr R_2$ over the contractible open subset $U_0$, denoted by
		\[
		 \beta_0 =\beta_{0,\varnothing} \in \Gamma( U_0; \mathscr R_2)  \  .
		\]
		
		\item For $0<\ell \leqslant n+1$, let $g_\ell(\zeta)$ be a square root of the nonvanishing holomorphic function 
		\[
		\zeta\mapsto  
		\prod_{i=0}^{\ell-1} (r_\ell - \bar a_i \zeta) \cdot \prod_{i=\ell}^n (r_\ell \zeta - a_i)
		\equiv h(r_\ell\zeta) \cdot \left( \prod_{i=0}^{\ell-1} \frac{r_\ell \zeta - a_i}{r_\ell -\bar a_i \zeta} \right)^{-1}  \  .
		\]
		For any 
		\[
		I \subset \{0,1,\dots, \ell-1\} =:[\ell] \ ,
		\]
		we introduce the following holomorphic disk
		\[
		\zeta\mapsto \left(
		g_\ell(\zeta) \prod_{i\in I} \frac{r_\ell \zeta - a_i}{r_\ell -\bar a_i \zeta} , \qquad 
		g_\ell(\zeta) \prod_{i\notin I} \frac{r_\ell \zeta - a_i}{r_\ell -\bar a_i \zeta}, \qquad r_\ell \zeta 
		\right) 
		\]
		in $\bar X$ bounded by $L_{q_\ell}$
		Its topological class defines a section of $\mathscr R_2$ over the contractible open subset $U_\ell$, denoted by
		\[
		\beta_{\ell, I} \in \Gamma(U_\ell ; \mathscr R_2)  \  .
		\]
		It can be essentially characterized by the intersection numbers as follows:
		\begin{equation}
			\label{beta_I_Duv_eq}
		\beta_{\ell, I}\cdot D_u^k = 
		\begin{cases}
			1 & \text{if} \ k\in I \\
			0 & \text{if} \ k\in [\ell]\setminus  I  \\
			0 & \text{if} \ \ell \leqslant k\leqslant n 
		\end{cases}
\qquad \text{and} \qquad
		\beta_{\ell, I}\cdot D_v^k = 
		\begin{cases}
			0 & \text{if} \ k\in I \\
			1 & \text{if} \ k\in [\ell]\setminus  I  \\
			0 & \text{if} \ \ell \leqslant k\leqslant n 
		\end{cases}
		\end{equation}

	\end{itemize}

By construction, for any $0\leqslant \ell\leqslant n$ and $I$, we also have
\begin{equation}
	\label{beta_I_anti_divisor_eq}
	\beta_{\ell, I}\cdot \mathscr D=1
\end{equation}

Clearly, these sections $\beta_{\ell,I}$'s are linear dependent in $\Gamma(U_\ell; \mathscr R_2)$. We can further find the linear relations among them as follows.
We set 
\[\beta_\ell=\beta_{\ell,\varnothing}  \  .
\]
By construction, $\partial\beta_{\ell, I}-\partial\beta_\ell$ is given by $|I|\cdot \sigma$ topologically.
Since the kernel of $\partial$ is $\pi_2(\bar X)$, it follows from (\ref{H2bar X_eq}) that 
\[
\beta_{\ell, I} =\beta_\ell + |I| \delta_{\ell} + \sum_{j=1}^n m_j\cdot S_j
\]
over $R_\ell$ for some integers $m_j\in\mathbb Z$ ($1\leqslant j\leqslant n$).
Note that $\delta_\ell|_{R_\ell} \cdot D_u^\ell=1$ and $\delta_\ell|_{R_\ell}\cdot D_u^j=0$ for other $j\neq \ell$.
Note also that the same topological classes must give the same intersection numbers.
Taking the various intersection numbers with $D_u^k$ for $0\leqslant k\leqslant n$ on the both sides, we can use the relations (\ref{beta_I_Duv_eq}) and (\ref{S_ell_intersection_number}) to find the following relations:
\[
\begin{cases}
	\one_I(0) = -m_1 \\
	\one_I(k) = -m_{k+1}+m_k  & \text{if } 0<k<\ell \\
	-|I|=-m_{\ell+1} +m_\ell  \\
	0= -m_{k+1}+m_k & \text{if }  k> \ell 
\end{cases}
\]
where $\one_I(j)$ is the characteristic function with regard to the set $I$ in the sense that it takes $1$ when $j\in I$ and $0$ when $j\notin I$.
We also remark that one can derive the same result by taking the intersection numbers with $D_v^k$'s instead.

Finally, after some tedious but routine computations, we obtain the following concise formula
\begin{equation}
	\label{beta_I_back_to_beta}
	\beta_{\ell, I} =\beta_\ell + |I| \delta_\ell  -\sum_{j=1}^\ell |I\cap [j]| S_j  \qquad \text{over } R_\ell
\end{equation}
where we write $[j]=\{0,1,\dots, j-1\}$.

\subsubsection{Topological wall-crossing}
We aim to study the relations between the two collections $\{\beta_{\ell, I}\}_{I\subset [\ell] }$ and $\{\beta_{\ell+1, I} \}_{I \subset [\ell+1] }$ restricted over $\mathscr N_{\ell+}$ or $\mathscr N_{\ell-}$.
Recall that the intersection
\[
U_\ell\cap U_{\ell+1} = \mathscr N_{\ell+} \sqcup \mathscr N_{\ell-}
\]
and that the union $U_\ell\cup U_{\ell+1}$ admits the section $\delta_k$ of $\mathscr R_2$ (\ref{delta_k_eq}).
The main difference is that the intersection numbers with $D_v^\ell$ are undefined over $\mathscr N_{\ell+}$, while the intersection numbers with $D_u^\ell$ are undefined over $\mathscr N_{\ell-}$ because it follows from (\ref{pi_of_divisors}) that $\pi(D_u^\ell)=H_{\ell+}$ and $\pi(D_v^\ell)=H_{\ell-}$.
Due to (\ref{beta_I_back_to_beta}), it suffices to study $\beta_\ell$ and $\beta_{\ell+1}$.

Over $\mathscr N_{\ell+}$, we may write
\[
\beta_{\ell+1} = \beta_\ell + \lambda  \delta_\ell+ \sum_{j=1}^n \nu_{j} S_j 
\]
for some integers $\lambda$ and $\nu_{j}$.
Observe that by Remark \ref{Maslov_0_explicit_rmk}, we have $\delta_\ell|_{\mathscr N_{\ell+}} \cdot D_u^k = 1$ for $k=\ell$ and $0$ for other $k$.
Since the intersection numbers with $D_u^k$'s make sense over $\mathscr N_+$, applying (\ref{beta_I_Duv_eq}) implies:
\[
\begin{cases}
	0=- \nu_1 \\
	0= -\nu_{k+1} +\nu_k  & \text{if } 0<k<\ell \\
	0=\lambda -\nu_{\ell+1} +\nu_\ell  \\
	0=-\nu_{k+1} +\nu_k & \text{if } k>\ell  \\
	0=\nu_n
\end{cases}  \  .
\]
Thus, we conclude that
$
\beta_{\ell+1}=\beta_\ell$ over $\mathscr N_{\ell+}$.

Over $\mathscr N_{\ell-}$, we may similarly write
\[
\beta_{\ell+1} =\beta_\ell +\lambda'\delta_\ell +\sum_{j=1}^n \nu_j'S_j 
\]
for some integers $\lambda'$ and $\nu_j'$. By Remark \ref{Maslov_0_explicit_rmk}, we know $\delta_\ell|_{\mathscr N_{\ell-}} \cdot D_v^k = -1$ for $k=\ell$ and $0$ for other $k$.
Now, the intersection numbers with $D_u^k$'s are no longer well-defined, but the ones with $D_v^k$'s still make sense. It follows from (\ref{beta_I_Duv_eq}) that
\[
\begin{cases}
	1=1+ \nu'_1 \\
	1=1+\nu'_{k+1} -\nu'_k  & \text{if } 0<k<\ell \\
	1=-\lambda' +\nu'_{\ell+1} -\nu'_\ell  \\
	0=\nu'_{k+1} - \nu'_k & \text{if } k>\ell  \\
	0=\nu'_n
\end{cases}  \  .
\]
Hence, we obtain
$
\beta_{\ell+1}=\beta_\ell-\delta_\ell$ over $\mathscr N_{\ell-}$.
In summary,
\begin{equation}
	\label{beta_ell_beta_ell+1_eq}
	\beta_\ell =\begin{cases}
		\beta_{\ell+1} & \text{over} \ \mathscr N_{\ell+} \\
		\beta_{\ell+1} +\delta_\ell & \text{over} \ \mathscr N_{\ell-}
	\end{cases}
\end{equation}
Together with the relation (\ref{beta_I_back_to_beta}) and (\ref{delta_ell_ell-1_relation}), we have completely determined the absolute monodromy information of the local system $\mathscr R_2$. The general formula
is as follows:
\begin{equation}
	\label{beta_ell_beta_ell+1_general_eq}
	\beta_{\ell, I} =
	\begin{cases}
		\beta_{\ell+1,I}  & \text{over} \ \mathscr N_{\ell+} \\
		\beta_{\ell+1, I\cup\{\ell\} } & \text{over} \ \mathscr N_{\ell-} 
	\end{cases}
\end{equation}

\subsection{Atlas of integral affine structure from topological wall-crossing}
By the `topological wall-crossing' we mean an overall understanding of the monodromy of the local system $\mathscr R_2$. Applying the symplectic form $\omega$ to the disks, as sections of the local system $\mathscr R_2$, yields the integral affine coordinates.

For $0\leqslant k\leqslant n+1$, there is a local frame (also called basis) of $\mathscr R_2$ over $U_k$ given by $\{\delta_k,\beta_k\}$. Denote the corresponding integral affine coordinate chart by
\begin{equation}
	\label{chi_k_eq}
	\chi_k: U_k\to  V_k\subset  \mathbb R^2 \qquad q=(s,r) \mapsto ( E(\delta_k(q)), E(\beta_k(q)))  \equiv (s, E(\beta_k))
\end{equation}
where $E(\beta_k(q))=\frac{1}{2\pi} \int_{\beta_k(q)} \omega$ and $E(\delta_k(q))=\frac{1}{2\pi} \int_{\delta_k(q)} \omega$ agree with the moment map $\mu$ by (\ref{moment_map_restrict_V_ell_eq}).
Recall that the collection $\{U_k\}$ forms an open covering of $B_0$. By virtue of (\ref{U_k_eq}), $U_k$ only intersects with the adjacent $U_{k-1}$ and $U_{k+1}$, and $U_k\cap U_{k+1}=\mathscr N_{k+} \sqcup \mathscr N_{k-}$ for any $0\leqslant k\leqslant n$.
Further exploiting the relation (\ref{beta_ell_beta_ell+1_eq}) concludes that
\begin{equation}
	\label{E_ell_ell+1_eq}
	E(\beta_k)=E(\beta_{k+1})+\min\{0,s\}
\end{equation}
In other words, for $(s,t)\in \chi_k(\mathscr N_{k+}\sqcup \mathscr N_{k-})$, we have
$
\chi_{k} \circ \chi_{k+1}^{-1} (s,t)  =  (s, t +\min\{0,s\} )
$
where we view $t$ as $E(\beta_{k+1})$ which is sent to $E(\beta_{k+1})+\min\{0,s\}$ and is exactly identified with $E(\beta_k)$.

On the other hand, we define
\begin{equation}
	\label{psi_U_k}
\psi(q)=\psi(s,r)=
	E(\beta_k)+k \min\{0,s\} \qquad \text{when} \
q\in U_k
\end{equation}
By (\ref{E_ell_ell+1_eq}), this assignment initially gives a well-define smooth function on $B_0$, denoted as $\psi:B_0\to\mathbb R$.
By (\ref{beta_I_back_to_beta}), we remark that $E(\beta_k) +ks=E(\beta_k + k\delta_k)=E(\beta_{k,[k]})$ for $s<0$. Since there is no wall-crossing over $R_k$ and $\beta_{k,[k]}$ can be represented by a holomorphic disk at $q_k$, we know $\beta_{k,[k]}$ can be represented by a holomorphic disk at any point $q\in R_k$. Thus, $E(\beta_{k,[k]})$ is always positive on $R_k$. So, $\psi$ is a positive smooth function on $B_0$. Besides, since the Lagrangian fibration $\pi_0$ over $B_0$ continuously extend to the singular Lagrangian fibration $\pi$ over $B=B_0\sqcup \Delta$, we see that $\psi$ can be continuously extended over the singular points $\Delta$.
In summary, we obtain a continuous function (cf. Figure \ref{figure_psi(s,r)})
\begin{equation}
	 \label{psi_over_B_eq}
 \psi: B\to \mathbb R_+
\end{equation}
such that the relation (\ref{psi_U_k}) holds and $(s,\psi)$ form a set of action coordinates on any contractible open subset in $B_0$.


The following estimates of the symplectic areas turn out to be quite useful in the non-archimedean analytic side. One may readily accept it given the intuition of Figure \ref{figure_psi(s,r)}.

\begin{prop}
	\label{increasing_psi_prop}
	For a fixed $s$, the function $\psi(q)=\psi(s,r)$ is increasing in $r$.
\end{prop}

\begin{proof}
The proof is almost identical to the one in \cite{Yuan_local_SYZ}. Observe first that the symplectic area is topological.
Let $u=u_r$ be a topological disk whose symplectic area is $\psi(s,r)$. Taking a homotopy if necessary, we may assume $u$ is contained in the level set $\mu^{-1}(s)$ of the moment map $\mu$. Then, $\int u^* \omega= \int (p\circ u)^* \omega_{red,s}$ where $\omega_{red,s}$ is the reduced form and the projection map $p:(u,v,z)\mapsto z$ identifies the reduced space $X_{red,s}$ with $\mathbb C$. Since the disk $p\circ u$ is exactly enclosed by the circle of the radius $r$ centered at the origin in $\mathbb C$, the $\omega$-symplectic area of $u$, being identified with the $\omega_{red,s}$-symplectic area of $p\circ u$, is increasing in $r$.
\end{proof}

\begin{rmk}
	\label{psi_reduced_kahler_geometry_rmk}
	It is worth noting that the function $\psi: B\to\mathbb R$ given in (\ref{psi_over_B_eq}) effectively reflects the K\"ahler geometry of the $A_n$-smoothing. In fact, consider that we can identify the reduced space at moment $s$ with the complex plane $(\mathbb C, \omega_{red,s})$ equipped with the reduced K\"ahler forms. In this context, $\psi(s,r)$ represents the $\omega_{red,s}$-symplectic area of the disk with radius $r$, centered at the origin.
	In summary, the function $\psi$ essentially encapsulates all the reduced K\"ahler spaces associated with the natural $S^1$-action (\ref{S_ell_intersection_number}).
	Intriguingly, although it plays a crucial role, $\psi(s,r)$ is only continuous on $B$ and smooth on $B_0$. This limitation serves as a key justification for employing the concept of tropically continuous fibration.
\end{rmk}

\section{Quantum correction and T-duality construction}

The previous section is almost purely topological. Now, we aim to enrich the wall-crossing picture by including both the ingredients of the symplectic geometry and the non-archimedean geometry.

\subsection{Symplectic and non-archimedean integrable system: review}
\label{ss_integrable_system}

In the symplectic world, the Arnold-Liouville theorem tells that a Lagrangian fibration $\pi:X\to B$ admits the action-angle coordinates near the smooth fiber $L_q$ over a point $q$ in the smooth locus $B_0$.
Roughly put, a base point $q$ is called \textit{smooth} if there is a neighborhood $U$ of $q$ such that the fibration $\pi^{-1}(U) \to U$ is isomorphic to $\mathrm{Log}^{-1}(V)\to V$ that covers an integral affine coordinate chart $U\to V$. Namely, we have the following commutative diagram:
\[
\xymatrix{
	\pi^{-1}(U)\ar[d]^\pi \ar[rr] & & \mathrm{Log}^{-1}(V) \ar[d] \\
	U\ar[rr] & & V
}
\]
Here
\begin{equation}
	\label{Log_map_eq}
	\mathrm{Log} : (\mathbb C^*)^n \equiv \mathbb R^n \times (\mathbb R /2\pi \mathbb Z)^n \to \mathbb R^n
\end{equation}
is the natural map sending 
\[
z_k=e^{r_k+i\theta_k} \qquad (r_k\in\mathbb R, \quad \theta_k\in\mathbb R/2\pi\mathbb Z, \quad \text{and  } 1\leqslant k\leqslant n)
\]
to $r_k$, namely, $z_k\mapsto \log|z_k|$. Besides, the total space of $\mathrm{Log}$ is equipped with the standard symplectic form $\sum_k dr_k\wedge d\theta_k$ on the complex torus.
Here $(r_k)$ and $(\theta_k)$ are called action and angle coordinates respectively.

Now, in the world of non-archimedean geometry, there is an extremely similar definition, called \textit{affinoid torus fibration}, due to Kontsevich-Soibelman \cite[\S 4]{KSAffine}.

Let $\Bbbk$ be a non-archimedean field. The example we keep in mind is the \textit{Novikov field}
\[
\Lambda=\mathbb C((T^{\mathbb R}))=\left\{ x=
\sum_{i=0}^\infty a_i T^{\lambda_i} \mid a_i\in\mathbb C, \lambda_i \nearrow \infty
\right\}
\]
which admits a non-archimedean valuation map (resp. norm)
\[
\val(x) =\min\{ \lambda_i \mid a_i\neq 0\}  \qquad (\text{resp.} \ |x|=e^{-\val(x)})  \  .
\]
This field is commonly used in the realm of symplectic geometry especially concerning various Floer-theoretic invariants, but temporarily we just work with a general non-archimedean field $\Bbbk$.

A non-archimedean version of GAGA principle exists, allowing for a functorial association from any algebraic variety $Y$ (or a scheme of locally finite type over $\Bbbk$) to a Berkovich analytic space $Y^{\mathrm{an}}$ (see e.g. \cite[\S 3.4]{Berkovich_2012spectral}).
We will not delve into Berkovich theory, as our focus is on the analytification of an affine algebraic variety over $\Bbbk$.
Set-theoretically, the analytification $Y^{\mathrm{an}}$ comprises all closed points of $Y$ and certain extra generic points of non-archimedean seminorms in the Berkovich sense.
However, in practice, it is often sufficient for an overall understanding to focus only on the closed points and think of the generic points implicitly.
For readers more familiar with Tate's rigid analytic geometry \cite{Tate_origin}, it is worth noting that there is a one-to-one correspondence between rigid analytic spaces and good Berkovich spaces \cite{Berkovich1993etale}. One may alternatively consider the rigid analytification $Y^{rig, \mathrm{an}}$ of the algebraic variety $Y$. Set-theoretically, it is the same as the set of closed points of $Y$ \cite[p. 113]{BoschBook}; however, it only admits the structure of a G-topology, which is not a genuine topological space. Consequently, in most cases, modern Berkovich geometry is a superior theory, as the addition of generic points renders $Y^{\mathrm{an}}$ a true topological space. Moreover, when the algebraic variety $Y$ is separated, it is known that $Y^{\mathrm{an}}$ is a Hausdorff space \cite[3.4.8]{Berkovich_2012spectral}.

Initially, we consider the algebraic variety $\mathbb G_m^n=\mathrm{Spec} (\Bbbk[y_1^\pm,\dots, y_n^\pm])$, and let $(\mathbb G_m^n)^{\mathrm{an}}$ be its Berkovich analytification.
However, by the above discussion and for clarity, let's only consider the set $(\Bbbk^*)^n$ of the closed points in the analytic torus $(\mathbb G_m^n)^{\mathrm{an}}$ where we put $\Bbbk^*= \Bbbk\setminus\{0\}$. Next, we consider the \textit{tropicalization map}:
\begin{equation}
	\label{trop_map_eq}
	\trop: (\Bbbk^*)^n \to \mathbb R^n
\end{equation}
defined by $z_k \mapsto -\log |z_k|\equiv \val(z_k)$.
Be careful that here we use the non-archimedean norm and valuation.
Observant readers may discern that it is fundamentally provided by the same formula as the previous $\mathrm{Log}$ over $\mathbb C$, with the sole distinction being the usage of the norm on $\Bbbk$ rather than on $\mathbb C$. Consequently, by substituting the map $\mathrm{Log}$ over $\mathbb C$ with the map $\trop$ over $\Bbbk$, the Lagrangian fibration in the symplectic context should admit a fairly natural counterpart in the non-archimedean context. Further elaboration is provided below.

Let $f:  Y^{\mathrm{an}} \to B$ be a continuous map.
Note that the continuity makes sense since a Berkovich space is a topological space as said above.
A base point $q$ is called \textit{\textbf{smooth}} or \textit{\textbf{$f$-smooth}} if there is a neighborhood $U$ of $q$ such that the fibration $f^{-1}(U)\to U$ is isomorphic as analytic spaces over $\Bbbk$ to a fibration $\trop^{-1}(V)\to V$ that covers a homeomorphism $U\to V$.
Namely, we have a very similar commutative diagram as follows: (cf. \cite{KSAffine}	and \cite{NA_nonarchimedean_SYZ})
\begin{equation}
	\label{affinoid_torus_fibration_diagram}
\xymatrix{
	f^{-1}(U)\ar[d]^\pi \ar[rr] & & \trop^{-1}(V) \ar[d] \\
	U\ar[rr] & & V
}
\end{equation}
Remark that the isomorphism $f^{-1}(U)\to \trop^{-1}(V)$ amounts to specify $n$ invertible analytic functions on $f^{-1}(U)$.
If we write $B_0$ for the set of all $f$-smooth points, then due to \cite[\S 4, Theorem 1]{KSAffine}, there is a natural integral affine structures on $B_0$ induced by the above homeomorphisms $U\cong V$. More specifically, let $U\subset B_0$ be a small open subset, the set of integral affine functions on $U$ is given by
\[
\mathrm{Aff} (U)=\{ \val(h) \mid h \ \text{is an invertible analytic function on $f^{-1}(U)$} \ \}  \  .
\]

The restriction of such $f$ over $B_0$ is often called an \textit{affinoid torus fibration}, the non-archimedean analog of smooth Lagrangian torus fibration or Hamiltonian integrable system in the symplectic context.

\subsubsection{Definition of tropically continuous map}
\label{sss_tropically_continuous_map_defn}

Let $\mathcal Y$ be a Berkovich analytic space over a non-archimedean field and let $\mathcal B\subset \mathbb R^m$ be a topological manifold that is embedded into the Euclidean space $\mathbb R^m$.
Let $F:\mathcal Y\to  \mathcal B$ be a continuous map with respect to the Berkovich topology on $\mathcal Y$ and Euclidean topology.

\begin{defn}
	\label{tropically_continuous_defn}
	We say $F$ is \textit{tropically continuous} if for any point $x$ in $\mathcal Y$, there exists non-zero rational functions $f_1,\dots, f_N$ on an analytic neighborhood $\mathcal U$ of $x$, and there exists a continuous map $\varphi:U\to\mathbb R^m$ on an open subset $U\subset [-\infty, +\infty]^N$ such that
	$
	F|_{\mathcal U} = \varphi(\val(f_1),\dots, \val(f_N))
	$.
\end{defn}


\begin{rmk}
	Following the definition by Chambert-Loir and Ducros in their work \cite[(3.1.6)]{Formes_Chambert_2012} closely, the functions $f_1, \dots, f_N$ should be invertible analytic functions. However, we relax the requirement to only nonzero rational functions, aligning more with Kontsevich-Soibelman's work \cite[\S 4.1]{KSAffine}. Indeed, we need to include a continuous non-smooth function $\varphi: U \to \mathbb{R}^m$ as explained in Remark \ref{psi_reduced_kahler_geometry_rmk} concerning the A-side K\"ahler geometry. This is presented in \cite{Formes_Chambert_2012} but not in \cite{KSAffine}. It may pose new problems in Hodge-theoretic aspects of Berkovich geometry but currently does not impede our geometric scope for SYZ duality in Conjecture \ref{conjecture_our_SYZ}.
\end{rmk}

\begin{rmk}
	Under the above definition, if $U \subset \mathbb{R}^N$ with $N = m$ and $\varphi$ represents the identity function up to an integral affine transformation, then the local structure recovers that of an affinoid torus fibration. Tropical continuous fibration is a notion that serves to make sense of the singular extension of an affinoid torus fibration in a suitable manner, by imposing additional mild constraints on the extension.
	Moreover, we note that if $\jmath: \mathcal B\to\mathcal B'$ is a homeomorphism, then $\jmath \circ F$ is also a tropically continuous map and one can check that the smooth locus is also preserved. In practice, choosing an appropriate $\jmath$ can transfer $F$ into a more explicit form.
\end{rmk}

\subsection{Family Floer SYZ duality construction: generalities}
\label{ss_family_floer_dual}

\subsubsection{Background and motivation}
The Strominger-Yau-Zaslow (SYZ) conjecture originally emerged from string-theoretic physical concepts such as T-duality and D-branes. Nonetheless, translating physical arguments into mathematical conjectures and statements is often a highly challenging task. Following D. Joyce \cite[\S 9.4]{Joyce_book}, a preliminary mathematical approximation of the SYZ conjecture can be stated as follows:

\begin{conjecture}
	\label{conjecture_SYZ_conventional}
	Let $X$ and $Y$ be ``mirror'' Calabi-Yau manifolds. Then (under some additional conditions) there should exist a base manifold $B$ and surjective, continuous maps $\pi:X\to B$ and $f:Y\to B$, with fibers $L_q=\pi^{-1}(q)$ and $F_q=f^{-1}(q)$ for $b\in B$, and a closed set $\Delta$ in $B$ with $B\setminus \Delta$ dense, such that
	
	\begin{enumerate}[(a)]
		\item For each $b\in B\setminus \Delta$, the fibers $L_q$ and $F_q$ are nonsingular special Lagrangian tori in $X$ and $Y$, which are in some sense `dual' to one another.
		\item For each $b\in \Delta$, the fibers $L_q$ and $F_q$ are singular Lagrangian submanifolds in $X$ and $Y$.
	\end{enumerate}
\end{conjecture}

As Joyce \cite{Joyce_book} points out, this version of SYZ conjecture faces challenges. The original T-duality relies on string-theoretic argument and fails near singular fibers. Quantum corrections were expected to modify the moduli space of special Lagrangian branes in an unclear manner. The meaning of ``dual'' in item (a) was unknown, particularly concerning singular fibers in item (b).
Auroux \cite{AuTDual} suggests a Floer-theoretic approach to T-duality with quantum correction, but wall-crossing discontinuity and convergence issues remain. Gross's topological mirror symmetry \cite{Gross_topo_MS} supports the SYZ philosophy for the quintic threefold, but understanding beyond the topological level was limited until the first example in \cite{Yuan_local_SYZ}. Joyce \cite{Joyce_Singularity} demonstrates that the strong form of the SYZ conjecture is not generally correct, as there can be singularity types where the two singular loci of $\pi$ and $f$ do not coincide in $B$.
In response, Joyce \cite{Joyce_book} remarks that \textit{"This does not mean that the SYZ conjecture is false, only that we have not yet found the right statement."}

\subsubsection{T-duality mirror construction: review}

Our approach to a mathematically accurate SYZ conjecture start with the observation of two natural methods, as discussed in \S \ref{ss_integrable_system}, to equip a base manifold with an integral affine structure. 
Let's start with a K\"ahler manifold $X$ and a Lagrangian fibration $\pi: X \to B$. We cover the smooth locus $B_0$ of $\pi$ with small integral affine charts $\chi_i: U_i \to V_i \subset \mathbb R^n$ and artificially select non-archimedean analytic open domains $\trop^{-1}(V_i)$ concerning the tropicalization map (\ref{trop_map_eq}).
Though we know the collection of open domains $\mathrm{Log}^{-1}(V_i) \cong \pi^{-1}(U_i)$ are local pieces of the global Lagrangian torus fibration $\pi_0 := \pi|_{B_0} : X_0 \to B_0$, it remains unclear if the collection $\trop^{-1}(V_i)$ can also be realized as local pieces of a global affinoid torus fibration in the mean time.
The achievement in \cite{Yuan_I_FamilyFloer} is uncovering a unique canonical algorithm to glue local models $\trop^{-1}(V_i)$, stemming from the quantum-correcting holomorphic disks for $\pi_0$ within $X$. 

\begin{figure}
	\centering
		\captionsetup{font=footnotesize}
	\begin{subfigure}{0.25\textwidth}
		\centering
		\includegraphics[scale=0.3]{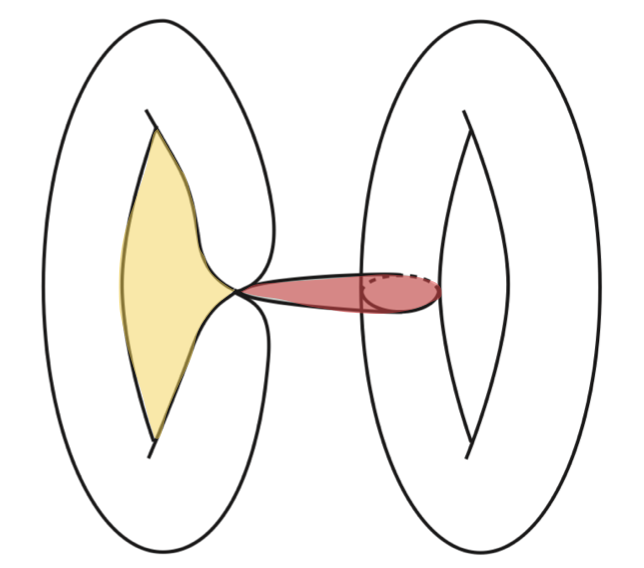}
	\end{subfigure}
	\begin{subfigure}{0.38\textwidth}
		\centering
		\includegraphics[scale=0.28]{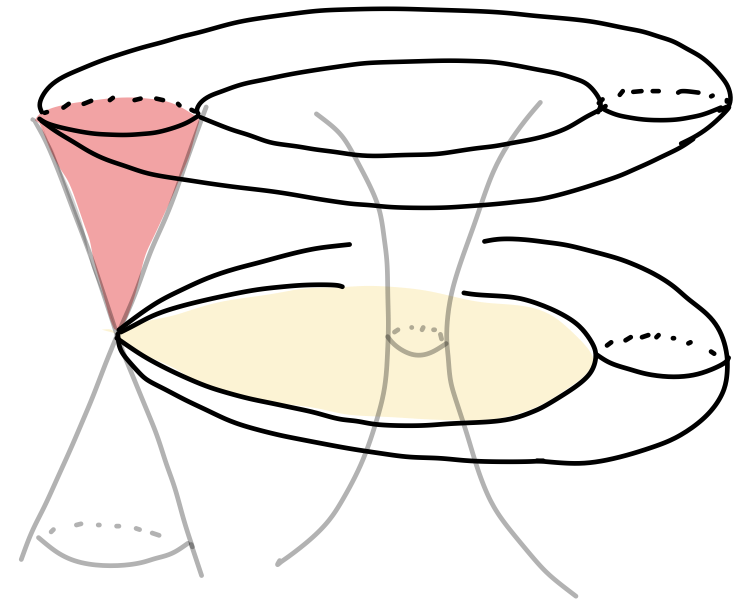}
	\end{subfigure}
	\begin{subfigure}{0.35\textwidth}
		\centering
		\includegraphics[scale=0.39]{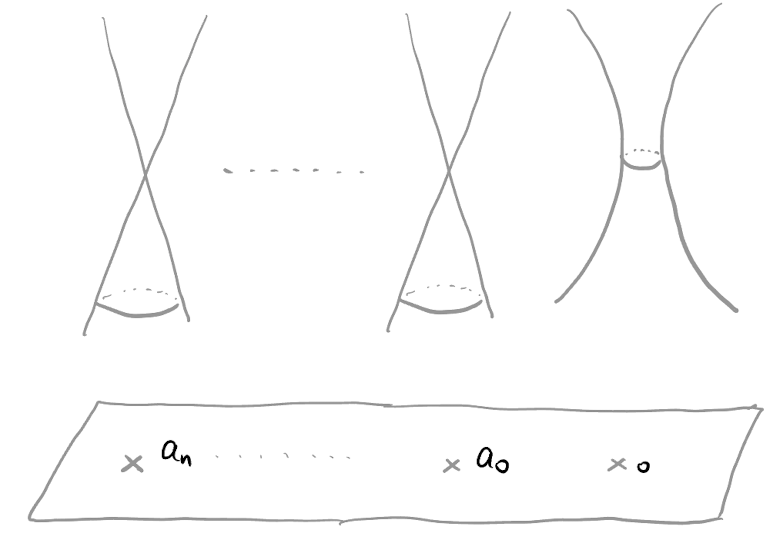}
	\end{subfigure}
	\caption{
		\scriptsize 
		\textit{Left}: General picture. \, \textit{Middle}: the simplest local model in \cite{Yuan_local_SYZ}. \, \textit{Right}: $A_n$ Milnor fiber in our case	}
	\label{figure_area_singular}
\end{figure}

For the given $(X, \pi_0)$, we assume all Lagrangian $\pi$-fibers do not bound nontrivial holomorphic stable disks with negative Maslov index. As per \cite[Lemma 3.1]{AuTDual}, every special or graded Lagrangian satisfies this criterion.
We often require an unobstructedness condition, termed \textit{proper unobstructedness} in \cite{Yuan_unobs}.
This condition holds under a simple sufficient criterion: our Lagrangian fibers of
(\ref{pi_Lagrantian_fibration_eq}) admit the anti-symplectic involution
$
\phi:(u,v,z)\longmapsto (\bar u,\bar v,\bar z)
$ through the complex conjugates.
Specifically, by Solomon's result \cite{Solomon_Involutions}, any potential obstruction from a disk $u$
cancels with its conjugate disk $\phi\circ u$.

The main result of \cite{Yuan_I_FamilyFloer} states the following:

\begin{thm}
	\label{Main_theorem_thesis_thm}
	Given $(X,\pi_0)$ as above, there exists a triple
	$(X_0^\vee,W_0^\vee, \pi_0^\vee)$
	consisting of a non-archimedean analytic space $X_0^\vee$ over $\Lambda$, a global analytic function $W_0^\vee$, and a {\textit{dual affinoid torus fibration}} $\pi_0^\vee: X_0^\vee\to B_0$, which exhibits the following properties:
	
	\begin{enumerate}[(a)]
		\item The non-archimedean analytic structure is unique up to isomorphism.
		\item The integral affine structure on $B_0$ induced by $\pi_0^\vee$ coincides with the one induced by $\pi_0$
		\item The set of closed points in $X_0^\vee$ coincides with 
\[
			\textstyle
			\bigcup_{q\in B_0} H^1(L_q; U_\Lambda) \equiv R^1\pi_{0*}(U_\Lambda)
\]
where $L_q=\pi^{-1}(q)$ and $U_\Lambda=\{ x\in \Lambda\mid |x|=1\}$ is the unit circle in the Novikov field $\Lambda$. Besides, $\pi_0^\vee$ maps every point in $H^1(L_q;U_\Lambda)$ to $q$.
	\end{enumerate}
\end{thm}

\begin{rmk}
	We refer to \cite{Yuan_I_FamilyFloer} for a thorough treatment of the details. However, once the foundational work is established, applying the results becomes relatively straightforward. In practice, understanding the overall framework of the mirror construction and a few properties of it is sufficient. The results can admit quite explicit realizations, as demonstrated in \cite{Yuan_conifold, Yuan_local_SYZ}.
	To realize a concrete example, one still
	needs to make auxiliary choices, such as an atlas of integral affine charts on the base, etc. The
	point is that Theorem \ref{Main_theorem_thesis_thm} ensures that, as long as one follows the family Floer-theoretic
	construction, the mirror can always be constructed regardless of these choices, and any two outcomes
	differ only up to analytic isomorphism.
\end{rmk}

\subsubsection{Local picture in an easy-to-use manner}
\label{sss_local_picture_readytouse}
Following \cite{Yuan_local_SYZ}, we give a brief review as follows.

Let 
$
\chi: (U,q_0) \xrightarrow{\cong} (V,c)\subset \mathbb R^n
$
be a pointed integral affine coordinate chart in $B_0$, meaning that $\chi$ is the restriction of an integral affine chart $\tilde\chi$ with $\tilde\chi(q_0)=c$; for simplicity, we will usually suppress explicit mention of $\tilde\chi$.
Here we further require that $U$ is sufficiently small and $q_0$ is sufficiently close to $U$ so that the reverse isoperimetric inequalities hold uniformly over a neighborhood of $U\cup \{q_0\}$ (see the appendix in \cite{Yuan_I_FamilyFloer}).
Then, we have an identification
\begin{equation}
	\label{affinoid_torus_chart}
	\tau:	(\pi_0^\vee)^{-1}(U) \xrightarrow{\cong} \trop^{-1}(V-c)
\end{equation}
such that $\trop\circ \ \tau=\chi\circ \pi_0^\vee$. Compare the diagram in (\ref{affinoid_torus_fibration_diagram}) for the definition of affinoid torus fibration.
Note that set-theoretically the left side is the disjoint union
\[
(\pi_0^\vee)^{-1}(U)\equiv \bigcup_{q\in U} H^1(L_q; U_\Lambda)  \  .
\]
A closed point $ \mathbf y$ in the dual fiber $H^1(L_q; U_\Lambda)$ can be viewed as a group homomorphism $\pi_1(L_q)\to U_\Lambda$ (or a flat $U_\Lambda$-connection modulo gauge equivalence), so we have the natural pairing 
\begin{equation}
	\label{connection_view}
	\pi_1(L_q)\times H^1(L_q; U_\Lambda)\to U_\Lambda, \qquad  (\alpha, \mathbf y) \mapsto \mathbf y^\alpha
\end{equation}
Write $\chi=(\chi_1,\dots, \chi_n)$, and it gives rise to a family $e_i=e_i(q)$ of $\mathbb Z$-bases of $\pi_1(L_q)$ for all $q\in U$. Alternatively, in view of (\ref{local_system_H_2_H_1_eq}), it induces a frame of the local system $\mathscr R_1$ over $U$.
Then, the corresponding affinoid tropical chart $\tau$ has a very concrete description:
\begin{equation}
	\label{affinoid_torus_explicit_eq}
	\tau(\mathbf y)= (T^{\chi_1(q)} \mathbf y^{e_1(q)},\dots, T^{\chi_n(q)} \mathbf y^{e_n(q)} )
\end{equation}

By Theorem \ref{Main_theorem_thesis_thm}, there is a canonical way to glue all these local affinoid charts. A short review is as follows.
Let's consider two pointed integral affine charts on $B_0$. Taking their intersection if necessary, we may assume they are defined over the same small open subset $U\subset B_0$ and write
\[
\chi_1=(\chi_{11},\dots, \chi_{1n}) : (U, q_1)\to (V_1, c_1) \qquad \text{,}\qquad 
\chi_2=(\chi_{21},\dots, \chi_{2n}) : (U, q_2)\to (V_2, c_2)  \  .
\]
Accordingly, we have two affinoid tropical charts
\[
\tau_1: (\pi_0^\vee)^{-1}(U) \to \trop^{-1}(V_1 -c_1) \qquad \text{,} \qquad
\tau_2: (\pi_0^\vee)^{-1}(U) \to \trop^{-1}(V_2-c_2)
\]
such that $\chi_i\circ \pi_0^\vee=\trop\circ \ \tau_i$.
By Theorem \ref{Main_theorem_thesis_thm}, there exists a unique transition map
\[
\Phi= \trop^{-1}(V_1-c_1)\to \trop^{-1}(V_2-c_2)
\]
determined by the symplectic information of $(X,\pi_0)$ such that the analytic cocycle conditions among all such transition maps hold.
Roughly speaking, the transition map $\Phi$ is determined by two aspects. First, we study the virtual counts of Maslov-0 holomorphic disks (cf. the red disks in Figure \ref{figure_area_singular}) along a Lagrangian isotopy between $L_{q_1}$ and $L_{q_2}$, which is addressed by the bifurcation moduli space
\begin{equation}
	\label{moduli_space_family}
\bigcup_{t\in[0,1]} \{t\} \times \mathcal M (L_{q(t)}; J)
\end{equation}
where $J$ is the (almost) complex structure and the $t\mapsto q(t)$ represents a path from $q_1$ to $q_2$ and where $\mathcal M(L_{q(t)}; J)$ denotes the moduli space of $J$-holomorphic curves $u:(\Sigma, \partial\Sigma)\to (X,L_{q(t)})$ of Maslov index zero.
Second, we use homological perturbation theory (cf.~\cite{FOOO_canonical_model_Morse_complex}) to
collect the contributions of all Maslov-$0$ disks to the curvature term. The delicate point is that,
even if there is only a single Maslov-$0$ disk, one obtains infinitely many such contributions; see
Figure~\ref{figure_hp_maslov_0} and compare the discussion in \S\ref{ss_miscell_Dehn_affinoid_coeff}.

Let's briefly describe $\Phi$ in a coordinate-free way as follows. We consider
\[
\phi:=\tau_2^{-1}\circ \Phi \circ \tau_1:
(\pi_0^\vee)^{-1}(U)\equiv \bigcup_{q\in U} H^1(L_q; U_\Lambda) \to \bigcup_{q\in U} H^1(L_q; U_\Lambda)
\]
and write $\tilde {\mathbf y}=\phi(\mathbf y)$. Be cautious that the both sides of $\phi$ are only presented set-theoretically for clarity, and we have not specify the structure sheaf. Then, using the notation in (\ref{connection_view}), we have
\[
\tilde {\mathbf y}^\alpha = \mathbf y^\alpha  \exp \langle \alpha, \mathfrak F(\mathbf y)\rangle
\]
where
\begin{equation}
	\label{transition_map_F_eq}
	\mathfrak F=   \sum_{\mu(\beta)=0} T^{E(\beta)} Y^{\partial\beta} \mathfrak f_{0,\beta}
\end{equation}
is a vector-valued adic-convergent formal power series in $\Lambda[[\pi_1(L_q)]]\hat\otimes H^1(L_q)$. Here we refer to \textit{some} $A_\infty$ homotopy equivalence 
\[
\mathfrak f=\{\mathfrak f_{k,\beta}: H^*(L_q)^{\otimes k} \to H^*(L_q) \mid k\in\mathbb N, \beta\in\pi_2(X,L_q)\}
\]
derived from the parameterized moduli space (\ref{moduli_space_family}).
It is a collection of multi-linear maps $\mathfrak f_{k,\beta}$ satisfying the $A_\infty$ associativity relations.
It merits mentioning that the individual term $\mathfrak f_{k,\beta}$ takes into account not only the sole moduli space in class $\beta$, but also the assemblage of diverse moduli spaces in classes $\beta_1,\dots, \beta_m$, where $\beta_1+\cdots+\beta_m=\beta$.
Under the assumptions of Theorem \ref{Main_theorem_thesis_thm}, any other such $A_\infty$ homotopy equivalence $\mathfrak f'$ will remarkably produce the exact same analytic map $\Phi$. Simply put, the symmetry related to the involution map renders the error due to different choices negligible.
Though calculating a single term $\mathfrak f_{0,\beta}$ is nearly infeasible, the entire formal power series $\mathfrak F$ is uniquely well-defined and explicitly computable under certain advantageous circumstances (cf. \cite{Yuan_conifold,Yuan_local_SYZ,Yuan_2022disk}).

To sum up, the transition map $\Phi$ is ascertained by the $A_\infty$ homotopy equivalence $\mathfrak f$, which originates from the parameterized moduli space (\ref{moduli_space_family}). Notwithstanding the reliance on selections, any alternative $A_\infty$ homotopy equivalence yields an identical analytical map $\Phi$; see \cite[Section 4.4]{Yuan_I_FamilyFloer}. Moreover, although determining individual terms $\mathfrak f_{0,\beta}$ can be elusive, the formal power series $\mathfrak F$, taken collectively, is uniquely well-defined and explicitly computable, albeit indirectly, in particular propitious situations.

\subsubsection{Dependence on choices}
	\label{ss_family_floer_transition_map}
	
In this section, we provide instructive commentary to elucidate why the dependence on choices is of significant concern.

The moduli space represented in (\ref{moduli_space_family}) is typically characterized by a high degree of singularity, necessitating the selection of perturbations to define virtual counts (refer to \cite{FOOODiskOne, FOOODiskTwo, FOOOKuBook}). A formidable question arises regarding whether the transition map $\Phi$ might be heavily dependent on these various choices. However, it is crucial to note that the cocycle conditions required for local-to-global analytic gluing pertain to `equality' rather than `isomorphism'.
It is worth emphasizing that automorphism groups are studied intensively in mathematics, as they explicitly delineate the difference between isomorphisms and equality. We are not willing to say that every automorphism of a mathematical structure is really just the identity map, since under such a perspective the automorphism group would evaporate.
In our situation, let $V$ be an open subset in $\mathbb R^n$, the automorphism group of $\trop^{-1}(V)$ in the category of analytic spaces is indeed very huge. For instance, for any arbitrary positive real numbers $\epsilon_1,\dots, \epsilon_n$ and polynomials $f_1,\dots, f_n$ whose coefficients have sufficiently positive valuations, the assignment $y_k\mapsto y_k(1+T^{\epsilon_k} f_k)$ is an automorphism on $\trop^{-1}(V)$. This shows the difficulty and complexity of the problem.
	
Fortunately, the issue of choice-dependence is effectively addressed in \cite{Yuan_I_FamilyFloer} through the skillful integration of various concepts. The solution can be broadly divided into two primary aspects.
Firstly, we take advantage of homological perturbation.
Among other reasons for adopting this approach, a crucial requirement lies in the fact that the relevant $A_\infty$ structures must be defined over certain finite-dimensional spaces in order to significantly reduce the severity of choice-dependence.
Secondly, it is essential to focus not on individual virtual count but on the entirety of all the virtual counts (cf. Figure \ref{figure_hp_maslov_0}), consolidating the data as a comprehensive whole so as to produce such $A_\infty$ structures. In our specific situation, the ambiguity arising from different choices is mitigated and absorbed by certain homotopy relations among the $A_\infty$ homomorphisms. Accordingly, it suffices to demonstrate that the non-archimedean analytic map $\Phi$ only relies on this kind of homotopy class of such an $A_\infty$ homomorphism. Notably, this implies that the aforementioned $\Phi$ is unique.
	
The final question to address involves determining the appropriate homotopy theory for the $A_\infty$ structures in symplectic geometry, which ensures both the invariance of $\Phi$ and the cocycle condition among the corresponding transition maps. To achieve this, we must develop a specific \textit{geometric} homotopy theory for $A_\infty$ structures that further incorporates the topological information of $\pi_2(X,L_q)$ and the divisor axiom with respect to the cyclic symmetry of the moduli spaces.
By the divisor axiom, we refer to the open-string analog of the divisor axiom in closed-string Gromov-Witten theory, as outlined by Kontsevich and Manin in \cite{Kontsevich_Manin_1994gromov}. The axiom was initially studied by Seidel \cite[(5.10)]{Seidel2006biased} and later by Auroux, Fukaya, and Tu \cite{AuTDual, FuCyclic, Tu}. However, to the best of our knowledge, the divisor-axiom-preserving homotopy theory of $A_\infty$ structures has only been studied in \cite{Yuan_I_FamilyFloer}, which serves as one of the essential components in ultimately establishing Theorem \ref{Main_theorem_thesis_thm}.
It is worth noting that this innovative homotopy theory also plays a crucial role in \cite{Yuan_c_1} beyond the scope of the SYZ conjecture.

\subsubsection{Void wall-crossing}

Let $B_1\subset B_0$ be a contractible open set.
Let $B_2=\{x\in B_0\mid \mathrm{dist} (x, B_1) < \epsilon\}$ be a slight thickening of $B_1$ in $B_0$. We assume it is also contractible and $\epsilon>0$ is a sufficiently small number so that the estimate constant in the reverse isoperimetric inequalities for any Lagrangian fiber over $B_1$ exceeds $\epsilon$ uniformly (see \cite{Yuan_I_FamilyFloer}).
Then, as \cite[Proposition 4.4]{Yuan_local_SYZ}, we have:

\begin{prop}
	\label{trivial_translation_prop}
	Let $\chi: B_2\xhookrightarrow{} \mathbb R^n$ be an integral affine coordinate chart. If for every $q\in B_1$, the Lagrangian fiber $L_{q}$ bounds no non-constant Maslov index zero holomorphic disk, then there is an affinoid tropical chart $(\pi_0^\vee)^{-1} (B_2) \cong \trop^{-1}(\chi(B_2))$.
\end{prop}

\begin{proof}
Since $B_2$ is contractible, we can first single out a fixed pointed integral affine chart $\chi: (B_2,  q_0) \to (V,  c) \subset \mathbb R^n$ for some point $q_0\in B_2$.
	Next, we can cover $B_2$ by pointed integral affine coordinate charts $\chi_i: (U_i, q_i) \to (V_i, v_i)$, $ i\in \mathcal I$. We may require $\chi_i=\chi|_{U_i}$ and the diameters of $U_i$ are less than $\epsilon$. In particular, we may require all $q_i$'s are contained in $B_1$, and there will be no Maslov-0 disks along a Lagrangian isotopy among the fibers between any pair of $q_i$'s inside $B_1$.
	On the other hand, just like (\ref{affinoid_torus_chart}), we have many affinoid tropical charts $\tau_i: (\pi_0^\vee)^{-1}(U_i)\cong \trop^{-1}( V_i-v_i)$.
	Due to the non-existence of the Maslov-0 holomorphic disks, the gluing map $\Phi$ among these tropical charts take the simplest form: $y_k\mapsto T^{c_k}y_k$ for some $c_k\in\mathbb R$. In conclusion, we can get a single affinoid tropical chart by gluing all these $\tau_i$'s.
\end{proof}

\subsubsection{Mirror Landau-Ginzburg superpotentials}
\label{sss_LG_superpotential}
In practice, the various analytic transition maps $\Phi$ are often found explicitly by indirect methods.
One helpful observation is that we can often place the Lagrangian fibration $\pi_0$ in different ambient symplectic manifolds, say $\overline X_1$ and $\overline X_2$, without affecting the Maslov-0 holomorphic disks.
Since the transition maps for the mirror structure only depend on the Maslov-0 disks, as described in (\ref{transition_map_F_eq}), applying Theorem \ref{Main_theorem_thesis_thm} to the two pairs $(\overline X_1,\pi_0)$ and $(\overline X_2, \pi_0)$ yields two tuples $(X_0^\vee, W_i^\vee, \pi_0^\vee)$ possessing the same mirror analytic space $X_0^\vee$ and the same dual affinoid torus fibration $\pi_0^\vee$. The only difference is that they may have different superpotentials $W_i^\vee$ for $i=1,2$, as the latter rely on both Maslov-0 and Maslov-2 disks (cf. Figure \ref{figure_hp_maslov_0}).

Assume $\beta\in \pi_2(X,L_{q_0})$ has Maslov index two, i.e. $\mu(\beta)=2$, and it also induces $\beta \equiv \beta(q)\in \pi_2(X, L_q)$ for any $q$ in a small contractible neighborhood of $q_0$ in $B_0$.
Denote by $\mathsf n_\beta\equiv \mathsf n_{\beta(q)}$ the corresponding \textit{open Gromov-Witten invariant}, the virtual count of holomorphic stable disks in the class $\beta$.
It depends on the base point $q$ and the almost complex structure $J$ in use. For our purpose, unless the Fukaya's trick is applied, we always use the same $J$ in this paper. Then, due to the wall-crossing phenomenon, one may roughly think the numbers $\mathsf n_{\beta(q)}$ will vary dramatically in a discontinuous manner when we move $q$.
Specifically, if $\mathfrak m=\{\mathfrak m_{k,\beta}\}$ is \textit{a} Fukaya's $A_\infty$ algebra on $H^*(L_{q})$ determined by virtual counts of moduli spaces and homological perturbation, then we define $\mathsf n_{\beta}$ to be the number corresponding to $\mathfrak m_{0,\beta}$ under the identification $H^0(L_{q})\cong \mathbb R$.

This observation suggests that understanding the variation in open Gromov-Witten invariants as $q$ moves can provide insight into the behavior of the superpotentials $W_i^\vee$ for different ambient symplectic manifolds. By studying the wall-crossing phenomenon, we can gain a deeper understanding of how the transition maps and mirror structures relate to each other, enabling us to construct the mirror structures more effectively and explicitly.

Now, we describe the superpotential $W^\vee:=W_0^\vee$ in Theorem \ref{Main_theorem_thesis_thm}.
Fix a pointed integral affine chart $\chi:(U,q_0)\to (V,c)$, and pick an affinoid tropical chart $\tau$ that covers $\chi$ as in (\ref{affinoid_torus_chart}).
Then, the local expression of $W^\vee$ with respect to $\tau$ is given by 
\begin{equation}
	\label{superpotential_in_chart_pointwise_eq}
	W^\vee|_\tau : \bigcup_{q\in U} H^1(L_q; U_\Lambda)\to \Lambda,
	\qquad
	\mathbf y\mapsto \sum_{\beta\in \pi_2(X,L_q), \mu(\beta)=2} T^{E(\beta)} \mathbf y^{\partial\beta} \mathsf n_{\beta(q_0)}
\end{equation}
where $\mathbf y\in H^1(L_q; U_\Lambda)$ for any $q\in U$ and we use $\mathsf n_{\beta(q_0)}$ for the fixed $q_0$.

\begin{rmk}
	When there is no nontrivial holomorphic disk of Maslov index zero bounded by $L_{q_0}$, then the number $\mathsf n_{\beta(q_0)}$ is a true invariant. However, when there exists a nontrivial holomorphic disk of Maslov index zero bounded by $L_{q_0}$, then the number $\mathsf n_{\beta(q_0)}$ is \emph{not} an invariant even if we conventionally still call it open Gromov-Witten invariant (cf. \cite{CLL12}). 	
In brief, the issue is the wall-crossing phenomenon \cite{AuTDual}, and the presence of Maslov-0 disks renders it highly sensitive to the choices. Especially, a single Maslov-0 disk can lead to infinitely many different tree diagrams to be considered; see Figure \ref{figure_hp_maslov_0}.
However, fortunately, despite such ambiguities, it follows from Theorem \ref{Main_theorem_thesis_thm} that the mirror Landau-Ginzburg superpotential $W^\vee|_{\tau}$ is well-defined up to affinoid algebra isomorphism. This means that, even though the open Gromov-Witten invariants may be sensitive to choices and not invariant in the presence of Maslov-0 disks, the mirror superpotential remains well-defined and robust. This is an important property that allows us to study mirror symmetry and construct mirror structures, even in cases where the open Gromov-Witten invariants are not straightforwardly invariant.
\end{rmk}

\vspace{-1em}

\begin{rmk}
	{ 
		Compactifying $X$ with an anti-canonical divisor yields a global analytic function, which counts clusters of holomorphic disks intersecting the divisor combined with quantum correction disks contributing to the mirror non-archimedean structure (see Figure \ref{figure_hp_maslov_0}). 
		These functions shape and frame the mirror Berkovich space, and investigating their relations unveils its analytic space structure, even if they are not fully explicit. This may be evidenced in our ongoing work on a duality from negative to positive vertices. In our current case, they are quite explicit yet, and the principle finally leads to (\ref{f_eq_intro}).
	}
\end{rmk}

\begin{figure}
	\centering
		\captionsetup{font=footnotesize}
	\includegraphics[scale=0.35]{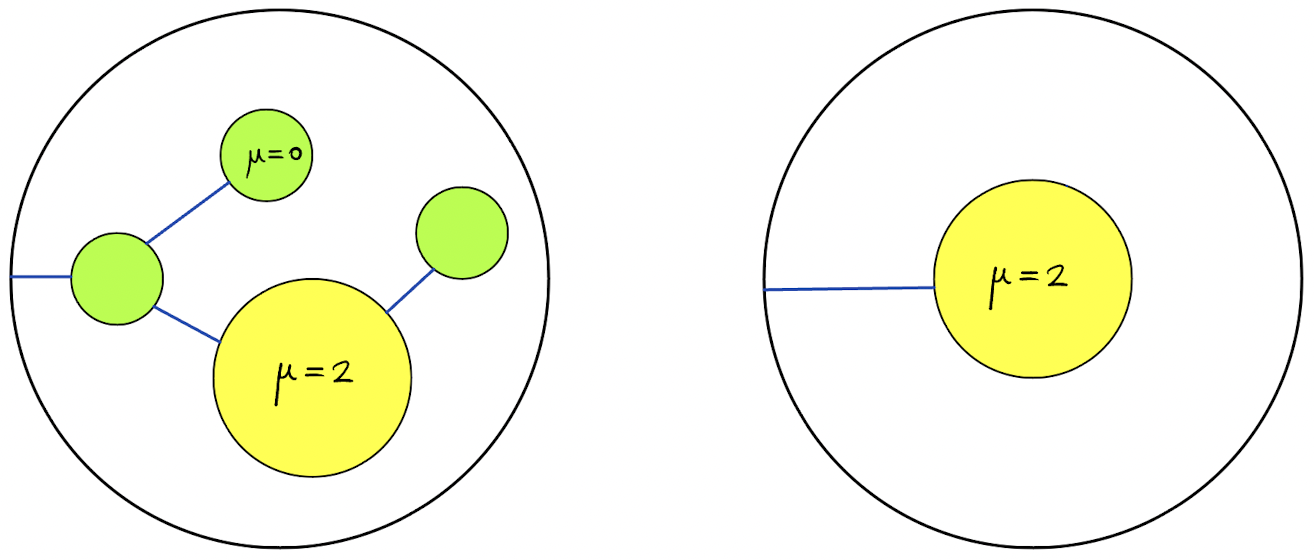}
	\caption{
		\scriptsize A contribution to the Landau-Ginzburg superpotential \textbf{with / without} the inclusion of Maslov-0 disks. It is important to note that even if there is just a single Maslov-0 disk, we must still handle infinitely many tree diagrams.
	}
	\label{figure_hp_maslov_0}
	\vspace{-0.5em}
\end{figure}

In practice, it is more convenient to use `coordinates'. Namely, in view of (\ref{affinoid_torus_chart}), we view $W^\vee|_\tau$ as
\begin{equation}
\label{W_tau_coordinates_eq}
\mathcal W_\tau \equiv W^\vee \circ \tau^{-1}  : \trop^{-1}(V-c)\to (\pi_0^\vee)^{-1}(U)\to \Lambda
\end{equation}
Now, we take two pointed integral affine charts $\chi_a: (U,q_a)\to (V_a, c_a)$ and two corresponding affinoid tropical charts $\tau_a: (\pi_0^\vee)^{-1}(U)\to \trop^{-1}(V_a-c_a)$ for $a=1,2$ as before. Also, let $\Phi$ be the transition map from the chart $\tau_1$ to the $\tau_2$ as before.
Then, due to Theorem \ref{Main_theorem_thesis_thm}, we must have $W^\vee|_{\tau_2} (\phi(\mathbf y)) = W^\vee|_{\tau_1} (\mathbf y)$ for $\phi:=\tau_2^{-1}\circ \Phi\circ \tau_1$; or equivalently, for $y\in \trop^{-1}(V_1-c_1)$, we have
\begin{equation}
	\label{superpotential_gluing_map_eq}
	\mathcal W_{\tau_2}(\Phi(y))=\mathcal W_{\tau_1}(y)
\end{equation}

\subsection{T-duality construction for $A_n$ smoothing}

Let's go back to the example in hand in \S \ref{ss_Lag_fib}.
Since there is no Maslov-0 holomorphic disks in $X$ bounded by $\pi_0$-fibers, applying Theorem \ref{Main_theorem_thesis_thm} to the pair $(X,\pi_0)$ produces a pair $(X_0^\vee, \pi_0^\vee)$.
Besides, as we discussed in \S \ref{sss_LG_superpotential}, we can also apply Theorem \ref{Main_theorem_thesis_thm} to the pair $(\bar X,\pi_0)$, producing an additional analytic function $W^\vee:=W_0^\vee$ on $X_0^\vee$.

Initially, we can use Proposition \ref{trivial_translation_prop} to largely decide the analytic structure of $X_0^\vee$.
Let $U_k$ be the open subsets of $B_0$ defined in (\ref{U_k_eq}).
Then, it follows from Proposition \ref{wall_place_prop} and \ref{trivial_translation_prop}  that we have the following affinoid tropical charts
\[
\tau_k: (\pi_0^\vee)^{-1}(U_k)  \to \trop^{-1}(V_k) \qquad 0\leqslant k\leqslant n+1
\]
that cover the integral affine coordinate charts $\chi_k:U_k\to V_k=\chi_k(U_k)\subset \mathbb R^2$ in (\ref{chi_k_eq}).
Its explicit formula is
\begin{equation}
	\label{tau_k_explicit_eq}
	\tau_k (\mathbf y) = (T^s \mathbf y^{\sigma}, T^{E(\beta_k)} \mathbf y^{\partial\beta_k} )
\end{equation}
Let's write
\begin{equation}
	\label{T_k_eq}
T_k := \trop^{-1}(V_k) \ \subset (\Lambda^*)^2
\end{equation}
for the corresponding analytic open domains.

\subsubsection{Mirror local superpotentials}
We compute the analytic function $W_0^\vee$ explicitly in each chart $\tau_k: (\pi_0^\vee)^{-1}(U_k)\cong T_k$.
Due to Proposition \ref{trivial_translation_prop} and the description (\ref{superpotential_in_chart_pointwise_eq}), it suffices to find $W_0^\vee$ over the fixed point $q_k=(0,r_k)$ in $U_k$ (cf. \S \ref{sss_preferred_disks_beta}) since there is no wall-crossing phenomenon.
In other words, it suffices to find the virtual counts of Maslov-2 holomorphic stable disks (with one boundary marked point) bounded by $L_{q_k}$ up to reparametrization.
Clearly, there is no nontrivial holomorphic sphere in $\bar X$, so it suffices to consider holomorphic disks.
Suppose $\varphi : (\mathbb D, \partial \mathbb D) \to (X,L_{q_k})$ is such a holomorphic disk, and we write
\[
\varphi(\zeta)= (\mathfrak u(\zeta), \mathfrak v(\zeta), \mathfrak z(\zeta) ) \qquad \text{for $\zeta\in\mathbb D$}  \  .
\]
Since it has Maslov index 2, the intersection number $\varphi\cdot \mathscr D=1$, and $\mathfrak z(\zeta)$ has a single zero. Without loss of generality, we may assume $\mathfrak z(\zeta)=r_k\zeta$ up to a reparametrization.
Hence,
\[
\mathfrak u(\zeta) \mathfrak v(\zeta)= h(r_k\zeta)= c (\zeta-a_0)\cdots (\zeta-a_{k-1}) \cdot (1-\bar a_k\zeta) \cdots (1-\bar a_n \zeta)
\]
where we set
\[
a_j=\begin{cases}
	a_j / r_k & \text{if }   0\leqslant j\leqslant k-1 \\
	r_k / \bar a_j & \text{if }  k\leqslant j\leqslant n
\end{cases}
\]
and 
\[
c: =(-1)^{n-k+1}a_k\cdots a_n r_k^k
\]
is a constant. Note that $|a_j|<1$ for all $j$ by (\ref{r_k_eq}).
The zero sets of $\mathfrak u$ and $\mathfrak v$ give a partition of $\{a_0,\dots, a_{k-1}\}$. We can find an index set $I\subset [k]=\{0,1,\dots, k-1\}$ such that the two zero sets are $\{a_i\mid i\in I\}$ and $\{a_i\mid i\in [k]\setminus I\}$.
Now, we write
\[
\mathfrak u(\zeta)= \prod_{i\in I} \frac{\zeta -a_i}{1-\bar a_i \zeta} \cdot  \tilde {\mathfrak u}(\zeta)
\]
and
\[
\mathfrak v(\zeta)= \prod_{i\in [k]\setminus I} \frac{\zeta-a_i}{1-\bar a_i\zeta} \cdot \tilde {\mathfrak v}(\zeta)  \  .
\]
Then,
\[
\tilde {\mathfrak u}(\zeta) \tilde {\mathfrak v}(\zeta)
=
c(1-\bar a_0\zeta) \cdots (1-\bar a_n \zeta)
\]
has no zero. Besides, for $\zeta\in \partial \mathbb D$, we have
\[
|\tilde {\mathfrak u}(\zeta)  |=|\tilde {\mathfrak v } (\zeta)| 
\]
since $|\mathfrak u(\zeta)|=|\mathfrak v(\zeta)|$.
By maximal principle, $\tilde {\mathfrak u}(\zeta)=e^{i\theta} \tilde{\mathfrak v}(\zeta)$ for a fixed $\theta$. Using the $S^1$-action (\ref{S1_action_eq}) if necessary, let's assume $\tilde {\mathfrak u}(\zeta)=\tilde {\mathfrak v}(\zeta)$ for simplicity, and it is exactly the square root of the nonvanishing holomorphic function $\zeta\mapsto c(1-\bar a_0\zeta) \cdots (1-\bar a_n \zeta)$.
This is exactly the holomorphic disk in the class $\beta_{k,I}$ as we constructed in \S \ref{sss_preferred_disks_beta}. By construction, we can also check that its open Gromov-Witten invariant is one.
Therefore, by (\ref{superpotential_in_chart_pointwise_eq}),
\[
W^\vee|_{\tau_k} = \sum_{I\subset [k]} T^{E(\beta_{k,I})} Y^{\partial\beta_{k,I}}  \  .
\]
By using the formulas (\ref{beta_I_back_to_beta}) and (\ref{moment_map_restrict_V_ell_eq}), we further find that
\begin{align*}
	W^\vee|_{\tau_k} &= \sum_{I\subset [k]} T^{E(\beta_k)+|I| s} Y^{\partial\beta_k+ |I| \sigma} \\
	&
	=\sum_{I\subset[k]} T^{E(\beta_k)} Y^{\partial\beta_k} \cdot \big( T^s Y^{\sigma} \big)^{|I|}  \\
	&
	=
	T^{E(\beta_k)}Y^{\partial\beta_k} \sum_{j=0}^k \binom{k}{j} (T^s Y^\sigma)^j \\
	&
	=
	T^{E(\beta_k)} Y^{\partial\beta_k} \left(1+ T^s Y^\sigma \right)^k
\end{align*}
By (\ref{W_tau_coordinates_eq}) and (\ref{tau_k_explicit_eq}), we obtain an analytic function on $T_k$ as follows: for $y=(y_1,y_2)\in T_k$,
\begin{equation}
	\label{W_k_coordinates_eq}
	\mathcal W_k (y) := W^\vee \circ \tau_k^{-1} (y) = y_2(1+y_1)^k  \  .
\end{equation}

\subsubsection{Mirror analytic structure from calculating local superpotentials}

As described in \S \ref{sss_local_picture_readytouse}, 
the two analytic open domains $T_k$ and $T_{k+1}$ are glued along a transition map
\[
\Phi_{k,k+1} : \trop^{-1} (\chi_{k+1}(U_k\cap U_{k+1})) \to \trop^{-1}( \chi_k(U_k\cap U_{k+1}))
\]
where the source and target are analytic subdomains in $T_{k+1}$ and $T_k$ respectively. Note that each of them have two components over $\mathscr N_{k+}$ and $\mathscr N_{k-}$ respectively since $U_k\cap U_{k+1}=\mathscr N_{k+}\sqcup \mathscr N_{k-}$.
We can determine $\Phi_{k,k+1}$ explicitly as follows.
First, concerning the $S^1$-symmetry of the Lagrangian fibration we know the transition map preserves the first coordinate of each $\tau_k$. Besides, the existence of global superpotentials in Theorem \ref{Main_theorem_thesis_thm} implies that the transition map must match the various local superpotentials in the sense that $\mathcal W_{k+1}(y)= \mathcal W_{k}(\Phi_{k,k+1}(y))$ for any $y=(y_1,y_2)$ in the domain of $\Phi_{k,k+1}$; see also (\ref{superpotential_gluing_map_eq}). 
To wit, if we set $(y_1,\tilde y_2)=\Phi_{k,k+1}(y_1,y_2)$, then $y_2(1+y_1)^{k+1}=\tilde y_2 (1+y_1)^{k}$. Since $\val(y_1)\neq 0$ in its domain, $1+y_1$ cannot be zero, and thus $\tilde y_2=y_2(1+y_1)$.
In summary, in its domain, we have
\begin{equation}
	\label{Phi_k_k+1_eq}
	\Phi_{k,k+1}(y_1,y_2) = (y_1, y_2(1+y_1) )
\end{equation}

In view of Theorem \ref{Main_theorem_thesis_thm}, the mirror analytic space $X_0^\vee$ is identified with the adjunction space obtained by gluing all the $T_k$'s via the transition maps $\Phi_{k,k+1}$'s. Namely, it is identified with the disjoint union $\bigsqcup_{k=0}^{n+1} T_k$ modulo the relation $\sim$: we set $y\sim y'$ whenever $y\in T_{k+1}$, $y'\in T_k$, and $\Phi_{k,k+1}(y)=y'$ for some $k$.
Therefore, we have the following explicit identification:
\begin{equation}
	\label{identification_mirror_eq}
	X_0^\vee \equiv \bigsqcup_{k=0}^{n+1} T_k /\sim
\end{equation}
Under this identification, the dual affinoid torus fibration $\pi_0^\vee: X_0^\vee\to B_0$ is also identified with the induced gluing of the various tropicalization maps $\trop : T_k \to V_k $ restricted on $T_k$.
Now, given any $0\leqslant k\leqslant n+1$ and $y=(y_1,y_2)\in T_k$, we have
\begin{equation}
	\label{identification_dual_affinoid_torus_fibration__eq}
	\mathsf v(y_1) = \mathrm{pr}_1 \circ \pi_0^\vee(y)      \qquad \text{and} \qquad \mathsf v(y_2) = \psi( \pi_0^\vee(y)) -k\min\{0,\mathsf v(y_1)\}
\end{equation}
where $\mathrm{pr_1}$ is the projection $\mathbb R^2_{s,r} \to\mathbb R_s$ to the first component and where the second relation holds because of (\ref{chi_k_eq}) and (\ref{psi_U_k}).
Notice that the restriction of $\pi_0^\vee$ on $T_k$ is identified with $\chi_k^{-1}\circ \trop: T_k\to V_k\to U_k$.
Remark that by (\ref{Phi_k_k+1_eq}), the coordinate $y_1\in \Lambda^*$ actually gives rise to a global analytic function on $X_0^\vee$.

\section{Explicit representation of mirror analytic structure}

\subsection{General principles}

The mirror non-archimedean analytic structure in Theorem \ref{Main_theorem_thesis_thm} is generally non-explicit. According to Theorem \ref{Main_theorem_thesis_thm}, there exists an abstract analytic space $X_0^\vee$ over the Novikov field $\Lambda=\mathbb C((T^{\mathbb R}))$, furnished with an affinoid torus fibration $\pi_0^\vee:X_0^\vee \to B_0$. It is unique up to isomorphism and reflects the geometry of the A-side space. This structure is derived from various inherent $A_\infty$ structures, which are generally challenging to explicitly write down.

Consequently, we must accept a certain degree of non-explicitness in general. Indeed, throughout the history of mathematics, numerous non-explicit existence results have emerged, with the Riemann mapping theorem serving as a basic example. Although it may have explicit formulas for a disk or a half-plane, the Riemann mapping is predominantly non-explicit for most simply-connected domains.

\subsubsection{Superpotentials frames the mirror analytic space}
Nonetheless, a major advancement in our foundational groundwork \cite{Yuan_I_FamilyFloer} lies in the incorporation of globally defined mirror superpotentials. As explained in \S \ref{sss_LG_superpotential}, we may consider various compactification spaces, either partial or complete, of the original symplectic manifold, resulting in various global superpotential functions.
The combination of these global functions naturally yields an analytic morphism from the mirror analytic space into the analytification of an affine space. Provided we have sufficient global functions and arrange their combinations appropriately, e.g. ensuring pairwise disjoint zero sets, this analytic morphism becomes an embedding that describes $(X_0^\vee,\pi_0^\vee)$ properly.

Moreover, we have discovered numerous explicit examples in \cite{Yuan_local_SYZ,Yuan_conifold} and in this paper when there are enough symmetries of the Lagrangian fibration to make the superpotentials explicit. Even with fewer symmetries, we may still grasp the information of the mirror analytic space by examining the global superpotentials and their relationships. In an upcoming work, we will demonstrate that dualizing a Lagrangian fibration on the negative vertex variety produces the positive vertex variety. The rationale is that certain partial regions of $X_0^\vee$ where superpotentials can be made explicit are sufficient to determine the mirror analytic structure, even if they cannot be made explicit everywhere. As a matter of fact, the analytic geometry shares a certain degree of rigidity.

\subsubsection{Mirror analytic structure controls singular extension}

An explicit representation generates a deeper understanding of the mirror pair $(X_0^\vee, \pi_0^\vee)$. In our specific example in hand, by examining the quantum correction of the Lagrangian fibration, we produce a moderately explicit model as shown in (\ref{identification_mirror_eq}).
Moreover, as we discussed in the previous paragraph, we will soon see that the framing of superpotentials can be used to derive an explicit embedding from the mirror space $X_0^\vee$ in (\ref{identification_mirror_eq}) into the minimal resolution of the $A_n$ singularity. This embedding enables us to \textit{witness} an explicit model of the dual fibration $\pi_0^\vee$ such that its singular extension over $B=B_0\cup \Delta$ becomes quite straightforward and elementary.

Based on the above discussion, we propose the following principle for the singular extension: \textit{the non-archimedean analytic structure matters a lot}.
 With a singular Lagrangian fibration, the Floer theory of the smooth fibers already captures a substantial amount of information about the singular fibers. Although this may seem counter-intuitive, it is crucial to recognize that symplectic geometry primarily focuses on global properties. The existence of holomorphic disks bounded by smooth fibers is fundamentally global and closely linked to the presence of singular fibers, even if their precise locations remain unknown (see also Figure \ref{figure_area_singular}).

In our specific context, the holomorphic disks lead to a non-archimedean analytic space $X_0^\vee$, as described by Theorem \ref{Main_theorem_thesis_thm}. Compared to the rigidity of holomorphic functions and complex manifolds, a non-archimedean analytic structure is also sufficiently `rigid', limiting the potential freedom of a singular extension. This point of view was vaguely proposed in \cite{Yuan_I_FamilyFloer} and subsequently supported by explicit examples in \cite{Yuan_local_SYZ,Yuan_conifold} and the present paper.

\subsection{Toric geometry for $A_n$ resolution}
\label{ss_toric_geometry}

In our study of toric geometry, we refer to the following standard references \cite{MS_AG_Cox_Katz, Toric_Cox_Little_Schenck}.

\subsubsection{Cones and fans}
We work over the Novikov filed $\Lambda=\mathbb C((T^{\mathbb R}))$. 
Let $N=\mathbb Z^2$ and $M= \mathrm{Hom} (N,\mathbb Z)\cong \mathbb Z^2$ be two lattices dual to each other.
Then, we consider the 2-dimensional fan $\Sigma$ (see Figure \ref{figure_fan}) in $N_{\mathbb R}=N\otimes \mathbb R\cong \mathbb R^2$ generated by the rays 
\[
v_k=(k,1) \in N
\]
for $0\leqslant k\leqslant n+1$.
The toric surface $Y_\Sigma=Y_\Sigma(\Lambda)$ associated to this fan is Calabi-Yau and is known to be a crepant resolution of the $A_n$ singularity.

Let $\sigma_k$ be the cone generated by $v_{k}$ and $v_{k+1}$ for $0\leqslant k\leqslant n$.
In particular, $\sigma_k\cap \sigma_{k+1}=v_{k+1}$.
Now, we set 
\[
m_k=(1,-k)\in M  \  .
\]
Then, $\langle m_k,v_k\rangle =0$, and the dual cone $\sigma_k^\vee$ is generated by $-m_{k+1}$ and $m_{k}$.
We denote the affine toric variety associated to a cone $\sigma$ by
$
Y_{\sigma}= \mathrm{Spec} ( \Lambda[\sigma^\vee \cap M])
$
where $\sigma^\vee$ denotes the dual cone of $\sigma$. We also denote the toric character associated to $m\in M$ by $\chi^m$.

We choose the coordinates of $Y_{\sigma_k}$ to be
\begin{equation}
	\label{toric_geometry_z_k_w_k_eq}
z_k :=\chi^{-m_{k+1}} \qquad \text{and} \qquad w_k:= \chi^{m_{k}} 
\end{equation}
Notice that
$
Y_{\sigma_k} \supseteq (Y_{\sigma_k})_{\chi^{m_k}} = Y_{v_{k+1}} = (Y_{\sigma_{k+1}})_{\chi^{m_k}}  \subseteq  Y_{\sigma_{k+1}}
$.
Here $(Y_{\sigma})_{\chi^m}$ denotes the affine open subset of $Y_\sigma$ in which $\chi^m\neq 0$
(see e.g. \cite[\S 3.1]{Toric_Cox_Little_Schenck}).
Concretely,
\[
Y_{\sigma_k} \equiv \Lambda^2_{z_k,w_k}
\qquad 
\text{and}
\qquad 
Y_{\sigma_k} \supseteq
\Lambda^*_{z_k} \times \Lambda_{w_k}  
\equiv
Y_{v_{k+1}} 
\equiv  
\Lambda_{z_{k+1}} \times  \Lambda^*_{w_{k+1}} \subseteq Y_{\sigma_{k+1}}
\]
with the coordinate relations
$w_{k+1}= z_k^{-1}$ and $z_{k+1}w_{k+1}=z_k w_k=\chi^{(0,1)}$.

By the orbit-cone correspondence, every cone $\sigma$ in $\Sigma$ corresponds to a torus orbit $O(\sigma)$. The affine toric chart $Y_{\sigma_k}\equiv \Lambda^2_{z_k,w_k}$ is the disjoint union of $O(\sigma_k)=(0,0)$, $O(v_k)\cong \{0\}\times \Lambda^*_{w_k}$, $O(v_{k+1})\cong \Lambda^*_{z_{k}}\times \{0\}$, and the open dense orbit $\Lambda^*_{z_k} \times \Lambda^*_{w_k}\equiv \Lambda^*_{z_{k+1}}\times \Lambda^*_{w_{k+1}}$.


\begin{figure}
	\centering
	\captionsetup{font=footnotesize}
	\includegraphics[width=10cm]{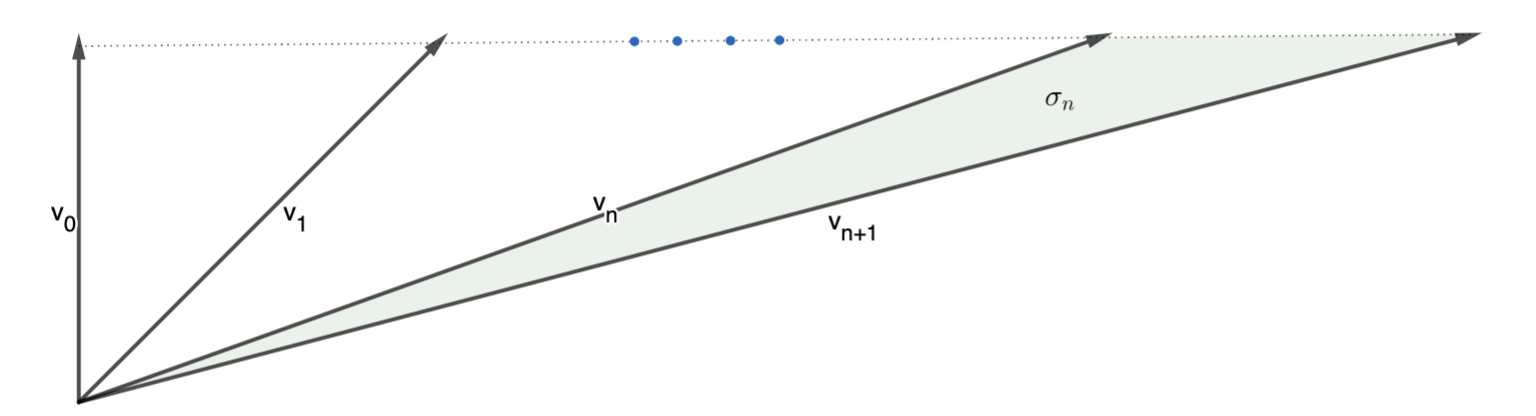}
	\caption{\footnotesize Fan of $A_n$ resolution.}
	\label{figure_fan}
\end{figure}

\subsubsection{Homogeneous coordinates}
\label{sss_homogeneous_coordinates}

Recall that there is a short exact sequence:
\[
0\to M \xrightarrow{V} \mathbb Z^{n+2} \to \mathcal A_{n-1}\to 0
\]
where the first arrow $V$ sends $m$ to the tuple $(\langle m, v_k\rangle: k=0,1,\dots, n+1)$ and $\mathcal A_{n-1}=\mathcal A_{n-1}(Y_\Sigma)$ is the Chow group of Weil divisors in $Y_\Sigma$ modulo the linear equivalence. Since $Y_\Sigma$ is smooth, we know $\mathcal A_{n-1}(Y_\Sigma) \cong Pic(Y_\Sigma)$. 
Notice that $V$ can be regarded as a $(n+2)\times 2$ matrix whose $k+1$-th row is $v_k=(k,1)$ for $0\le k\le n+1$.
We also denote its matrix transpose by $V^T$.

Applying $Hom_{\mathbb Z}( -, \Lambda^*)$ gets another exact sequence
\begin{equation}
	\label{exact_seq_eq}
1\to G\to (\Lambda^*)^{n+2} \xrightarrow{V^T} N\otimes \Lambda^* \to 1
\end{equation}
where the structure of $G$ can be described as a subgroup of $(\Lambda^*)^{n+2}$ (see \cite[Lemma 5.1.1]{Toric_Cox_Little_Schenck}):
\begin{equation}
	\label{G_group_eq}
G
=\ker(V^T)= Hom_{\mathbb Z}( \mathcal A_{n-1}, \Lambda^*)  = \left\{
t=(t_0,t_1,\dots, t_{n+1}) \in (\Lambda^*)^{n+2} \mid  \prod_{k=0}^{n+1} t_k =\prod_{k=0}^{n+1} t_k^k =1
\right\}
\end{equation}

Introduce a variable $x_k\in\Lambda$ for the ray $v_k$ in $\Sigma$ $(0\leqslant k\leqslant n+1)$, and let
\[
S=\Lambda[x_0,x_1,\dots, x_{n+1}]
\]
be the {total coordinate ring} of $Y_\Sigma$ so that we identify $\mathrm{Spec} (S)=\Lambda^{n+2}$.
A subset $C\subset \{v_0,v_1,\dots, v_{n+1}\}$ is called a \textit{primitive collection} if $C$ is not contained in any cone of $\Sigma$ and every proper subset $C'$ of $C$ is contained in some cone of $\Sigma$.
In our case, any primitive collection takes the form
$C_{ij} =\{v_i,v_j\} $ for $j-i\ge 2$.
For each cone $\sigma$, we define the monomial $x(\sigma)=\prod_{j : v_j\notin \sigma} x_j$. In our situation, we have
\[
\hat x_k := x(\sigma_k)=\prod_{j : v_j\notin \sigma_k} x_j =x_0x_1\cdots x_{k-1} x_{k+2}\cdots x_{n+1}=\frac{x_0x_1\cdots x_{n+1}}{x_kx_{k+1}}
\]
for $0\leqslant k\leqslant n$.
The {irrelevant ideal} is
$\mathscr I(\Sigma)=\langle x(\sigma)\mid \sigma\in\Sigma \rangle  =\langle \hat x_k \mid 0\leqslant k \leqslant n \rangle$, and the variety of the irrelevant ideal is a union of irreducible components given by
$
Z(\Sigma) = \bigcup_C  \{x_k=0 \mid v_k\in C\}
$
where the union is over all primitive collections $C$.
Concretely, we have
\begin{equation}
	\label{Z_Sigma_irrelevant}
Z(\Sigma)= \bigcup_{j-i\ge 2, 0\le i,j\le n+1} \{x_i=0, x_j=0\}
\end{equation}

Let $e_0,e_1,\dots, e_{n+1}$ be the standard basis of $\mathbb Z^{n+2}$. For each cone $\sigma$ in $\Sigma$, we define a corresponding cone
$
\tilde \sigma = \mathrm{Cone} ( e_k \mid v_k\in \sigma )
$
in $\mathbb R^{n+2}$.
Then, $\Lambda^{n+2}\setminus Z(\Sigma)$ is the toric variety of the fan $\tilde \Sigma$ generated by these $\tilde \sigma$ (cf. \cite[Proposition 5.1.9]{Toric_Cox_Little_Schenck}).
The integral linear map $\bar\tau: \mathbb Z^{n+2}\to N$ defined by $e_k\mapsto v_k$ is compatible with the fans $\tilde \Sigma$ and $\Sigma$.
The resulting toric morphism
\[
\tau: \Lambda^{n+2}\setminus Z(\Sigma) \to  Y_\Sigma
\]
is constant on $G$-orbits.
It is a geometric quotient and induces an identification (cf. \cite[Theorem 5.1.11]{Toric_Cox_Little_Schenck}):
\begin{equation}
	\label{homogeneous_coordinates_identification}
	 \tau^G: (\Lambda^{n+2} \setminus Z(\Sigma) ) / G  \xrightarrow{\cong} Y_\Sigma 
\end{equation}
Given $p\in Y_\Sigma$, we say a point $x=[x_0:x_1:\cdots :x_{n+1}]\in \tau^{-1}(p)$ gives \textit{homogeneous coordinates} for $p$. This is a generalization of homogeneous coordinates for the projective spaces $\mathbb P^n= (\Bbbk^{n+1}\setminus\{0\})/ \Bbbk^*$.
Because $\tau$ is a geometric quotient, $\tau^{-1}(p)=G\cdot x$ being the $G$-orbit of $x$, and all homogeneous coordinates for $p$ are of the form $t\cdot x$ for $t\in G$. Thus, knowing one of homogeneous coordinates implies all other ones.

We further review the structure of $\tau$ and $\tau^G$ as follows.
The restriction of $\tau$ over an affine toric chart $Y_{\tilde \sigma}$ for a cone $\tilde \sigma$ in the fan $\tilde \Sigma$ gives rise to
a toric morphism 
\[
\tau_\sigma: Y_{\tilde \sigma}\to Y_\sigma  \  .
\]
In particular, choosing the trivial cone retrieves the previous map $V^T: (\Lambda^*)^{n+2}\to N\otimes \Lambda^*$ in (\ref{exact_seq_eq}).
It also induces a geometric quotient and gives the identification
\[
\tau_\sigma^G: Y_{\tilde \sigma} /G \xrightarrow{\cong} Y_\sigma  \  .
\]
Specifically, for the cone $\sigma_k$ in $\Sigma$ and $\tilde{\sigma}_k$ in $\tilde \Sigma$, we note that the coordinate ring of 
\[
Y_{\tilde \sigma_k} \equiv \Lambda^{n+2}\setminus \{ \hat x_k =0\}  \cong \Lambda^2 \times (\Lambda^*)^n
\]
is the localization 
\[
S_{\hat x_k} =\Lambda\left[ \textstyle \prod_{k=0}^{n+1} x_k^{a_k} \mid a_k\ge 0 ,  a_{k+1}\ge 0 \right] = \Lambda[ x_k, x_{k+1}, (x_j^\pm )_{j\neq k,k+1} ]
\]
of $S$ at $\hat x_k$ and that the coordinate ring of $Y_{\sigma_k}$ is $\Lambda[\sigma_k^\vee\cap M]$.
For the coordinates $z_k$ and $w_k$ (\ref{toric_geometry_z_k_w_k_eq}), we also have $Y_{\sigma_k}\equiv \Lambda_{z_k,w_k}^2$.
Then, the morphism $\tau^G_{\sigma_k}$ corresponds to the isomorphism 
\[
(\tau_{\sigma_k}^G)^*: \Lambda[\sigma_k^\vee\cap M] \to S_{\hat x_k}^G
\]
given by $\chi^m\mapsto \prod_{j=0}^{n+1} x_j^{\langle m, v_j \rangle}$. Equivalently, the isomorphism $\tau_\sigma^G$ sends a point $x=[x_0:x_1:\cdots x_{n+1}]$ to $(z_k,w_k)$ with
\begin{equation}
	\label{homogeneous_coordinates_z_k_w_k_eq}
	z_k=\prod_{j=0}^{n+1} x_j^{k+1-j}
	\qquad \text{and} \qquad  w_k=\prod_{j=0}^{n+1} x_j^{j-k}  
\end{equation}
Moreover, by (\ref{toric_degeneration_eq}),
\begin{equation}
	\label{homogeneous_coordinates_1+y+eq}
	1+y=  z_k w_k =\prod_{j=0}^{n+1} x_j
\end{equation}
Observe that for any $t=(t_0, \dots, t_{n+1}) \in G$, substituting $(x_j)$ with $(t_jx_j)$ leaves both Equation (\ref{homogeneous_coordinates_z_k_w_k_eq}) and Equation (\ref{homogeneous_coordinates_1+y+eq}) unchanged. This observation highlights the invariance of these equations under the action of the group $G$.

Conversely, we can describe the affine toric chart $Y_{\sigma_k}$ in terms of homogeneous coordinates as follows (see e.g. \cite[Proposition 5.2.10]{Toric_Cox_Little_Schenck}):
\begin{equation}
	\label{homogeneous_coordiantes_affine_chart}
	\phi_{\sigma_k}: Y_{\sigma_k}\equiv \Lambda^2_{z_k,w_k}  \xhookrightarrow{}  Y_\Sigma \qquad (z_k,w_k) \mapsto [1: \cdots : 1 : x_k=z_k: x_{k+1}=w_k: 1 :\cdots :1]
\end{equation}
under the previous identification (\ref{homogeneous_coordinates_identification}).
Given any homogeneous coordinate $x$ for some $p\in Y_{\sigma_k}$, there exists some $t\in G$ such that the components of $t\cdot x$ are all $1$ except the $k$-th and $(k+1)$-th ones.

\subsubsection{Torus-invariant divisors}
\label{sss_divisor_toric_invariant}
We introduce a global $\Lambda$-variable 
\begin{equation}
	\label{toric_degeneration_eq}
	1+y= \chi^{(0,1)}: Y_\Sigma \to \Lambda
\end{equation}
in $Y_\Sigma$. It comes from the $\mathbb Z$-linear map projecting $N= \mathbb Z^2$ to the second component $\mathbb Z$ compatible with the fan $\Sigma$ and the obvious 1-dimensional fan in $\mathbb Z$ (see \cite[\S 3.3]{Toric_Cox_Little_Schenck}).
Intuitively, $\chi^{(0,1)}=0$ corresponds to a sort of {``toric degeneration''}. In the non-archimedean perspective, this means $\val(\chi^{(0,1)})=\infty$ and is more or less relevant to the notion of tropically continuous map (\S \ref{sss_tropically_continuous_map_defn}).
Beware that we artificially create the variable $y$ and make $\chi^{(0,1)}=1+y$ for our purpose.

Each ray $v_k$ gives a torus-invariant prime divisor 
\begin{equation}
	\label{divisor_torus_invariant_Bside}
	\mathcal D_k :=\overline{O(v_k)} =\bigcup_{\sigma: \text{$v_k$ is a face of $\sigma$}} O(\sigma)
	=
	\begin{cases}
		O(v_0)\sqcup O(\sigma_0)	& \text{if  } k=0 \\
		O(\sigma_{k-1})\sqcup O(v_k) \sqcup O(\sigma_k)	& \text{if  } 1\leqslant k \leqslant n \\
		O(\sigma_n)\sqcup O(v_{n+1})	& \text{if  } k=n+1
	\end{cases}
\end{equation}
We claim that $\mathcal D_0$ and $\mathcal D_{n+1}$ are non-compact and identified with the affine line, while each other $D_k$ for $1\leqslant k\leqslant n$ is isomorphic to the projective line. 
To see this, let $N(v_k)=N/ \mathbb Z v_k$, and we define
\[
\Sigma(v_k)=\{ \bar \sigma \mid  \text{$v_k$ is a face of $\sigma\in \Sigma$}\}
\]
where $\bar \sigma$ is the image cone under the quotient map $N_{\mathbb R} \to N(v_k)_{\mathbb R}=N(v_k)\otimes \mathbb R \cong \mathbb R$.
Then, we just need to apply \cite[Proposition 3.2.7]{Toric_Cox_Little_Schenck}.
Additionally, it is standard to verify that
\begin{equation}
	\label{chi_01=0_eq}
	(\chi^{(0,1)})^{-1}(0)= \textstyle \bigsqcup_{k=0}^n O(\sigma_k) \sqcup \bigsqcup_{k=0}^{n+1} O(v_k)
	=\mathcal D_0\cup \mathcal D_1\cup \cdots \cup \mathcal D_{n+1}
\end{equation}

It is known that the divisor $\mathcal D_k$ is described by the equation $x_k=0$ for the homogeneous coordinates $[x_0:\cdots : x_n]$; see e.g. \cite[Example 5.2.5]{Toric_Cox_Little_Schenck}.
Note that $\chi^{(0,1)}=1+y=0$ precisely means $\prod_{j=0}^{n+1}x_j=0$. 
Moreover, since the homogeneous coordinates avoid the variety $Z(\Sigma)=\bigcup_{j-i\ge 2} \{x_i=x_j=0\}$ of the irrelevant ideal (\ref{Z_Sigma_irrelevant}), the intersection $\mathcal D_i\cap \mathcal D_j$ is empty whenever $j-i\ge 2$. Indeed, for any $0\leqslant k \leqslant n$, $\mathcal D_k\cap\mathcal D_{k+1}=\{x_k=x_{k+1}=0\}=O(\sigma_k)$ consists of a single point.

\subsection{Semi-global and global analytic embedding into $A_n$ resolution}
\label{ss_analytic_embedding_semiglobal_global}

Following \cite{Yuan_local_SYZ,Yuan_conifold}, we aim to find an explicit analytic embedding of $X_0^\vee$ in (\ref{identification_mirror_eq}) into the analytification of an algebraic variety $Y$.
According to the various previous works \cite{Chan_An_Tduality,Chan_Ueda_2013dual,Seidel_Thomas_2001braid,Seidel_Khovanov2002quivers,ishii_Stability_condition_An_ishii2010stability}, we know that the algebraic variety $Y$ is expected to be the \textit{minimal resolution of $A_n$ singularity}.
However, let's pretend that we do not know of this in order to understand how the A-side geometric data reveals the B-side mirror structure and the desired analytic embedding.

The idea is roughly to break the problem into the smaller ones in that the wall-crossing behavior around the singular point $(0,|a_k|)$ of the Lagrangian fibration is basically the same as the basic example of self-mirror space. 

\subsubsection{Semi-global analytic embedding maps $g_k$}

We introduce
\begin{equation}
	\label{Y*_local_affine_toric_chart_eq}
	Y_{\sigma_k}^*:= Y_{\sigma_k}\setminus\{y=0\}
	=\{ zw=1+y \text{ in }  \Lambda^2_{z,w}\times \Lambda^*_y  \}
\end{equation}
Here we omit the subscripts and write $z=z_k$ and $w=w_k$ for clarity.
Recall $y=zw-1$ by (\ref{toric_degeneration_eq}).

Let's put the two adjacent domains $T_{k}$ and $T_{k+1}$ together and imitate the embedding formula in \cite[Section 5.4]{Yuan_local_SYZ} to the analytic open domain 
\[
T_{k,k+1}:=T_{k}\cup T_{k+1} /\sim
\]
in $X_0^\vee$ (\ref{identification_mirror_eq}) for $0\leqslant k\leqslant n$.
Specifically, we define
\begin{align*}
g_k^+ : T_{k+1}  \to  \Lambda^2_{z,w}  \qquad (y_1, y_2) &\mapsto ( y_2^{-1}, y_2(1+y_1) )  \\
g_k^- : T_{k} \ \ \ \to \Lambda^2_{z,w} \qquad (y_1,y_2)&\mapsto ( y_2^{-1}(1+y_1), y_2)
\end{align*}
They are compatible with $\Phi_{k,k+1}$ by (\ref{Phi_k_k+1_eq}) and glue to a \textit{semi-global} analytic embedding
\begin{equation}
	\label{g_k_z_k_w_k_eq}
g_k=(z_k,w_k) : T_{k,k+1}   \to  Y_{\sigma_k}^*
\end{equation}
Here we abuse the notations and still use $z_k$ and $w_k$ as in (\ref{toric_geometry_z_k_w_k_eq}).
For the sake of completeness, we will briefly explain why $g_k$ is injective, which is actually quite elementary to verify (cf. \cite{Yuan_local_SYZ,Yuan_conifold}). 

Suppose $g_k(y)=g_k(y')$, and we aim to show $y = y'$. First, let's assume $y=(y_1,y_2)\in T_{k}$ and $y'=(y_1',y_2')\in T_{k+1}$.
Then, we have $y_1=y_1'$ and $y_2=y'_2(1+y'_1)$.
Next, we set $q=(s,r)=\pi_0^\vee(y)=\chi_{k}^{-1}\circ \trop (y) \in U_{k}$ and $q'=(s', r')=\pi_0^\vee(y')=\chi_{k+1}^{-1}\circ \trop (y')\in U_{k+1}$.
Using (\ref{identification_dual_affinoid_torus_fibration__eq}) implies that $s=s'$, $\val(y_2)=\psi(s,r)-k\min\{0,s\}$, and $\val(y_2')=\psi(s,r')-(k+1)\min\{0,s\}$.
It follows that $\min\{0,s\}+\psi(s,r)-\psi(s,r') =\val(y_2)-\val(y'_2)=\val(1+y_1)\geqslant \min\{0, s\}$ and that $\psi(s,r)\geqslant \psi(s,r')$. By Proposition \ref{increasing_psi_prop}, we have $r\geqslant r'$.
Since most part of $U_{k}$ is `below' $U_{k+1}$ (cf. Figure \ref{figure_milnor_fiber} and (\ref{U_k_eq})), the condition $r\geqslant r'$ can happen only if $q$ and $q'$ are contained in $U_{k}\cap U_{k+1}=\mathscr N_{k+}\cup\mathscr N_{k-}$. In particular, $s\neq 0$ and $y,y'$ are in the same domain $T_{k}$ or $T_{k+1}$ via the gluing. Since $g_k^+$ and $g_k^-$ are clearly injective, we can finally check that $y=y'$.

Keeping in mind that our goal is to glue the various $g_k$ into a global embedding, we proceed by comparing $g_k$ with $g_{k+1}$.
The overlap of their domains $T_{k,k+1}$ and $T_{k+1,k+2}$ is precisely $T_{k+1}$. Recall that $Y_{v_{k+1}}\equiv \Lambda_{z_k}^*\times \Lambda_{w_k}\equiv \Lambda_{z_{k+1}}\times \Lambda_{w_{k+1}}^*$, and we introduce
\begin{equation}
	\label{Y*_overlap_v_k+1}
Y^*_{v_{k+1}} = Y_{v_{k+1}}\setminus\{y=0\}
\end{equation}
By (\ref{g_k_z_k_w_k_eq}), a useful observation is that the image of $g_k|_{T_{k+1}}\equiv g_k^+$ is contained in $Y^*_{v_{k+1}}$ and that the image of $g_{k+1}|_{T_{k+1}}\equiv g_{k+1}^-$ is contained in $Y^*_{v_{k+1}}$ as well.
By (\ref{g_k_z_k_w_k_eq}), we also know that $z_k|_{T_{k+1}}=\frac{1}{y_2}$, $w_k|_{T_{k+1}}=y_2(1+y_1)$, $z_{k+1}|_{T_{k+1}}=\frac{1+y_1}{y_2}$, and $w_{k+1}|_{T_{k+1}}=y_2$ for any $(y_1,y_2)\in T_{k+1}$.
Accordingly, they are subject to the following relation:
\begin{equation}
	\label{toric_gluing_relation}
	w_{k+1}=z_k^{-1} \qquad \text{and} \qquad z_{k+1}w_{k+1}=z_kw_k=1+y_1
\end{equation}
This is exactly the gluing relation in the toric surface $Y_\Sigma$.

\subsubsection{Global analytic embedding map $g$}

Our final goal in this subsection is to show that the various $g_k: T_{k,k+1}\to Y_{\sigma_k}^*$ in (\ref{g_k_z_k_w_k_eq}) glue to a global analytic embedding map
	\begin{equation}
		\label{g_embed_map_eq}
		g: X_0^\vee \to  (Y_\Sigma^*)^{\mathrm{an}}
	\end{equation}
into the analytification of the variety
\begin{equation}
	\label{Y*_Sigma_global_eq}
	Y_\Sigma^*=Y_\Sigma \setminus\{y=0\}
\end{equation}
Remark that the $Y_\Sigma^*$ is also the union of $Y_{\sigma_k}^*$'s in (\ref{Y*_local_affine_toric_chart_eq}) along the overlaps $Y_{v_{k+1}}^*$'s in (\ref{Y*_overlap_v_k+1}).

\begin{convention}
	From now on, we will not always distinguish $Y_\Sigma^{\mathrm{an}}$ from $Y_\Sigma$, and $(Y_\Sigma^*)^{\mathrm{an}}$ from $Y_\Sigma^*$. We often use them interchangeably if the context is clear.
Additionally, we will usually call the various $g_k$'s \textit{semi-global} analytic embeddings and call the $g$ a \textit{global} analytic embedding.
\end{convention}

Confirming (\ref{g_embed_map_eq}) is a straightforward process. Indeed, we first observe that, based on the previous discussion, it is evident that setting $g|_{T_{k,k+1}}=g_k$ results in a well-defined analytic morphism $g$. Additionally, this morphism is locally an isomorphism onto its image.

It remains to show that the map $g$ defined in this way is injective.
Namely, suppose $\pmb x=g(y)=g(y')$ for some $y,y'\in X_0^\vee$; we aim to show that $y=y'$.
	Under the identifications (\ref{identification_mirror_eq}) and (\ref{identification_dual_affinoid_torus_fibration__eq}), we may assume $y=(y_1,y_2)\in T_k$ and $y' =(y'_1,y'_2) \in T_\ell$ for $0\leqslant k\leqslant  \ell \leqslant n+1$, and we set $q=(s,r)=\pi_0^\vee(y)=\chi_k^{-1}\circ \trop (y)\in U_k$ and $q'=(s',r') = \pi_0^\vee(y') = \chi_\ell^{-1}\circ \trop(y') \in U_\ell$.
	Since $1+y_1=z_k(y)w_k(y)=z_\ell(y')w_\ell(y')=1+y_1'$ (cf. (\ref{toric_gluing_relation})), we know $y_1=y_1'$.
	By (\ref{identification_dual_affinoid_torus_fibration__eq}), we know $s=s'$, \[
	\val(y_2)=\psi(s,r)-k\min\{0,s\}, \quad \text{and} \quad \val(y'_2)=\psi(s, r') -\ell \min\{0,s\} \  .  \]
	
	 If $\ell-k=0$ or $1$, there is nothing new compared to (\ref{g_k_z_k_w_k_eq}). If $\ell-k\geqslant 2$, the argument is similar. Indeed, $\sigma_k\cap \sigma_\ell$ is then the trivial cone that consists of the single point. Then, $\pmb x$ is contained in $Y_{\sigma_k}\cap Y_{\sigma_\ell}\equiv Y_{\sigma_k\cap \sigma_\ell}\cong (\Lambda^*)^2$, the open dense complex torus in $Y_\Sigma$. 
		By (\ref{toric_gluing_relation}), we may use the coordinate system $(z_k, w_k)$ uniformly. First, since $y\in T_k$, we clearly have $z_k(y)=y_2^{-1}$ and $w_k(y)=y_2(1+y_1)$.
		Second, for $y'\in T_\ell$, we can still use (\ref{g_k_z_k_w_k_eq}) and (\ref{toric_gluing_relation}) to compute $z_k(y') =(1+y_1)^{k-\ell} z_\ell(y')=(1+y_1)^{k-\ell} (y_2')^{-1}$ and $w_k(y')=(1+y_1)^{\ell-k} w_\ell(y')= (1+y_1)^{\ell-k+1} y_2'$.
		Now, the condition $g(y)=g(y')$ means that $z_k(y)=z_k(y')$ and $w_k(y)=w_k(y')$. Thus, $y_2=(1+y_1)^{\ell-k} y_2'$ and $\val(y_2) =(\ell-k) \val(1+y_1) + \val(y_2')$. It follows that
		\[
		\psi(s,r)-\psi(s,r')=(\ell-k) \left(\val(1+y_1) -\min\{0,s\} \right)  \geqslant 0  \  .
		\]
		From Proposition \ref{increasing_psi_prop} it follows that $r\geqslant r'$.
		However, $\ell-k\geqslant 2$ implies that $U_k$ and $U_\ell$ are disjoint and particularly $r'>r$ (cf. (\ref{U_k_eq}) and Figure \ref{figure_milnor_fiber}). This is a contradiction. Hence, $g$ is indeed injective and we have justified (\ref{g_embed_map_eq}).

\section{Explicit representation of mirror dual fibration}

\subsection{Strategy}
\label{ss_outline_F_k_F}

\subsubsection{Brief purpose}
Our objective is to \textit{explicitly} construct a tropically continuous map from the image of $g$ in $Y_\Sigma^*$ to $B$, such that the singular locus precisely coincides with $\Delta$. As a somewhat vague analogy, we note that one can embed a manifold into a Euclidean space to make it explicit, but there are typically multiple ways to do so. We encounter a similar situation here as we use explicit choices to create an explicit representation of $\pi_0^\vee$. Naturally, such an explicit model may not be unique. The main advantage of making everything explicit is that it allows us to clearly \textit{observe} the appearance of singular mirror fibers.

Since $g$ is obtained by gluing the analytic embedding maps $g_k$'s in the previous section in (\ref{g_embed_map_eq}), we begin with these corresponding semi-global pieces.
Inspired by the local SYZ model in \cite{Yuan_local_SYZ}, we aim to realize a commutative diagram in the following form:
\[
\xymatrix{
	T_{k,k+1} \ar[rr]^{g_k}\ar[d]_{\pi_0^\vee} & & Y_{\sigma_k}^*  \ar[d]_{F_k} \\
	U_{k,k+1} \ar[rr]^{j_k} & & \mathbb R^N
}
\]
where $U_{k,k+1} := U_{k}\cup U_{k+1} \subseteq B_0$, the $F_k$ is some tropically continuous map (see \S \ref{sss_tropically_continuous_map_defn} for the definition), the $j_k$ is a topological embedding, and the $N$ is some large integer.
By (\ref{toric_gluing_relation}), the $y$ in (\ref{toric_degeneration_eq}) or (\ref{Y*_local_affine_toric_chart_eq}) agrees with $y_1$ in the coordinates of $T_{k,k+1}$ and actually defines a non-vanishing global function on $Y_\Sigma^*$

Furthermore, 
we recall that $X_0^\vee$ is the union of $T_{k,k+1}$'s by (\ref{identification_mirror_eq}). Then, we would like to glue the semi-global choices of $(j_k,F_k)$ to obtain a global topological embedding $j$ and a tropically continuous map $F$ that fit into the following commutative diagram:
\[
\xymatrix{
	X_0^\vee \ar[rr]^g \ar[d]_{\pi_0^\vee} & & Y_\Sigma^* \ar[d]_{F} \\
	B_0 \ar[rr]^{j} & & \mathbb R^N 
}
\]

This is an outline of what we want to do in the next.
All of $j_k$, $F_k$, $j$, $F$, and $N$ are to be determined, and an appropriate choice of them gives an explicit representation of $\pi_0^\vee$.
Specifically, we aim to carefully design $F_k$'s such that $F_k\circ g_k|_{T_{k+1}}=F_{k+1}\circ g_{k+1}|_{T_{k+1}}$.

\subsubsection{Motivation and thought process}
\label{sss_thought_process}
Our aim is to provide a glimpse into the thought process behind the discovery of the desired formula. However, if the reader prefers not to delve into the underlying thought process, they may directly refer to \S \ref{ss_order_statistics} and \S \ref{ss_tropically_continuous_map_explicit} and accept the formulae as presented, even though they might seem somewhat unmotivated.

We begin our exploration with several preliminary attempts, followed by the introduction of various ideas for crucial modifications.
It is worth noting that an affine toric chart $Y_{\sigma_k}^*$ closely resembles a local model found in \cite[Section 5.1]{Yuan_local_SYZ}.
Additionally, drawing inspiration from (\ref{identification_dual_affinoid_torus_fibration__eq}), we undertake an initial naive attempt at constructing the following:
\begin{equation*}
	F_k^{naive}=( F^{naive}_{0k}, F^{naive}_{1k}) :  Y_{\sigma_k}^*  \to \mathbb R^2
\end{equation*}
where we set
\begin{align*}
	F^{naive}_{0k}(z_k,w_k,y) &= \min \Big\{ \ \val(z_k)-(k+1)\min\{0,\val(y)\} \  ,   \   -\psi (\val(y), |a_k|)  \Big \}   \\
	F^{naive}_{1k}(z_k,w_k,y) &= \min \Big \{ \  \val(w_k)+k\min\{0,\val(y)\}  \   ,  \  \qquad  \psi ( \val(y),  |a_k|)\} \  \Big\} 
\end{align*}

Recall that $Y_{\sigma_k}^*=Y_{\sigma_k}\setminus\{y=0\}$ is identified with $z_kw_k=1+y$ in $\Lambda^2_{z,w}\times \Lambda_y^*$ by (\ref{Y*_local_affine_toric_chart_eq}).
Here requiring $y\neq 0$ ensures that $\val(y)\neq \infty$, so the term $\psi(\val(y),|a_k|)$ makes sense.

In view of (\ref{identification_dual_affinoid_torus_fibration__eq}), the previous two terms $k\min\{0,\val(y) \}$ and $(k+1)\min\{0,\val(y)\}$ serve to normalize the non-archimedean valuations, aligning them with the function $\psi$ in (\ref{psi_over_B_eq}) that reflects the reduced K\"ahler geometry on the A-side.
Keep in mind that $g_k$ maps into $Y_{\sigma_k}^*$, and our aim is to modify $F_k^{naive}$ so that it is somewhat supported on the image of $g_k$.

To successfully assemble the various $F_k^{naive}$'s, we must introduce additional modifications due to the rigidity of the analytic structure.
Our subsequent attempt involves constructing:
\begin{equation}
	\label{F_k_pre}
	\begin{aligned}
		F_{0k}^{pre}(z_k,w_k,y) &= \Big\{ \val(z_k)-(k+1) \min\{0,\val(y)\},  \quad  \{-\psi(\val(y),|a_j|): 0\leqslant j \leqslant n\}  \Big\}_{[n-k]}  \\
		F_{1k}^{pre}(z_k,w_k,y) &=  \Big\{  \val(w_k)+ k \min\{0,\val(y)\},  \qquad  \{\psi(\val(y),|a_j|): 0\leqslant j \leqslant n\}  \Big\}_{[k]} 
	\end{aligned}
\end{equation}
where we use the notion of order statistics in \S \ref{ss_order_statistics}.
The idea is that the order statistics limit the range of the valuations and render them constant outside the range. In simpler terms, each order statistic is a piecewise linear function locally modeled on $\min$ / $\max$ (cf. Figure \ref{figure_order_statistic}), and the above $F_k^{pre}$ effectively recovers the previous $F_k^{naive}$ in certain local regions. Intuitively, this should be related to the non-archimedean version of the partition of unity, as investigated by Chambert-Loir and Ducros in \cite{Formes_Chambert_2012}.
For instance, let's consider the affine line with variable $x$; there is a `bump function' defined by $x\mapsto \mathrm{median}\{\val(x),0,1\}\equiv \{\val(x),0,1\}_{[1]}$ such that it is equal to $0$ in the region $\{\val(x)\le 0\}$ and equal to $1$ in the region $\{\val(x)\ge 1\}$.
The assignment $(F^{pre}_{0k},F^{pre}_{1k})$ provides a tropically continuous map from $Y_{\sigma_k}^*$ to $\mathbb R^2$. Nonetheless, we still need to extend its domain and perform further modifications if necessary.

On the other hand, when comparing $(F_{0k}^{pre}, F_{1k}^{pre})$ to $(F_{0,k+1}^{pre}, F_{1,k+1}^{pre})$, it becomes evident that their relationship closely resembles the `tropicalization' of the toric gluing relation (\ref{toric_gluing_relation}). The toric geometry over a non-archimedean field offers the advantage of allowing for tropicalization methods (see \cite{tropical_introduction,tropical_Mikhalkin_Rau}). One of the simplest toric varieties is the affine line $\mathbb A^1_\Lambda\equiv \Lambda$, with its tropicalization corresponding to $\overline {\mathbb R} :=\mathbb R\cup \{\infty\} =\{\val(x)\mid x\in \Lambda\}$. Similarly, the tropicalization of $\Lambda^*$ corresponds to $\mathbb R$.
Expanding on this concept, we can associate a tropical toric variety with a fan in the same manner as in toric geometry. However, instead of using $\Lambda$ and $\Lambda^*$, we substitute them with $\overline {\mathbb R}$ and $\mathbb R$, respectively.

Let $\Sigma$ be the smooth fan we previously considered in \S \ref{ss_toric_geometry}.
Then we can define the \textit{tropical smooth toric variety} $Y_\Sigma^{\trop}$ as follows.
We associate to each cone $\sigma$ the space 
$Y_{\sigma}^{\trop} = \mathrm{Hom} (\sigma^\vee \cap M , \overline {\mathbb R})$
of semigroup homomorphisms.
We can similarly identify 
$
Y_{\sigma_k}^{\trop} \equiv  \overline {\mathbb R}_{b_k, c_k}^2 $
with the `tropicalized' toric gluing relations
$c_{k+1}=-b_k$ and $b_{k+1}+c_{k+1}=b_k+c_k $.
In our case, using the property (\ref{order_statistic_reverse_sign}) of the order statistic yields that $F_{0k}^{pre}$ mostly coincides with $-F_{1,k+1}^{pre}$, reflecting the first relation, whereas the second toric gluing relation appears to be somewhat distorted.
The homogeneous coordinates for toric geometry also tropicalizes naturally.
Similar to (\ref{homogeneous_coordinates_z_k_w_k_eq}), the coordinates $(b_k,c_k)$ on $Y^\trop_{\sigma_k}\equiv \overline{\mathbb R}^2$ can be transferred to homogeneous ones, say $[r_0:\cdots r_{n+1}]$, by
$b_k=\sum_{j=0}^{n+1} (k+1-j) r_j $ and $c_k=\sum_{j=0}^{n+1} (j-k) r_j$.
However, the relation (\ref{homogeneous_coordinates_1+y+eq}) is also distorted and cannot be tropicalized here.

\subsection{Order statistics}
\label{ss_order_statistics}
We introduce the notion of order statistics which serves as the generalizations of minimum and maximum values.

Fix two integers $1\leqslant d\leqslant m$.
The $d$-th {\textit{order statistic}} of a sample of $m$ real numbers $x_1,\dots, x_m$ is equal to its $d$-th smallest value (cf. Figure \ref{figure_order_statistic}).
For instance, the first order statistic is the minimum of the sample, that is, $\min\{x_1,\dots, x_m\}$. Similarly, the $m$-th order statistic is the maximum of the sample, that is, $\max\{x_1,\dots, x_m\}$.
Given $n+1$ real numbers $x_0, x_1,\dots, x_n$ and any $0\leqslant k\leqslant n$, we denote the \textbf{$(k+1)$-th} order statistic by (not the $k$-th one; apologies for the potential confusions in notation, but it makes many formulae cleaner):
\[
\{x_0,\dots, x_n\}_{[k]}
\]

Going back to our context, we abbreviate
\begin{equation}
	\label{psi_k_abbrev_notation}
	\psi_k=\psi_k(s)=\psi (s,|a_k|)
\end{equation}
Due to Proposition \ref{increasing_psi_prop}, they are subject to the condition
$
\psi_0< \psi_1 <\cdots < \psi_n
$. They give rise to a partition of the real line into $n+2$ intervals $[-\infty,\psi_0], [\psi_0,\psi_1],\dots, [\psi_{n-1},\psi_n],$ and $[\psi_n,+\infty]$.

For any $c\in\mathbb R\cup \{\pm\infty\}$, it is straightforward to check the following properties:
\begin{equation}
	\label{order_statistics_computation}
\begin{aligned}
\{c,\psi_0,\psi_1,\dots, \psi_n\}_{[0]}
&
=
\begin{cases}
	c & \text{if  }   c\in[-\infty , \psi_0] \\
	\psi_0 & \text{if  }  c\in [\psi_0 , +\infty ] \\
\end{cases} \\
\{c,\psi_0,\psi_1,\dots, \psi_n\}_{[k]}
&
=
\begin{cases}
	\psi_{k-1} & \text{if  }   c\in [-\infty , \psi_{k-1}] \\
	c & \text{if  }   c\in [\psi_{k-1} , \psi_k] \\
	\psi_k & \text{if  }  c\in[\psi_k  , +\infty] \\
\end{cases}
\qquad (1 \leqslant k \leqslant n) 
\\
\{c,\psi_0,\psi_1,\dots, \psi_n\}_{[n+1]}
&
=
\begin{cases}
	\psi_{n} & \text{if  }   c\in [-\infty , \psi_{n}] \\
	c & \text{if  }   c\in [\psi_{n} , +\infty] 
\end{cases}
\end{aligned} 
\end{equation}

Moreover, it is also direct to verify that for $0\leqslant k \leqslant n+1$,
\begin{equation}
	\label{order_statistic_reverse_sign}
\{ -c,-\psi_0,-\psi_1,\dots, -\psi_n\}_{[n+1-k]}=\{c,\psi_0,\psi_1,\dots, \psi_n\}_{[k]}
\end{equation}


We introduce
\begin{equation}
	\label{gamma_k}
	\begin{aligned}
		\gamma_k(s,c)
		&=	
		\{c, \psi_0,\psi_1,\dots, \psi_n \}_{[k]}   \qquad (0\leqslant k \leqslant n+1)
	\end{aligned}
\end{equation}
See Figure \ref{figure_order_statistic} for an illustration of the graphs of $\gamma_k$'s.

\begin{rmk}
	The motivation behind the construction of $\gamma_k(s,c)$ will become clear once we delve into the study of tropically continuous maps later. In fact, during the process of finding explicit solutions, we initially construct tropically continuous maps and ensure their compatibility. The lines $\gamma_k(s,c)$ are discovered along the way. However, to provide a clearer presentation in the paper, we first present the formulas for $\gamma_k(s,c)$ without elaborating on the reasons behind these specific formulas. The justification and intuition for these constructions will be revealed in subsequent sections as we further explore the properties and compatibility of the tropically continuous maps.
\end{rmk}

Finally, we set
\begin{equation}
	\label{pmb_gamma_eq}
	\pmb \gamma: \mathbb R^2\to\mathbb R^{n+3} ,\qquad
	\pmb \gamma(s,c) = (\gamma_0(s,c) ,\gamma_1(s,c) ,\dots, \gamma_{n+1}(s,c) , \  s )
\end{equation}

By construction, $\pmb \gamma$ gives a homeomorphism from $\mathbb R^2$ onto its image. Namely, $\gamma$ defines an embedded surface in $\mathbb R^{n+3}$.
In reality, the above computation also suggests that for a fixed $s$, the curve $c\mapsto \pmb\gamma(s,c)$ is a broken line in $\mathbb R^{n+2}\times \{s\}$ with $n+2$ line segments corresponding to the $n+2$ intervals $(-\infty,\psi_0], [\psi_0,\psi_1],\dots, [\psi_{n-1},\psi_n], [\psi_n,+\infty)$ in the $c$-domain. Namely,
\begin{equation}
	\label{line_segments}
\pmb \gamma(s,c)
=
\begin{cases}
(c,\psi_0,\psi_1,\dots, \psi_n  \ , \  s)   & \text{if  } c\in (-\infty , \psi_0] \\
(\psi_0,\dots, \psi_{k-1}, c, \psi_k, \dots, \psi_n  \ , \ s ) & \text{if  } c\in [\psi_{k-1},\psi_k]   \qquad  (1\leqslant k\leqslant n) \\
(\psi_0, \psi_1,\dots, \psi_n , c  \ , \ s) & \text{if  } c\in [\psi_n,+\infty)
\end{cases}
\end{equation}
The $n+1$ corner points at $c=\psi_0, \psi_1, \dots, \psi_n$, where the $n+2$ line segments meet, are respectively provided by:
\begin{align}
	\label{corner_point_gamma}
A_{k}=A_{k}(s)=&\Big(
\psi_0, \psi_1,\dots, \psi_{k-1}, \psi_k, \psi_k , \psi_{k+1} \dots, \psi_n 
\Big) &&  0\leqslant k \leqslant n 
\end{align}

Denote by the $n+2$ line segments of $c\mapsto \pmb \gamma(s,c)$ by $\pmb\gamma_-(s,c),\pmb \gamma_{0,1}(s,c),\dots, \pmb\gamma_{n-1,n}(s,c),\pmb\gamma_+(s,c)$ respectively. In other words, given a fixed $s$, we assume

\begin{itemize}
\item $\pmb\gamma_-$ ($c\leqslant \psi_0$) is the first line segment that shots into $A_0$.
	\item  $\pmb\gamma_{k,k+1}$ ($\psi_k \leqslant c\leqslant \psi_{k+1}$) is the $(k+1)$-th line segment connecting $A_k$ and $A_{k+1}$ for $0\leqslant k \leqslant n-1$.
\item $\pmb\gamma_+$ ($c\geqslant \psi_n$) is the last line segment that emanates from $A_n$.
\end{itemize}

\subsection{Explicit tropically continuous maps}
\label{ss_tropically_continuous_map_explicit}

\subsubsection{Tropically continuous map in homogeneous coordinates}
\label{sss_trop_cont}
After revisiting our thought process in \S \ref{sss_thought_process} and undergoing numerous iterations, we now present the final construction as follows:

First, we observe that the tropicalization of the homogeneous coordinates' equations in (\ref{homogeneous_coordinates_z_k_w_k_eq}) results in:
\begin{equation}
	\label{val_z_k_w_k_homogeneous}
	\val(z_k)=\sum_{j=0}^{n+1} (k+1-j) \val(x_j) \qquad \text{and} \qquad 
	\val(w_k)= \sum_{j=0}^{n+1} (j-k) \val(x_j)
\end{equation}
Based on (\ref{homogeneous_coordinates_1+y+eq}), the following relationship can be also deduced:
\begin{equation}
	\label{val_1+y_homogeneous}
	\val(1+y)= \sum_{j=0}^{n+1} \val(x_j)
\end{equation}
It is instructive to notice that for any $0\leqslant k \leqslant n$, 
\[
-\val(z_k) + (k+1)\val(1+y)=\val(w_k)+k\val(1+y) =\sum_{j=0}^{n+1} \ j\cdot \val(x_j)  \  .
\]

Here we utilize the homogeneous coordinates of $Y_\Sigma$ since we would like a more global treatment.
By (\ref{homogeneous_coordinates_identification}) and (\ref{Y*_Sigma_global_eq}), we can identify
$
Y^*_{\Sigma} \equiv Y_\Sigma\setminus\{y=0\}
$
with the subvariety in 
\begin{equation}
	\label{homogeneous_coordinates_X_y_eq}
Y_\Sigma\times \Lambda^*_y\equiv \big(\Lambda^{n+2}\setminus Z(\Sigma)\big)/ G \times \Lambda_y^*
\end{equation}
defined by the equation 
\[
\prod_{j=0}^{n+1} x_j = 1+y
\]
where $[x_0:\cdots x_{n+1}]$ is the homogeneous coordinate in $Y_\Sigma=\big(\Lambda^{n+2}\setminus Z(\Sigma)\big)/ G$.
Finally, inspired by (\ref{F_k_pre}) and (\ref{order_statistic_reverse_sign}), we discover the following construction.

For any $0\leqslant k \leqslant n+1$, we define
\begin{equation}
	\label{F_k_defn_eq}
	F_k=\left\{
	\sum_{j=0}^{n+1} \  (j-k) \cdot \val(x_j)  +k\min\{0,\val(y)\}, 
	\quad \  \  \  \{\psi(\val(y),|a_j|): 0\leqslant j \leqslant n\} 
	\right\}_{[k]}
\end{equation}

\begin{rmk}
It can be hard to articulate precisely why we define $F_k$ in this specific manner as opposed to another. We hope that the discussion in \S \ref{ss_outline_F_k_F} will offer some clues. However, to truly pinpoint the appropriate formula, intensive trials and computations are unavoidable. We sincerely apologize the insufficiency of motivation here.
\end{rmk}

Integrating all components, we define
\begin{equation}
	\label{F_defn_eq}
	F= \big(F_0,F_1,\dots, F_{n+1}, \val(y) \big) :  Y_\Sigma^*  \to \mathbb R^{n+3}
\end{equation}

\begin{lem}
	\label{F_global_image_lem}
	The image of $F$ is given by the surface $\pmb\gamma$ in (\ref{pmb_gamma_eq}).
\end{lem}

\begin{proof}
	We are going to describe the image of $F$ by looking at its image restricted on $\val(y)=s$ for every fixed $s$.
	We write an input point as $\pmb x=([x_0:\cdots x_{n+1}], y)$ with respect to the homogeneous coordinates.
	First, we introduce a variable
	\begin{equation}
		\label{c_variable}
		c:= \sum_{j=0}^{n+1} \  j \cdot \val(x_j) \quad  \in  (-\infty, \infty] =\mathbb R\cup \{\infty\} 
	\end{equation}
	It is routine by (\ref{G_group_eq}) to check that for any $t=(t_0,\dots, t_{n+1})\in G$, replacing $(x_j)$ by $(t_jx_j)$ will not affect the value of $c$.
	Besides, for clarity, we set
	$
	\psi_k=\psi_k(s)=\psi(s,|a_k|)
	$ as in (\ref{psi_k_abbrev_notation}). Recall that by Proposition \ref{increasing_psi_prop}, $\psi_0<\psi_1<\cdots <\psi_n$.

	\vspace{0.7em}
	
	\noindent
	\textbf{Case 1: $s\neq 0$ . }

	Since $\val(1+y)=\min\{0,\val(y)\}=\min\{0,s\}\neq \infty$, it follows from (\ref{val_1+y_homogeneous}) that $\val(x_j)\neq \infty$ for all $j$ and thus $c\neq \infty$. Moreover,
	\[
	\sum_{j=0}^{n+1} (j-k)\val(x_j)+k\min\{0,\val(y)\} =c
	\]
	and therefore $F_k=\gamma_k(s,c)$.
	In summary, for every $s\neq 0$, the image of $F$ on $\val(y)=s$ precisely matches the image of $\pmb \gamma(s,\cdot)$ in (\ref{pmb_gamma_eq}).


	\vspace{0.7em}
	
	\noindent
	\textbf{Case 2: $s= 0$ and $y\neq -1$. }
	
	Here we exclude the infinite non-archimedean valuation and avoid another layer of complexity by assuming $y\neq -1$.
	We introduce a new variable
	\[
	\tau:=\val(1+y) =\sum_{j=0}^{n+1}\val(x_j)  \in [0, \infty)
	\]
	by virtue of (\ref{val_1+y_homogeneous}).
	Indeed, $\tau=\val(1+y)\geqslant \min\{0, s\}=0$ since $\val(y)=s=0$, and $\tau\neq \infty$ since we have assumed $1+y\neq 0$.
	Similarly, we have $c\neq \infty$.
	For a fixed $\tau$, we have the following increasing sequence of real numbers:
	\begin{equation} 
		\label{sequence_numbers_eq}
		\begin{matrix}
			& \  \psi_0  &\leqslant & \tau +\psi_0 \\
			<&  \ \tau +\psi_1 &\leqslant  & 2\tau +\psi_1 \\
			&  \vdots  & &  \vdots \\
			<& \ (k-1)\tau + \psi_{k-1} &\leqslant & k\tau + \psi_{k-1} \\
			<& \ k\tau + \psi_k &\leqslant & (k+1)\tau + \psi_k \\
			<& \ (k+1)\tau + \psi_{k+1} &\leqslant & (k+2)\tau + \psi_{k+1} \\
			&  \vdots & &   \vdots \\
			<& \ (n-1)\tau + \psi_{n-1} &\leqslant & n\tau +\psi_{n-1} \\
			<& \  n\tau+\psi_n &\leqslant & (n+1)\tau + \psi_n
		\end{matrix}
	\end{equation}
	In particular, when $\tau > 0$, all the aforementioned non-strict inequalities transform into strict ones.
	Furthermore, this sequence of ascending real numbers divides the real line into a set of intervals (where we choose to include the end-points):
	\begin{align*}
		&	\mathcal I_{-}= (-\infty, \psi_0] \\
		&	\mathcal I_k=[k\tau+\psi_k, (k+1)\tau+\psi_k]  && 0\leqslant k\leqslant n \\
		&	\mathcal I_{k,k+1}=[(k+1)\tau+\psi_k, (k+1)\tau+\psi_{k+1} ] && 0\leqslant k \leqslant n-1 \\
		&	\mathcal I_+=[(n+1)\tau+\psi_n, +\infty)
	\end{align*}
We compute
	\[
	\sum_{j=0}^{n+1} (j-k)\val(x_j)+k\min\{0,\val(y)\} = c-k\tau  \  .
	\]
By the formula (\ref{F_k_defn_eq}) of $F_k$, we obtain
\[
F_{k}=  \Big\{  c-k\tau ,   \psi_0,\dots, \psi_n  \Big\}_{[k]}  =\gamma_k(0,c-k\tau)  \  .
\]
Utilizing (\ref{order_statistics_computation}) can produce the results in the following subcases:

	\vspace{0.5em}
	\noindent
	\underline{\textit{Subcase 2-1:}} 
	$c\in\mathcal I_-$.
	
	We first have $F_0=\gamma_0(0,c)=c$ for all $c\in\mathcal I$. Besides, given each $1\leqslant k\leqslant n+1$, we have $F_k= \gamma_k(0,c-k\tau)=\psi_{k-1}$ since $c\leqslant \psi_0 \leqslant k\tau + \psi_{k-1}$ and thus $c-k\tau \leqslant \psi_{k-1}$. Consequently,
	$F=(\tilde c, \psi_0, \dots \psi_n)=\pmb\gamma(0,\tilde c) $ for $\tilde c=c\in (-\infty, \psi_0]$.

	\vspace{0.5em}
	\noindent
	\underline{\textit{Subcase 2-2:}} 
	$c\in\mathcal I_{k_0}$ for some $0\leqslant k_0\leqslant n$.
	
For $0\leqslant k\leqslant k_0$, we have $c-k\tau \geqslant c-k_0\tau\geqslant \psi_{k_0} \geqslant \psi_k$, hence, $\gamma_k(0,c-k\tau)=\psi_k$.
Moreover, for $k_0+1\leqslant k\leqslant n$, we have $c-k\tau \leqslant c-(k_0+1)\tau \leqslant \psi_{k_0} \leqslant \psi_{k-1}$, hence, $\gamma_k(0 ,c-k\tau) =\psi_{k-1}$. To sum up, $F=(\psi_0,\dots, \psi_{k_0}, \psi_{k_0},\psi_{k_0+1},\dots, \psi_n)=A_{k_0}(0)$ is the $k_0$-th corner point (\ref{corner_point_gamma}).

	\vspace{0.5em}
	\noindent
	\underline{\textit{Subcase 2-3:}} 
	$c\in\mathcal I_{k_0,k_0+1}$ for some $0\leqslant k_0\leqslant n-1$.
	
	For $0\leqslant  k\leqslant  k_0$, we get $c-k\tau \geqslant c-k_0\tau > c-(k_0+1)\tau \geqslant \psi_{k_0} \geqslant \psi_k$, therefore, $F_k=\gamma_k(s,c-k\tau)=\psi_k$.
	Besides, for $k_0+2\leqslant k\leqslant n$, we get $c-k\tau \leqslant c-(k_0+2)\tau < c-(k_0+1)\tau \leqslant \psi_{k_0+1} \leqslant \psi_{k-1}$, therefore, $F_k=\gamma_k(0,c-k\tau) =\psi_{k-1}$.
	Finally, for $k=k_0+1$, we note that $\tilde c:=c-k\tau=c-(k_0+1)\tau \in [\psi_{k_0},\psi_{k_0+1}]$ and so $F_{k_0+1}=\gamma_{k_0+1}(0 ,c-(k_0+1)\tau) =c-(k_0+1)\tau=\tilde c$.
	To sum up, we obtain $F=(\psi_0,\dots, \psi_{k_0}, \tilde c, \psi_{k_0+1},\dots, \psi_n)$ for $\tilde c\in [\psi_{k_0},\psi_{k_0+1}]$.
	This agrees with the line segment of $c\mapsto \pmb\gamma(0,c)$ that connect the two corner points $A_{k_0}$ and $A_{k_0+1}$ (\ref{corner_point_gamma}).

	\vspace{0.5em}
	\noindent
	\underline{\textit{Subcase 2-4:}} 
	$c\in\mathcal I_+$.
	
	Given any $0\leqslant k \leqslant n$, we have $c-k\tau \geqslant c-n\tau > c-(n+1)\tau \geqslant \psi_n \geqslant \psi_k$, thus, $F_k=\gamma_k(s,c-k\tau)=\psi_k$.
	For $k=n+1$, we have $\tilde c:=c-k\tau=c-(n+1)\tau \geqslant \psi_n $ and $F_{n+1}=\gamma_{n+1}(0,\tilde c)=\tilde c$. To sum up, we obtain $F=(\psi_0,\psi_1,\dots, \psi_n, \tilde c)=\pmb\gamma(0,\tilde c)$ for any $\tilde c\in[\psi_n,+\infty)$.

	\vspace{0.7em}
	\noindent
	\textbf{Case 3: $s=0$ and $y= -1$.}

By (\ref{chi_01=0_eq}), this case corresponds to the points in $\mathcal D_0\cup \mathcal D_1\cup\cdots \cup \mathcal D_{n+1}$, the union of torus-invariant prime divisors. Recall that $\mathcal D_i\cap \mathcal D_j=\varnothing$ for $j-i\ge 2$ and that $\mathcal D_k\cap \mathcal D_{k+1}=\{x_k=x_{k+1}=0\}=O(\sigma_k)$.
Now, we consider the following sub-cases of the input point $\pmb x=([x_0:\cdots: x_{n+1}], y)$ in view of (\ref{divisor_torus_invariant_Bside}).

	\vspace{0.5em}
	\noindent
	\underline{\textit{Subcase 3-1:}} $\pmb x\in O(v_{k_0})$ for $0\leqslant k_0\leqslant n+1$. Namely, $x_{k_0}=0$, and $x_{k_1}\neq 0$ for all other $k_1 \neq k_0$.
	We then have $\val(x_{k_0})=\infty$ and $c=\infty$.
	We compute
	\[
	\sum_{j=0}^{n+1} (j-k)\val(x_j)+k\min\{0,\val(y)\} 
	=
	\begin{cases}
	+\infty	& \text{if  } k \leqslant k_0-1 \\
	\tilde c & \text{if  } k=k_0 \\
	-\infty & \text{if  } k\geqslant k_0+1
	\end{cases}
	\]
	where $\tilde c$ can be arbitrary in $\mathbb R$.
	Then, $F_k=\psi_k$ for $k\leqslant k_0-1$, and $F_k=\psi_{k-1}$ for $k\geqslant k_0+1$. Besides, 
	$
	F_{k_0}=\gamma_{k_0}(0,\tilde c)$ is equal to $\psi_{k_0-1}$, $\tilde c$, or $\psi_k$ when $\tilde c\leqslant \psi_{k_0-1}$, $\psi_{k_0-1}\leqslant c\leqslant \psi_{k_0}$, or $c\geqslant \psi_{k_0}$ respectively. As a result, the image of $F$ in this case coincides with the path $\tilde c\mapsto ( \psi_0,\dots ,\psi_{k_0-1}, \tilde c,  \psi_{k_0},\dots, \psi_n )$ for only $\tilde c\in [\psi_{k_0-1},\psi_{k_0}]$. This is exactly the $(k_0+1)$-th line segment of $\pmb\gamma(0,\cdot)$ between the two corner points $A_{k_0-1}(0)$ and $A_{k_0}(0)$.

	\vspace{0.5em}
	\noindent
	\underline{\textit{Subcase 3-1:}} $\pmb x\in O(\sigma_{k_0})$ for $0\leqslant k_0\leqslant n$. Namely, $x_{k_0}=x_{k_0+1}=0$, and $x_{k_1}\neq 0$ for all other $k_1 \neq k_0, k_0+1$.
	We then compute
	\[
	\sum_{j=0}^{n+1} (j-k)\val(x_j)+k\min\{0,\val(y)\} 
	=
	\begin{cases}
		+\infty	& \text{if  } k \leqslant k_0 \\
		-\infty & \text{if  } k\geqslant k_0+1
	\end{cases}  \  .
	\]
	It follows that $F_k=\psi_k$ for $k\leqslant k_0$ and $F_k=\psi_{k-1}$ for $k\geqslant k_0+1$. Thus, $F=(\psi_0,\dots, \psi_{k_0},\psi_{k_0},\dots, \psi_n)$ is the corner point $A_{k_0}(0)$ (\ref{corner_point_gamma}).
	
Integrating the discussions in Case 2 and Case 3, we conclude that the image of $F$ on $\val(y)=0$ agrees with the image of $\pmb\gamma(0,\cdot)$.
The proof is now complete. 
\end{proof}

\begin{cor}
	\label{F_image_toric_divisor_cor}
	For $0\le k\le n+1$, the $F$-image of the irreducible toric divisor $\mathcal D_k$ is the $(k+1)$-th line segment $c\mapsto \pmb\gamma(0, c)$. In other words, using the notations at the end of \S \ref{ss_order_statistics},
	we have $F(\mathcal D_0)=\pmb \gamma_-$, $F(\mathcal D_{n+1})=\pmb\gamma_+$, and $F(\mathcal D_k)=\pmb\gamma_{k,k+1}$ for $0\leqslant k\leqslant n-1$.
\end{cor}

\begin{proof}
	This is an immediate byproduct of the Case 3 in the proof of Lemma \ref{F_global_image_lem}.
\end{proof}

Remark that every $\mathcal D_k=\{x_k=0\}$ avoid points with $y=0$. Hence, by (\ref{Y*_Sigma_global_eq}), $\mathcal D_k$ is contained in $Y_\Sigma^*=Y_\Sigma\setminus \{y=0\}$, the domain of $F$ (\ref{F_defn_eq}).
Then, the following result is also straightforward.

\begin{cor}
	\label{F_image_torus_open_dense_torus_cor}
	The $F$-image of the open dense toric orbit removing the divisor $y=0$ is the embedded surface $\pmb\gamma$ removing $n+1$ points $A_k(0)$ ($0\leqslant k\leqslant n$).
\end{cor}

We are next going to show that this actually agrees with the smooth locus of $F$ (Lemma \ref{F_global_smooth_locus_lem}).

\subsection{Singular and smooth loci}

The objective here is to study the smooth/singular loci of the tropically continuous map $F$ in the previous section (\ref{F_defn_eq}).

\begin{lem}
	\label{F_global_smooth_locus_lem}
	The singular locus of $F$ consists of the $n+1$ points $A_k(0)$ for $0\leqslant k\leqslant n$.
\end{lem}

\begin{proof}
By Lemma \ref{F_global_image_lem}, we know the image of $F=(F_0,\dots, F_{n+1},\val(y))$ in $\mathbb R^{n+3}$ is given by the embedded surface $(s,c)\mapsto \pmb\gamma(s,c)$ in (\ref{pmb_gamma_eq}).
Recall that for any given $s$, the slice $c\mapsto \pmb\gamma(s,c)$ is a broken line of with $n+2$ segments $\pmb\gamma_-,\pmb\gamma_{0,1},\dots, \pmb\gamma_{n-1,n}, \pmb\gamma_+$ and $n+1$ corner points $A_0(s),A_1(s),\dots, A_n(s)$ successively. Here we use the notations in \S \ref{ss_order_statistics}.

Let $p=\pmb \gamma(s,c)=(p_0,\dots,p_{n+2}, s)$ be an arbitrary point in the embedded surface of $\pmb\gamma$ in $\mathbb R^{n+3}$. 

\vspace{0.7em}
\noindent
\textbf{\textit{Case 1:}} \textit{$s\neq 0$ .}
Let $U$ be a small neighborhood of $p$ in the embedded surface with respect to the subspace topology in $\mathbb R^{n+3}$.
Shrinking $U$ if necessary, we may require that every other point in $U$ still has nonzero $s$-coordinate and that there exist small numbers $\delta,\epsilon>0$ such that $V:=(s-\epsilon, s+\epsilon) \times (c-\delta, c+\delta)$ is homeomorphic to $U$ by sending $(s',c')$ to $\pmb\gamma(s',c')$.
Besides, we define
\[
\Xi: (y_1',y_2')\mapsto \left( \Big[\frac{1+y'_1}{y_2'} : y_2': 1: 1 : \cdots : 1 \Big], y_1' \right)
\]
for any $(y'_1,y'_2)\in \trop^{-1}(V)$. Here we use the homogeneous coordinates on the right-hand side as (\ref{homogeneous_coordinates_X_y_eq}).
Remark that by definition, $1+y'_1\neq 0$, and thus the image of $\Xi$ is at least contained in the open dense torus in the toric variety.
By (\ref{F_k_defn_eq}), we compute 
\begin{align*}
F_k\circ \Xi  (y_1',y_2') 
&=
\left\{
-k\val\left(\frac{1+y_1'}{y_2'}\right) +(1-k) \val(y_2') +k\min\{0,\val(y_1')\}  , \quad  \psi_0(\val(y_1')),\dots, \psi_n(\val(y_1')) 
\right\}_{[k]} \\
&=
\left\{
c' , \quad \psi_0(s'),\dots, \psi_n(s')
\right\}_{[k]} =\gamma_k(s',c')
\end{align*}
where we put $(s',c')=\trop(y_1',y_2')$, namely, $s'=\val(y_1')$ and $c'=\val(y_2')$. Since $s'\neq 0$, we also note that $\val(1+y'_1)=\min\{0,s'\}$.
Accordingly, $F\circ \Xi=\pmb\gamma\circ \trop$, and $F:F^{-1}(U)\to U$ is isomorphic to $\trop: \trop^{-1}(V)\to V$.
Therefore, we conclude that $p$ is $F$-smooth.

\vspace{0.7em}
\noindent
\textbf{\textit{Case 2:}} \textit{$s= 0$ but the value of $c$ is not any of $\psi_k=\psi_k(0)$ for $0\leqslant k\leqslant n$.}

We may assume $c\in (\psi_{k_0-1}, \psi_{k_0})$ for some $0\leqslant k_0\leqslant n+1$. Here we may allow $k_0=0$ or $n+1$ as we temporarily set $\psi_{-1}=-\infty$ and $\psi_{n+1}=+\infty$.
Then, there exists some small number $\delta>0$ such that $(c-\delta,c+\delta)\subseteq (\psi_{k_0-1},\psi_{k_0})$.
Further, since $\psi$ is continuous, we are able to pick a sufficiently small number $\epsilon>0$ such that 
\begin{equation}
	\label{condition_c+-delta_eq}
(c-\delta,c+\delta)\subseteq (\psi_{k_0-1}(s'), \psi_{k_0}(s'))
\end{equation}
for any $-\epsilon \leqslant s'\leqslant \epsilon$.
Now, we put $V:=(-\epsilon,\epsilon)\times (c-\delta,c+\delta)$, and $U:=\pmb\gamma(V)$ gives a neighborhood of $p$.

Inspired by (\ref{homogeneous_coordiantes_affine_chart}), we define
\begin{align*}
\Xi: (y_1',y_2')\mapsto 
&
\left(\Big[ 
1: \cdots : 1 : x'_{k_0-1}=1:  \  \ \ \  x'_{k_0}=\frac{1+y_1'}{y_2'} :  \ \  \  \  \  \  \  
x'_{k_0+1}= y_2' : 1 :\cdots :1
\Big] , y_1'\right) \\
\equiv 
&
\left(\Big[ 
1: \cdots : 1 : x'_{k_0-1}=\frac{1}{y_2'} : \ \ x'_{k_0}= y_2'(1+y_1') :  \  \  x'_{k_0+1}=1  \  :1 :\cdots :1
\Big] , y_1'\right)
\end{align*}
for any $(y_1',y_2')\in \trop^{-1}(V)$ and $0\leqslant k_0\leqslant n$.
The equality holds because $(1, \dots, 1, t^{-1}, t^{2}, t^{-1}, 1 ,\dots, 1)$ for any $t\in \Lambda^*$ is always an element in the group $G$ (\ref{G_group_eq}).

We first claim that $\Xi$ gives the desired analytic isomorphism from $\trop^{-1}(V)$ to $F^{-1}(U)$.
To see this, let $\mathbf x=([x'_0:\cdots :x'_n], y')$ be an arbitrary point in $F^{-1}(U)$.
Then, there exists a unique pair $(s',c')$ in $V$ such that $F(\mathbf x)=\pmb\gamma(s',c')=(\gamma_0(s',c'),\dots, \gamma_{n+1}(s',c'), s')$.
Utilizing both the above condition (\ref{condition_c+-delta_eq}) and the computations in (\ref{line_segments}) concludes that we have
\[
\begin{cases}
	\gamma_k(s',c')= \psi_k(s') & \text{if  } k\leqslant k_0-1 \\
	\gamma_{k_0}(s',c')=c' \\
	\gamma_k(s',c')= \psi_{k-1}(s) & \text{if  } k\geqslant k_0+1
\end{cases}
\]
and that $F(\mathbf x)=(\psi_0(s'),\dots, \psi_{k_0-1}(s'), c', \psi_{k_0}(s'),\dots, \psi_n(s'), s')$.
Together with the formula (\ref{F_k_defn_eq}) and Lemma \ref{F_global_image_lem}, we derive that
\[
\psi_{k_0-1}(\val(y')) 
< 
\sum_{j=0}^{n+1} (j-k_0) \val(x'_j) +k_0\min\{0,\val(y')\}  
<
\psi_{k_0}(\val(y'))  \  .
\]
Hence, $\sum_{j=0}^{n+1}(j-k_0)\val(x_j')$ is a finite real number. For any $j\neq k_0$, we must have $\val(x_j')\neq \infty$, equivalently $x_j'\neq 0$.
Only $x'_{k_0}$ is possibly zero. In particular, we obtain two invertible analytic functions $y_1'$ and $y_2'$ on $F^{-1}(U)$:
\begin{equation}
	\label{proof_steps_y_12'}
	y_1'=y' \qquad \text{and} \qquad
	y_2'=y'_{2,(k_0)}=\prod_{j=0}^{n+1} (x_j')^{j-k_0}
\end{equation}
The latter is always non-vanishing since $\val(y_2')\neq \infty$. It is then direct to check $([x_0':\cdots :x_{n+1}'],y')\mapsto (y_1',y_2')$ gives the desired inverse of $\Xi$.

We next claim that $\Xi$ indeed intertwines the fibrations $F^{-1}(U)\to U$ and $\trop^{-1}(V)\to V$. Namely, it remains to show $F\circ \Xi=\pmb\gamma \circ \trop$.
Let's write $s'=\val(y_1')$, $c'=\val(y_2')$, and
\[
\Xi (y_1',y_2'):= ( \big[\Xi_0  \  : \cdots : \  \Xi_{n+1} \big] (y_1',y_2'), y_1' )
\]
where $[\Xi_0:\cdots :\Xi_{n+1}]$ refers to the homogeneous coordinates.
By the formula (\ref{F_k_defn_eq}) of $F_k$, we compute
\begin{align*}
F_k \circ \Xi(y_1',y_2') 
&
=\left\{
\sum_{j=0}^{n+1}(j-k) \val(\Xi_j) +k\min\{0, s'\} , \quad \psi_0(s'), \dots, \psi_n(s')
\right\}_{[k]} \\
&
=\left\{
\val(y_2') +(k_0-k)\val(1+y'_1) ,   \quad \psi_0(s'), \dots, \psi_n(s')
\right\}_{[k]} \\
&
=\left\{
c'+(k_0-k)\tau'  ,  \quad \psi_0(s'), \dots, \psi_n(s')
\right\}_{[k]} \\
&
=
\gamma_k( s', c'+(k_0-k)\tau')
\end{align*}
where we set $\tau'=\val(1+y'_1)\geqslant 0$.
Recall that by (\ref{condition_c+-delta_eq}), we have $\psi_{k_0-1}(s')<c'<\psi_{k_0}(s')$. We can use the computations in (\ref{order_statistics_computation}) for the order statistics to derive the following:

If $k\leqslant k_0-1$, then $c'+(k_0-k)\tau' \geqslant c' > \psi_{k_0-1}(s') \geqslant \psi_k(s')$ and thus $\gamma_k( s', c'+(k_0-k)\tau')=\psi_k(s')$.

If $k\geqslant k_0+1$, then $c'+(k_0-k)\tau' \leqslant c'<\psi_{k_0}(s')\leqslant \psi_{k-1}(s')$ and thus $\gamma_k( s', c'+(k_0-k)\tau')=\psi_{k-1}(s')$.

If $k=k_0$, then $\gamma_k( s', c'+(k_0-k)\tau')=\gamma_{k_0}(s',c')=c'$.

Therefore,
\[
F\circ \Xi(y_1',y_2')= \big(\psi_0(s'),  \dots  ,  \psi_{k_0-1}(s'),  c',  \psi_{k_0}(s'),\dots, \psi_n(s'), s'\big) =\pmb\gamma(s',c') =\pmb \gamma \circ \trop(y_1',y_2')  \  .
\]

In summary, we conclude that $p$ is $F$-smooth in the current case.

\vspace{0.7em}
\noindent
\textbf{\textit{Case 3:}} \textit{$s= 0$ and the value of $c$ is exactly one of $\psi_k=\psi_k(0)$ for $0\leqslant k\leqslant n$.}

Without loss of generality, we may assume there is some $0\leqslant k_0\leqslant n$ such that $c=\psi_{k_0}(0)$.
Our purpose is to show that $p$ is \textit{not} $F$-smooth in this case.

Arguing by contradiction, suppose $p$ is $F$-smooth.
There must be an induced integral affine structure on a neighborhood $U$ of $p=\pmb\gamma (0, \psi_{k_0})$ (see \S \ref{ss_integrable_system}).
In other words, it makes sense to say \textit{integral affine functions} on $U$.
Shrinking $U$ if necessary, we may assume $U=\pmb\gamma(V)$ where $V=(-\epsilon,\epsilon)\times (\psi_{k_0}-\delta,\psi_{k_0}+\delta)$ for sufficiently small numbers $\epsilon,\delta>0$.

Now, we choose two points $p_+= \pmb \gamma (0, \psi_{k_0}+\delta/2)$ and $p_-=\pmb\gamma(0,\psi_{k_0}-\delta /2)$ in $U$.
In particular, the condition (\ref{condition_c+-delta_eq}) holds near a neighborhood of $p_-$. Hence, the argument in the previous case can be repeated for $p_-$ and concludes by (\ref{proof_steps_y_12'}) that $\sum_{j=0}^{n+1}(j-k_0)\val(x_j)$ gives an integral affine function. Similarly, an analog of the condition (\ref{condition_c+-delta_eq}), replacing $k_0$ by $k_0+1$, holds near a neighborhood of $p_+$. Hence, we also conclude that $\sum_{j=0}^{n+1}(j-k_0-1)\val(x_j)$ offers an integral affine function.
In particular, their difference, $\sum_{j=0}^{n+1}\val(x_j)$, must be an integral affine function as well. However, by (\ref{val_1+y_homogeneous}), this means $\val(1+y)$ should be an integral affine function, which is impossible as $s=0$.
Therefore, we finally see that $p$ is not $F$-smooth.
\end{proof}

\subsection{Proof of family Floer SYZ conjecture}

We aim to confirm that the tropically continuous fibration $F$ offers a tangible representation for the singular extension of the canonical dual affinoid torus fibration $\pi_0^\vee$. Thus, we will explore our version of the SYZ conjecture for the $A_n$ singularity, ultimately examining the proof of Theorem \ref{Main_thm_oversimplified}.

The base manifold of the special Lagrangian fibration $\pi$ in (\ref{pi_Lagrantian_fibration_eq}) is $B=\mathbb R\times \mathbb R_{>0}$, while the tropically continuous fibration $F$ sends into the embedded surface $\pmb\gamma$ in $\mathbb R^{n+3}$. We aim to develop a matching between them.
Recall that the function $\psi: B\to\mathbb R$ in (\ref{psi_over_B_eq}) essentially captures the K\"ahler geometry of the $A_n$-smoothing under the $S^1$ symmetry. Indeed, we can identify the reduced space at moment $s$ with the complex plane $(\mathbb C, \omega_{red,s})$ furnished with the reduced K\"ahler forms. Accordingly, $\psi(s,r)$ corresponds to the $\omega_{red,s}$-symplectic area of the disk with radius $r$, centered at the origin in $\mathbb C$.
However, we observe that $\psi$ is only continuous on $B$ and smooth on $B_0$. This constraint serves as a primary reason for embracing the notion of tropically continuous fibration.
On the other hand, recall also that the function $\psi(s,r)$ largely forms the expression for the embedded surface $\pmb\gamma$ in (\ref{pmb_gamma_eq}). This will significantly guide the desired construction as follows.

Define
\begin{equation}
	\label{j_k_topo_eq}
	j_k: B=\mathbb R \times \mathbb R_{>0}\to \mathbb R \qquad (s,r)
	\mapsto
	\{\psi(s,r), \psi_0(s), \dots, \psi_n(s) \}_{[k]}
\end{equation}
for $0\leqslant k\leqslant n+1$. Recall that we abbreviate $\psi_k(s)=\psi(s,|a_k|)$ as in (\ref{psi_k_abbrev_notation}).
Then, we define
\begin{equation}
	\label{j_total_topo_eq}
	j=(j_0,j_1,\dots, j_{n+1}, s) : \  \mathbb R\times \mathbb R_{>0} \to \mathbb R^{n+3} 
\end{equation}
by
$
 (s,r)\mapsto \big( j_0(s,r),j_1(s,r),\dots, j_{n+1}(s,r), s \big)
$.
It is clear from the definition that
$
j(s,r)=\pmb\gamma(s,\psi(s,r))
$.
Further, by Proposition \ref{increasing_psi_prop}, the assignment $r\mapsto \psi(s,r)$ is strictly increasing for any fixed $s$ and gives rise to an $s$-dependent diffeomorphism from $\mathbb R_{>0}$ to itself. In particular, we conclude that
\begin{equation}
	\label{j_pmbgamma_corr_eq}
	j(B) = \pmb\gamma (\mathbb R\times \mathbb R_{>0})
\end{equation}

On the other hand, by Lemma \ref{F_global_image_lem}, the image of the tropically continuous fibration $F$ is $\pmb\gamma(\mathbb R\times \mathbb R)$.
A natural question that arises is regarding the $F$-preimage of the open subset in (\ref{j_pmbgamma_corr_eq}).

\begin{lem}
	\label{Y_mathscr_Y_lem}
	The $F$-preimage of $\pmb\gamma(\mathbb R\times \mathbb R_{>0})$ is the analytic open subset
\begin{equation}
	\label{Y_mathscr_eq}
	\mathscr Y
	=
	\left\{
\Big| \prod_{j=0}^{n+1}  x_j^j \Big|<1
	\right\}
\end{equation}
where we use the homogeneous coordinates on $Y_\Sigma^*$ as before. By definition, we immediately have
\begin{equation}
	\label{j=F_image_eq}
	j(B)=F(\mathscr Y)
\end{equation}
Therefore, there is a tropically continuous fibration
\begin{equation}
	\label{f_j_F_mainbody_eq}
	f=j^{-1}\circ F :  \mathscr Y  \to B
\end{equation}
\end{lem}

\begin{proof}
	Let $\mathbf x=([x_0:\cdots: x_n], y)$ be an arbitrary point in the preimage we consider. Then, there exists a unique pair $(s,c)$ such that $F(\mathbf x)=\pmb\gamma(s,c)$.
	Recall that $0<\psi_0(s)<\cdots <\psi_n(s)$ for any $s$.
	Hence, the condition is to require $c>0$.
	By the computation of the order statistics in (\ref{line_segments}) and by the defining formula of $F$ in (\ref{F_defn_eq}), this means that
	\[
	\sum_{j=0}^{n+1} j\cdot \val(x_j) >0  \  .
	\]
	In other words, the non-archimedean norm of $\prod_{j=0}^{n+1} x_j^j$ is smaller than $1$. 
	By (\ref{G_group_eq}), we observe that the expression $\prod_{j=0}^{n+1} x_j^j$ is well-defined for the homogeneous coordinates.
	The proof is complete.
\end{proof}


The following key result integrates all the preceding constructions and completes the proof of Theorem \ref{Main_thm_oversimplified}.

\begin{thm}
	\label{diagram_thm}
	We have the following commutative diagram
	\[
	\xymatrix{
X_0^\vee \ar[rr]^g \ar[d]^{\pi_0^\vee} & & \mathscr Y \ar[d]^F \\
B_0\ar[rr]^j & & \mathbb R^{n+3}	
}
	\]
\end{thm}

\begin{proof}
	Recall that we have established the identification $X_0^\vee\equiv \bigsqcup_{k=0}^{n+1} T_k / \sim$ in (\ref{identification_mirror_eq}) and the subsequent one for $\pi_0^\vee$ in (\ref{identification_dual_affinoid_torus_fibration__eq}).
	Recall also that $T_k=\trop^{-1}(\chi_k(U_k))$ by (\ref{T_k_eq}) and the smooth locus $B_0$ is covered by the contractible open subsets $U_0$, $U_1$, $\dots$, and $U_{n+1}$ in (\ref{U_k_eq}).
	
	Fix a point $\mathbf y$ in $X_0^\vee$, and write $(s,r)=\pi_0^\vee(\mathbf y)$ in $B_0$. Then, we immediately obtain
$
	j\circ \pi_0^\vee (\mathbf y)
	=\pmb\gamma(s,\psi(s,r))
$
which is expected to align with $F\circ g(\mathbf y)$. Let's verify this as follows.

There always exists some $0\leqslant k_0\leqslant n+1$ (perhaps not unique) such that $(s,r)\in U_{k_0}$.
Then, $\mathbf y$ is identified with a point $(y_1,y_2)$ in $T_{k_0}$.
Due to (\ref{identification_dual_affinoid_torus_fibration__eq}), we know that $\val(y_1)=s$ and $\val(y_2)=c-k_0\min\{0,s\}$, where for clarity we set
\[
c:=\psi(s,r)  \  .
\]
 If $s=0$, then by virtue of (\ref{U_k_eq}), we necessarily have \[|a_{k_0-1}|<r<|a_{k_0}|
	\]
	(Here we mean $0<r<|a_0|$ when $k_0=0$ and $|a_n|<r<\infty$ when $k_0=n+1$.)
	
	\noindent
	Using Proposition \ref{increasing_psi_prop} then obtains that
	\[
	\psi_{k_0-1}(s)  < c <\psi_{k_0}(s)
	\]
	(Here we mean $0<\psi(s,r)<\psi_0(s)$ when $k_0=0$ and $\psi_n(s)<\psi(s,r)<\infty$ when $k_0=n+1$.)
	According to (\ref{g_k_z_k_w_k_eq}) and (\ref{homogeneous_coordiantes_affine_chart}), we can compute each component of $F\circ g(\mathbf y)$ as follows:
	\begin{align*}
		F_k\circ g(\mathbf y)&= F_k \circ g_{k_0}(y_1,y_2)
		=F_k \Big(\big[1:\cdots : 1 : \underset{\big(x_{k_0}\big)}{ y_2^{-1}(1+y_1)}: y_2 : 1 : \cdots : 1
		\big] , y_1 \Big) \\
		&=
		\left\{
		c+(k_0-k)\left(\val(1+y_1)-\min\{0,\val(y_1)\}\right) \, , \, \psi_0(s), \dots, \psi_n(s) 
		\right\}_{[k]}
	\end{align*} 
Analogous to the proof of Lemma \ref{F_global_smooth_locus_lem}, we can ultimately confirm the desired relation:
\[
F\circ g(\mathbf y) = \left(\psi_0(s),\dots, \psi_{k_0-1}(s), c , \psi_{k_0}(s),\dots, \psi_n(s), \ s \right)=\pmb\gamma(s,c)=j\circ \pi_0^\vee(\mathbf y)
\]
If $s\neq 0$, then $\val(1+y_1)=\min\{0,\val(y_1)\}=\min\{0,s\}$ and thus the above computation directly implies that $F_k\circ g(\mathbf y)=\gamma_k(s,c)$.
To sum up, we also conclude that $F\circ g(\mathbf y)=\pmb\gamma(s,c)=j\circ \pi_0^\vee(\mathbf y)$.
\end{proof}

\begin{proof}[Proof of Theorem \ref{Main_thm_oversimplified}]
Recall that $g$ is an analytic embedding map.
By Theorem \ref{diagram_thm}, the image $g(X_0^\vee)$ agrees with $f^{-1}(B_0)$.
In other words, $g$ intertwines the affinoid torus fibration $\pi_0^\vee$ and $f_0$. Therefore, the integral affine structure induced by $f_0$ is exactly the one induced by $\pi_0^\vee$, while the latter is precisely the one induced by $\pi_0$ due to the identifications (\ref{identification_mirror_eq}) (\ref{identification_dual_affinoid_torus_fibration__eq}). This verifies the conditions (ii) and (iii).
Finally, by Lemma \ref{F_global_smooth_locus_lem}, it remains to verify that $j(0,|a_k|)=A_k(0)$, which is straightforward from their definitions (\ref{j_total_topo_eq}) and (\ref{corner_point_gamma}). The proof is now complete.
\end{proof}

We have proven Theorem \ref{Main_thm_oversimplified} for a generic $A_n$-smoothing under the condition that the norms $|a_k|$'s are pairwise distinct.
In this event, these $\mathcal P=\{a_0,\dots, a_n\}$ already constitute an open dense subset in the configuration space.
For an arbitrary $\mathcal P$ without this condition, minimal additional work is needed, as most of the main body can be repeated without changes.
Generally speaking, there exist a partition of the index set $\{0,1\dots, n\}$ into subsets $I_1,\dots, I_m$ ($1\leqslant m \leqslant n+1$) together with positive real numbers $\lambda_1<\cdots <\lambda_m$ such that $|a_k|=\lambda_j$ for $k\in I_j$. Then, the singular locus $\Delta$ of $\pi$ has $m$ singular points, with $m+1$ chambers to consider.

Our method applies to all cases, including the most complicated case with $m=n+1$ already addressed in the main text, with simpler cases remaining. We offer a sketch of the argument for the case with $|a_0|=\cdots=|a_n|=:\lambda$. 
It may be instructive to give some intuition. Note that the same formula $f=f_{\mathcal P}$ gives the desired solution for almost every $\mathcal P$ in the configuration space. For the exceptional choices of $\mathcal P$ in the above case, just imagine the various $\psi(\cdot, |a_j|)$ in the order statistics functions in $j$ and $F$ deforming to the same value $\psi(\cdot, \lambda)$ as shown in Figure \ref{figure_order_statistics_degenerate}. This serves as a guiding concept, not a complete analysis, since we do not aim to explore the deformation in the non-archimedean context. 
In any case, by simply repeating our arguments from the main text, we can always rigorously obtain the explicit mirror fibration $f=f_{\mathcal P}$ satisfying the desired properties.

The Lagrangian fibration $\pi(u, v, z) = (\frac{1}{2}(|u|^2 - |v|^2), |z|)$, as previously stated in (\ref{pi_Lagrantian_fibration_eq}), has a base $B = \mathbb R \times \mathbb R_{>0}$ containing only a singular point at $\mathfrak q:=(0, \lambda)$. We also set $B_0 = B \setminus \{\mathfrak q \}$.
The singular Lagrangian fiber $L_{\mathfrak q}$ consists of a chain of $n+1$ spheres, with singular points at $(u, v, z) = (0, 0, a_k)$ for each $0\leqslant k\leqslant n$.
Applying Theorem \ref{Main_theorem_thesis_thm} to the pair $(X, \pi_0)$ generates a pair $(X_0^\vee, \pi_0^\vee)$, and identifications analogous to (\ref{identification_mirror_eq}) and (\ref{identification_dual_affinoid_torus_fibration__eq}) exist as well.
As in Proposition \ref{wall_place_prop}, the Lagrangian torus fiber $L_q$ over a smooth point $q = (s, r)$ in $B_0$ bounds a nontrivial Maslov-0 holomorphic disk in $X$ if and only if $r = \lambda$. Thus, the walls are $H_- = \{\lambda\} \times (-\infty, 0)$ and $H_+ = \{\lambda\}\times (0, +\infty)$.
Let $\mathscr N_-$ and $\mathscr N_+$ denote small neighborhoods of $H_-$ and $H_+$ in $B_0$, respectively. We define $U_0 = \mathbb R\times (0, \lambda) \cup \mathscr N_+ \cup \mathscr N_-$ and $U_1 = \mathbb R\times (\lambda , +\infty) \cup \mathscr N_+ \cup \mathscr N_-$.
Similarly, we can identify $(\pi_0^\vee)^{-1}(U_i) \cong \trop^{-1}(V_i) =:T_i$ for $i = 0, 1$, and for some integral affine chart $\chi_i: U_i\to V_i$. Further, we can identify 
\[
X_0^\vee \equiv T_0\cup T_1/\sim
\]
with the gluing relation as follows: we require $y\sim y'$ if $y = (y_1, y_2) \in T_0$ and $y' = (y_1', y_2') \in T_1$ satisfy $(y_1', y_2') = (y_1, y_2(1 + y_1)^{n+1})$. Moreover, the dual affinoid torus fibration $\pi_0^\vee$ is characterized as follows: given $y \in X_0^\vee$, we assume $(s, r) = \pi_0^\vee(y) \in B_0$. When $y = (y_1, y_2) \in T_0$, we have $\val(y_1) = s$ and $\val(y_2) = \psi(s, r)$. When $y = (y_1, y_2) \in T_1$, we have $\val(y_1) = s$ and $\val(y_2) = \psi(s, r) - (n+1)\min\{0, s\}$.

With the above identifications, we examine an analytic embedding $g': T_0\to Y_{\sigma_0}^* \cong \Lambda^2_{z_0, w_0}$, defined by $(y_1, y_2) \mapsto (y_2^{-1}(1 + y_1), y_2)$, and an analytic embedding $g'': T_1\to Y_{\sigma_n}^* \cong \Lambda^2_{z_n, w_n}$, defined by $(y_1, y_2) \mapsto (y_2^{-1}, y_2(1 + y_1))$. By the toric gluing relation, we have $z_n = z_0 (1 + y)^n$ and $w_0 = w_n (1 + y)^n$. Thus, the maps $g'$ and $g''$ are compatible with the identification in $X_0^\vee = T_0\cup T_1/\sim$. Consequently, we deduce that $g'$ and $g''$ combine to form an analytic embedding $g: X_0^\vee \to Y_\Sigma^*$.

By (\ref{psi_k_abbrev_notation}), we similarly write $\psi=\psi(s)=\psi(s,\lambda)$. Inspired by (\ref{gamma_k}), we define
\[
\gamma_k(s, c) =\{ c,   \psi, \psi, \dots \psi  \}_{[k]} 
=
\begin{cases}
	\min\{c,\psi\} & \text{if  } k=0 \\
	\psi &\text{if  } 1\leqslant k\leqslant n \\
	\max\{c,\psi\} & \text{if  } k=n+1 
\end{cases}
\]
where there are \text{$n+1$ copies of the same $\psi$} in $\{\cdots \}_{[k]}$ such that the order statistic degenerates (see Figure \ref{figure_order_statistics_degenerate}).
Therefore, by (\ref{pmb_gamma_eq}), we may omit the redundant middle $n$ components and directly designate
\[
\pmb\gamma: \mathbb R^2 \to\mathbb R^3, \qquad \pmb\gamma(s,c)= (\gamma_0(s,c), \gamma_{n+1}(s,c), s) = (\min\{c,\psi\}, \max\{c,\psi\}, s)  \  .
\]
It has two line segments with a corner point $A=(\psi, \psi, s)$.
Accordingly, we follow (\ref{j_total_topo_eq}) to define $j_k(s,r)=\gamma_k(s,\psi(s,r))$ and $j=(j_0,j_{n+1},s)$.

In the mean time, the order statistic degeneration happens for the tropically continuous fibration as well. Inspired by (\ref{F_k_defn_eq}), we consider
\[
F_0 =
	\min\left\{ \sum_{j=0}^{n+1} j \cdot \val(x_j)  \quad  , \quad \psi(\val(y), \lambda) \right\} 
\]
and
\[
F_{n+1}= \max\left\{ \sum_{j=0}^{n+1} (j-n-1)\cdot \val(x_j) + (n+1) \min\{0,\val(y)\}   \quad , \quad  \psi(\val(y), \lambda )
\right\}
\]
but define $F_k=\psi(\val(y), \lambda)$ for all other $1\leqslant k\leqslant n$.
Here we use the homogeneous coordinates (\ref{homogeneous_coordinates_identification}).
By (\ref{F_defn_eq}), we may also omit the middle $n$ degenerating components again and define
\[
F=(F_0,F_{n+1}, \val(y)) :Y_\Sigma^* \to\mathbb R^3  \  .
\]
Similar to Lemma \ref{Y_mathscr_Y_lem}, we can verify that for the analytic open domain $\mathscr Y=\{ |\prod_{k=0}^{n+1} x_k^k| <1\}$, we exactly have $j(B)= F(\mathscr Y)$.
Therefore, just like (\ref{f_j_F_mainbody_eq}), we can define
\[
f=j^{-1}\circ F : \mathscr Y\to B  \  .
\]
Of course, one may insist on a common form of the formula, setting $\tilde j=(j_0,j_1,\dots, j_n, j_{n+1}, s)$ and $\tilde F=(F_0,F_1,\dots, F_n,F_{n+1},\val(y))$. But then, there is no essential difference as one can easily verify that $\tilde j^{-1}\circ \tilde F=j^{-1}\circ F$.
Finally, it remains to verify a commutative diagram as in Theorem \ref{diagram_thm}.
The argument is almost identical and even easier.
Fix a point $\mathbf y$ in $X_0^\vee$, and write $(s,r)=\pi_0^\vee(\mathbf y)\in B_0$.
Setting $c=\psi(s,r)$, it remains to verify that $\pmb\gamma(s, c)=F\circ g(\mathbf y)$.

When $(s,r)\in U_0$, we identify $\mathbf y$ with a point $(y_1,y_2)$ in $T_0$. Hence, $\val(y_1)=s$ and $\val(y_2)=\psi(s,r)$.
It follows that
\begin{align*}
F\circ g(\mathbf y)
&=
\big(F_0 \circ g'(\mathbf y), F_{n+1}\circ g'(\mathbf y), s \big)  \\
&=
\big(\min\{\val(y_2) , \psi(s,a)\} \  ,  \quad \max\{ \val(y_2)-(n+1)\val(1+y_1) +(n+1)\min\{0,s\} , \psi(s,\lambda)\} \  ,  \quad  s \big) \\
&=
\big(\min\{ c , \psi(s,\lambda)\} \  ,  \quad \max\{ c-(n+1)\tau , \psi(s,\lambda)\} \  ,  \quad s \big)
\end{align*}
where we put $\tau=\val(1+y_1)-\min\{0,s\}\geqslant 0$. If $s\neq 0$, then $\tau=0$. If $s=0$, then $r<\lambda$ and $c= \psi(0,r)<\psi(0,\lambda)$. Thus, $\max\{ c-(n+1)\tau , \psi(s,\lambda)\}=\max\{c,\psi(s,\lambda)\}$. In either cases, we can similarly conclude $F\circ g(\mathbf y)=\pmb\gamma(s,c)$ as desired.

When $(s,r)\in U_1$, $\mathbf y$ is identified with a point $(y_1,y_2)\in T_1$.
Hence, $\val(y_1)=s$ and $\val(y_2)=\psi(s,r)-(n+1)\min\{0,s\}$.
A similar computation yields
\begin{align*}
	F\circ g(\mathbf y)
	&=
	\big(F_0 \circ g''(\mathbf y), F_{n+1}\circ g''(\mathbf y), s \big)  
	=
	\big(\min\{c+(n+1)\tau , \psi(s,a)\} \  ,  \quad \max\{ c , \psi(s,a)\} \  ,  \quad  s \big)
\end{align*}
If $s\neq 0$, then $\tau=0$ as before. If $s=0$, then $r>\lambda$ and $c=\psi(0,r)>\psi(0,\lambda)$.
Thus, $\min\{c+(n+1)\tau, \psi(0,\lambda)=\min\{c,\psi(0,\lambda)\}$. We can conclude $F\circ g(\mathbf y)=\pmb\gamma(s,c)$ as well.

As extra evidence, we aim to consider the variants of Corollary \ref{F_image_toric_divisor_cor} and \ref{F_image_torus_open_dense_torus_cor}.
Recall that for any $1\leqslant k\leqslant n$, the irreducible toric divisor $\mathcal D_k$ is compact and is given by $x_k=0$ in the homogeneous coordinate. Then, $1+y=\prod_{j=0}^{n+1}x_j=0$ on such $\mathcal D_k$, which implies $\val(y)=0$. Moreover, $F_0|_{\mathcal D_k}=\min\{+\infty, \psi(0,\lambda)\}=\psi(0,\lambda)$ and $F_{n+1}|_{\mathcal D_k}=\max \{-\infty, \psi(0,a)\} = \psi(0,\lambda)$.
Hence, we see that $F(\mathcal D_k)= (\psi(0,\lambda),\psi(0,\lambda),0)$ which exactly agrees with $j(\mathfrak q)=j(0,q)$, the image of the unique singular point.

\appendix
\section{Miscellaneous related topics}
\label{s_miscellaneous_topics}



Recall that the reduced space $\bar X_{red,s}$ of $\bar X$ is diffeomorphic to $\mathbb C$ (i.e., the base in Figure \ref{figure_psi(s,r)}) through the projection $p(u,v,z) = z$. It is crucial to differentiate between $p$ and the Lagrangian fibration $\pi$, as well as distinguishing the set $\mathcal P=\{a_0,\dots, a_n\}$ of critical values of $p$ in $\mathbb C$ from the singular locus $\Delta$ in the base $B = \mathbb R \times \mathbb R_{>0}$ of $\pi$.
For clarity, let's first assume $|a_0|<\cdots <|a_n|$.

When we refer to \textit{a curve in $(\mathbb C, \mathcal P)$}, we mean a subset $c \subset \mathbb C$, which can either be a non-contractible simple closed curve in $\mathbb C \setminus \mathcal P$, or the image of an embedding $\mathfrak c: [0,1] \to \mathbb C$ where $\mathfrak c^{-1}(\mathcal P) = \{0, 1\}$, as described in \cite[3a]{Seidel_Khovanov2002quivers}. Given any such curve $c$ in $(\mathbb C, \mathcal P)$, we can associate a Lagrangian two-sphere $\mathcal L_c$, which is explicitly defined as 
$
\mathcal L_c =\{ (u,v,z)\in\mathbb C^3 \mid uv=h(z),  \   |u|=|v|  ,   \   z\in c\}
$.


Due to Khovanov and Seidel \cite{Seidel_Khovanov2002quivers}, the Lagrangian isotopy class of $\mathcal L_c$ recovers the isotopy class of $c$ in $(\mathbb C, \mathcal P)$. Considering an $n$-chain of smooth curves $c_1, \dots, c_n$ in $(\mathbb C, \mathcal P)$ with $c_k$'s endpoints at $a_{k-1}$ and $a_k$, we can further assume each $c_k$ is an admissible curve in normal form in the sense of \cite[3e]{Seidel_Khovanov2002quivers}, which is always achievable via an isotopy.
Let $\mathcal L_k=\mathcal L_{c_k}$, and the set $\mathcal L_{1}, \dots, \mathcal L_{n}$ forms an \textit{$(A_n)$-configuration of Lagrangian two-spheres} \cite[\S 8]{seidel1999lagrangian}, characterized as follows:
\[
|\mathcal L_{i}\cap \mathcal L_{j} | 
=
\begin{cases}
	1, & \text{if }  |i-j|=1 \\
	0, & \text{if } |i-j|\geqslant 2
\end{cases}
\]
It occurs as vanishing cycles in the smoothing of $A_n$-singularity (cf. \cite[3.5]{Seidel_Thomas_2001braid}, \cite[20a]{SeidelBook}).
For $1 \leqslant k \leqslant n$, we note that the image $\pi(\mathcal L_{k})$ gives a compact segment $\Gamma_k$ in $B = \mathbb R \times \mathbb R_{>0}$, which connects the two focus-focus singular points $(0, |a_{k-1}|)$ and $(0, |a_k|)$.
Note that $\pi$ maps any point in $\mathcal L_c$ to a point in $B$ with $s=0$, given $|u|=|v|$ on $\mathcal L_c$.
Hence, we find that:
\[
\Gamma_k=\pi(\mathcal L_k)=\{(s,r)  \mid s=0,  |a_{k-1}| \leqslant r \leqslant |a_k| \} = \{0\}\times [ |a_{k-1}|, \ |a_k| ]  \quad \subseteq B
\]

Denote the minimal resolution of the $A_n$-singularity by $\kappa: Y_\Sigma \to \Lambda^2 / \mathbb Z_{n+1}$. Recall that the fan $\Sigma$ has $n+2$ rays $v_k=(k,1)$ for $0 \leqslant k \leqslant n+1$. Let $\mathcal D_k$ be the toric divisor for $v_k$.
Then, $\mathcal D_k$ is described as $x_k=0$ in homogeneous coordinates.
The self-intersection number of $\mathcal D_k$ equals $-2$ for $1 \leqslant k \leqslant n$, excluding $\mathcal D_0$ and $\mathcal D_{n+1}$ yet. The exceptional locus is $\kappa^{-1}(0)=\mathcal D_1\cup \mathcal D_2\cup \cdots \cup \mathcal D_n$.
It is a chain of $n$ irreducible rational $(-2)$-curves such that $\mathcal D_i\cdot \mathcal D_{i+1}=1$ and $\mathcal D_i\cdot \mathcal D_j =0 $ if $|i-j|\geqslant 2$.
Notice that these relations resemble the above ones with respect to Lagrangian spheres $\mathcal L_{i}$'s.


\begin{figure}
	\centering
	\captionsetup{font=footnotesize}
	\includegraphics[width=7.5cm]{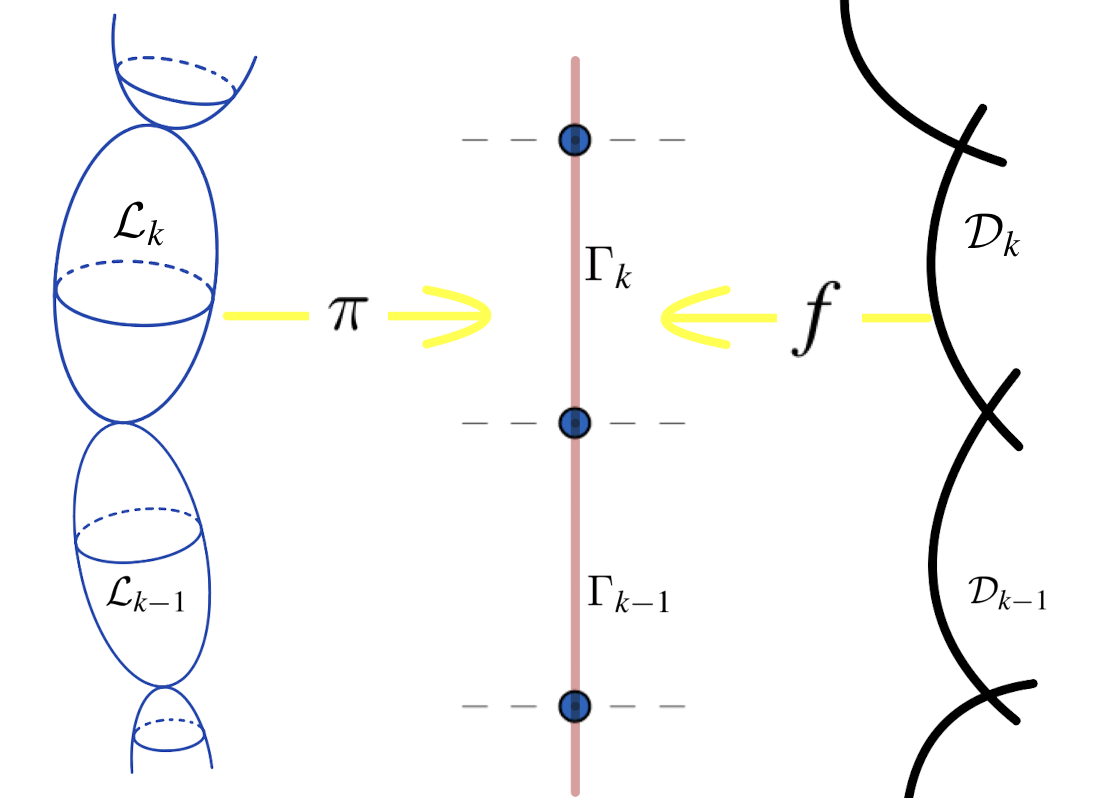}
	\caption{\footnotesize The illustration for Observation \ref{observation_A} linking Lagrangian spheres $\mathcal L_i$'s and exceptional rational $(-2)$-curves $\mathcal D_i$'s}
	\label{figure_observation}
\end{figure}

\begin{observation}
	\label{observation_A}
	For any $1\leqslant i \leqslant n$, the following coincidence holds (see Figure \ref{figure_observation}): 
	\begin{equation}
		\label{f(D_k)_intro}
		\pi(\mathcal L_{i})  =\Gamma_i= 	f(\mathcal D_i) 
	\end{equation}
\end{observation}

\begin{proof}[Sketch]
	Fix $i$, and set $c_k:=\sum_{j=0}^{n+1} (j-k) \val(x_j)$ for all $k$. Notice $\mathcal D_{i}$ is given by $x_{i}=0$, i.e. $\val(x_{i})=+\infty$.
	Then, $c_k$ is $+\infty$ whenever $k<i$, is $-\infty$ whenever $k>i$. Only when $k=i$, the $c_k$ can be seen as a free variable. By the order statistics for $F$ in (\ref{f_eq_intro}), the value of $F_k(\mathcal D_i)$ is fixed whenever $k\neq i$, while $F_i(\mathcal D_i)$ ranges in a line segment between $\psi(0,|a_{i-1}|)$ and $\psi(0,|a_i|)$, which finally matches the line segment $\Gamma_i$ via the embedding map $j$. (See Corollary \ref{F_image_toric_divisor_cor} or Lemma \ref{F_global_image_lem} for the full details.)
\end{proof}


The above verification is merely at the set-theoretic level. Thus, it seems that the knowledge required to validate Observation \ref{observation_A} is very standard. At the same time, the theoretical bedrock that reveals the explicit equation for $f$ in (\ref{f_eq_intro}) extends deep into several distinct fields of mathematics.
The relation in (\ref{f(D_k)_intro}) is interesting, given that $\pi$ and $f$ already have restrictive dualistic conditions connecting each other (Definition \ref{SYZ_mirror_defn}). 
Notably, the dual fibration $f=f_{\mathcal P}$ \textit{only} studies quantum correction holomorphic disks within the SYZ picture, with \textit{no} consideration of the Lagrangian spheres $\mathcal L_i$ nor the exceptional rational curves $\mathcal D_i$ whatsoever.
Had we not been serious in our efforts to make $f$ explicit as shown in (\ref{f_eq_intro}), we would not have made this striking observation.

Besides, for any other choice of $\mathcal P=\{a_0,\dots, a_n\}\subset \mathbb C$ in the configuration space $\mathscr C=\mathrm{Conf}_{n+1}(\mathbb C)$, we can similarly check that the \textit{same} formula in (\ref{f_eq_intro}) for $f=f_{\mathcal P}$ admits an analog of Observation \ref{observation_A}, with the expected reordering.
Then, at least in a preliminary sense and at the object level, these observations align quite well with the philosophy in the renowned works of Khovanov, Seidel, and Thomas \cite{Seidel_Khovanov2002quivers,Seidel_Thomas_2001braid}.

Now, a natural question arises:
\textit{Does the observation merely refer to a coincidence?}

Driven by an aesthetic taste, we contend otherwise: we are merely perceiving the "tip of the iceberg" with a vast bulk lying beneath, awaiting deeper exploration.
A starting point could be the development of affinoid coefficients in Lagrangian Floer cohomology \cite{Yuan_affinoid_coeff,Yuan_c_1} or the generalization of the family Floer functor approach to include quantum corrections.
In principle, the explicit nature of our formula (\ref{f_eq_intro}) allows us to visually represent any phenomenon that arises when moving $\mathcal P$ along a loop in $\mathscr C$, while capturing any additional structure in this process is a separate challenge that should be addressed somewhere else.
Note that the fundamental group $\pi_1(\mathscr C)$ is identified with the braid group $B_{n+1}$.

\subsection{Braid group action, affinoid coefficients, and family Floer functor}
\label{ss_miscell_Dehn_affinoid_coeff}

The subsequent discussion is of a heuristic nature and is not necessary for our main result. However, it highlights some potential avenues for future research. As such, we will not pursue strict rigor in the remainder of this section.
As indicated by Seidel in \cite{seidel1999lagrangian}, an $(A_n)$-configuration of Lagrangian spheres in a symplectic manifold gives rise to a homomorphism $\varrho$ from the braid group $B_{n+1}$ to the group of symplectic isotopy classes of automorphisms.
Following Khovanov and Seidel \cite{Seidel_Khovanov2002quivers}, we can describe this as follows.
Recall that the $A_n$-smoothing refers to the affine space defined by $uv=h(z)$, where 
\[
h(z)=
(z-a_0)(z-a_1)\cdots (z-a_n)
=: z^{n+1}+w_nz^n+\cdots +w_1z+w_0 \ .
\]
Assuming there are no multiple roots, the parameters $w=(w_0,\dots, w_n)$ form an open subset $\mathcal W\subset \mathbb C^{n+1}$ that is homotopy equivalent to the (unordered) configuration space $\mathscr C=\mathrm{Conf}_{n+1}(\mathbb C)$. This gives rise to a $\mathcal W$-family of Milnor fibers.
Since $\pi_1(\mathcal W)\cong \pi_1(\mathscr C)$ is naturally identified with the braid group $B_{n+1}$, the parallel transport for appropriate choices of connections in this family defines the aforementioned braid group action $\varrho$; see \cite[(1.4)]{Seidel_Khovanov2002quivers}.

There is a more direct definition of $\varrho$ in terms of (generalized) Dehn twists \cite[\S 6]{seidel1999lagrangian}.
A Lagrangian sphere $\mathcal L=\mathcal L_c$ leads to a symplectic automorphism $\tau_{\mathcal L}$ known as a Dehn twist along $\mathcal L$. It is supported within a small Weinstein neighborhood $\mathcal U$ of $\mathcal L$. However, even if it is small, this $\mathcal U$ inevitably intersects a nearby smooth Lagrangian torus fiber bounding a nontrivial Maslov-0 holomorphic disk $u$ (see Figures \ref{figure_psi(s,r)} and \ref{figure_area_singular}). The symplectic area of this $u$ can be arbitrary small, but we must consider the counts in the class $k[u]$ for any arbitrary large integer $k \gg 0$ (see Figure \ref{figure_hp_maslov_0}).
In other words, the mirror Berkovich analytic topology for $f$ necessitates simultaneous considerations of holomorphic disks with small symplectic areas and large multiples of them.
Even though a Lagrangian sphere itself may not enclose holomorphic disks, its intersection with other graded Lagrangian submanifolds that do bound such disks should also be taken into account for the global attributes.
In fact, any graded Lagrangian submanifold should define an object in the derived Fukaya category \cite[1.2]{Seidel_Thomas_2001braid}.

Seidel and Thomas \cite[1.3]{Seidel_Thomas_2001braid} anticipate that "twist functors and generalized Dehn twists correspond to each other under mirror symmetry".
By the evidence (\ref{f(D_k)_intro}), we are likely approaching a cogent geometric interpretation for their anticipation.
In general, if the mirror object of $\mathcal L$ is a line bundle $\mathscr E$, then according to their argument in \cite[(1.6)]{Seidel_Thomas_2001braid}, we should expect:
\begin{equation}
	\label{Hom_HF_intro_eq}
	\mathrm{Hom}_{\mathscr O_Y} (\mathscr E,\mathscr E)
	\cong \mathrm{HF}^*(\mathcal L,\mathcal L)
\end{equation}

\textit{However, the structure sheaf $\mathscr O_Y$ of the mirror Berkovich analytic space $Y$ is modeled locally on affinoid algebras.} Thus, a main trouble comes from the implication that the above relation (\ref{Hom_HF_intro_eq}) requires a version of Lagrangian Floer cohomology with affinoid algebra coefficients, or "\textit{affinoid coefficients}" for short. This requires further foundational work \cite{Yuan_affinoid_coeff}.
On the other hand, due to the Berkovich topology being a refinement of the Zariski topology, the concepts of derived categories of coherent sheaves, Fourier-Mukai transforms, and twist functors continue to apply in the non-archimedean setting. Thus, incorporating affinoid coefficients can possibly enrich and expand the aforementioned relationship and connect Seidel's long exact sequence \cite{Seidel2003_long_exact} with the natural exact sequence of twist functors associated with spherical objects. We anticipate that the family Floer functor will fulfill this task.

\begin{rmk}
	The affinoid coefficients, initially introduced in \cite{Yuan_c_1}, elucidate and generalize a folklore conjecture known to Auroux, Kontsevich, and Seidel. With verified new examples in \cite{Yuan_local_SYZ}, it links $c_1$-eigenvalues of quantum cohomology on the compactified spaces on the K\"ahler side to critical values of the resulting superpotential on the Berkovich side.
	A basic motivation for affinoid coefficients arises from the following idea. Selecting a single bounding cochain, while formally resolving the curvature term in the most general case, unavoidably results in losing information. However, in a slightly less general situation such as dealing with graded Lagrangian submanifolds \cite{SeidelGraded}, merging \textit{all} bounding cochains together leads to specific new Berkovich analytic structure, as demonstrated in \cite{Yuan_I_FamilyFloer}.
\end{rmk}

Eventually, a conceptual and functorial justification of Observation \ref{observation_A} is expected in relation to its upgrade to the \textit{family Floer functor}.
This approach is investigated by Abouzaid and Fukaya \cite{AboFamilyFaithful,FuFamily,FuPPT}.
Specifically, the Lagrangian sphere $\mathcal L_i$ should be assigned to a coherent sheaf representing the family of Floer cohomologies
\begin{equation}
	\label{assign_functor_intro_eq}
	\mathscr E_{\mathcal L_i}: \quad 
	q \quad  \text{"}\mapsto\text{"}  \quad HF(\mathcal L_i, L_q)
\end{equation}
as $q$ moves in the SYZ base, where we set $L_q=\pi^{-1}(q)$.

At the set-theoretic level, this tentative family Floer functor approach forms a perfect match with Observation \ref{observation_A}. Specifically, since $\pi(\mathcal L_i) = \Gamma_i$, we know that $\mathcal L_i \cap L_q$ is empty if and only if $q$ does not intersect $\Gamma_i$. Hence, we can reasonably assume that $\mathscr E_{\mathcal L_i}$ is supported by a subset of $f^{-1}(\Gamma_i)$. As $f(\mathcal D_i)=\Gamma_i$, it suggests that $\mathscr E_{\mathcal L_i}$ corresponds to a line bundle or coherent sheaf related to $\mathcal D_i$.
Moreover, the existing homological mirror symmetry results like \cite{ishii2005autoequivalences,ishii_Stability_condition_An_ishii2010stability,Chan_An_Tduality,Chan_Ueda_2013dual,Pomerleano2011curved_I} have predicted that $\mathcal L_i$ mirrors to $\mathscr O_{\mathcal D_i}(-1)$, which indeed agrees with the above speculation.

At a deeper level, incorporating non-archimedean analyticity is necessary. We must describe how the assignment in (\ref{assign_functor_intro_eq}) changes analytically as $q$ moves. This provides additional motivation to introduce affinoid coefficients.


\subsection{Singular fiber collision, order statistic degeneration, and Gamma conjecture}

\begin{figure}
	\centering
	\captionsetup{font=footnotesize}
	\begin{subfigure}{0.4\textwidth}
		\includegraphics[width=4.6cm]{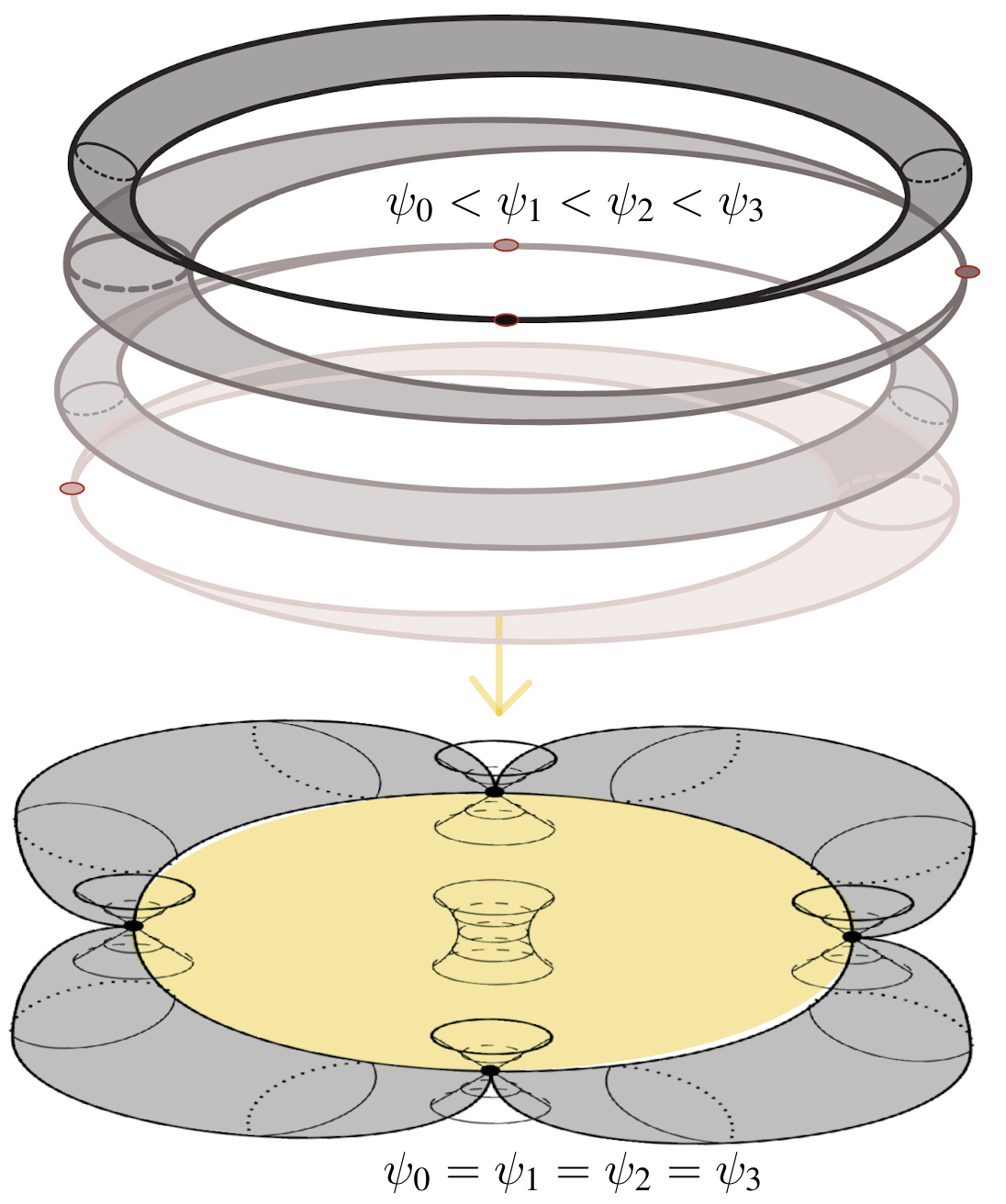}
		\caption{ \scriptsize 
			The bottom figure is indebted to \cite{evans_2021_book}.}
		\label{figure_collision_singular_fiber}
	\end{subfigure}
	\hfill
	\begin{subfigure}{0.44\textwidth}
		\includegraphics[width=6cm]{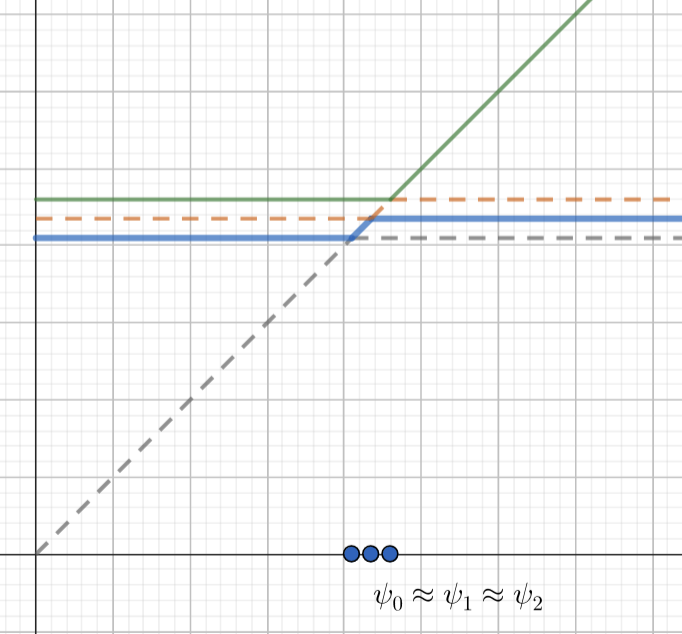}
		\caption{ \scriptsize 
			Order statistics degenerate to min, max, or constant}
		\label{figure_order_statistics_degenerate}
	\end{subfigure}
	\caption{\scriptsize 
		Here $\psi_j=\psi(s,|a_j|)$. Roughly, collapsing singular Lagrangian fibers of $\pi$ is `mirror' to degenerating order statistics in $f$.}
\end{figure}

Focusing back on the SYZ conjecture, our result suggests that "\textit{the braid group action corresponds to a dynamic process of singular locus as its components collide and scatter}".

The explicit formula (\ref{f_eq_intro}) is powerful as it always produces the desired dual fibration $f=f_{\mathcal P}$ with its integral affine structure and singular locus matching that of the Lagrangian fibration $\pi=\pi_{\mathcal P}$ despite the choice of $\mathcal P$. Let's first examine the dependence of singular locus $\Delta$ on the choice of $\mathcal P=\{a_0,\dots, a_n\}$.
The natural $S^1$-action has a fixed point set of $n+1$ points $(u,v,z)=(0,0,a_k)$. The singular locus $\Delta$ is given by the $\pi$-image of this set, and thus 
\[
\Delta=\Delta_{\mathcal P}=\{(0,|a_k|)\in B\mid 0\leqslant k\leqslant n\} \ .
\]
Clearly, the number of singular points may vary as some $|a_k|$ can coincide (cf. Figure \ref{figure_collision_singular_fiber} and \ref{figure_psi(s,r)}). If so, the order statistic functions in (\ref{f_eq_intro}) degenerate accordingly as shown in Figure \ref{figure_order_statistics_degenerate}.
For instance, the computation is simplest when all the norms $|a_k|$ are identical to some $\lambda$. (The calculations for other cases are nearly identical yet.)
Then, only a single singular point $(0,\lambda)$ appears in $B$.
Notice that we are utilizing the same formula as (\ref{f_eq_intro}).
All the $F_k$ for $k\neq 0, n+1$ degrade to $\psi(\val(y),\lambda)$. 
For $k=0,n+1$, the order statistics degenerate to the min / max functions.
Thus, the first component $F_0$ reduces to
\[
\textstyle 
F_0=\min\{\sum_{j=0}^{n+1} \ j\cdot \val(x_j), \ \psi(\val(y),\lambda) \}
\]
while the last one becomes 
\[
\textstyle
F_{n+1}=\max\{ \sum_{j=0}^{n+1}(j-n-1)\val(x_j) +(n+1)\min\{0,\val(y)\}, \ \psi(\val(y),\lambda) \} \ .
\]

We further explore the generators of the braid group $B_{n+1}$.
A preferred isomorphism $B_{n+1}\cong \pi_1(\mathscr C)$ relies on the choice of a basic set of curves \cite[3b]{Seidel_Khovanov2002quivers}. In our context, we may choose it to be the previous $c_1,\dots, c_n$ together with another path $c_{n+1}$ from $a_n$ to infinity.
Now, for $1\leqslant k\leqslant n$, the $k$-th generator $\tau_k$ of $B_{n+1}$ is given by the half-twist along $c_k$, which is the path in the configuration space $\mathscr C$ that rotates the two endpoints $a_{k-1}$ and $a_k$ of $c_k$ around their midpoint counterclockwise by 180 degrees, as depicted by Khovanov and Seidel in \cite[Figure 6(b)]{Seidel_Khovanov2002quivers}.
Let's denote the corresponding path in $\mathscr C$ by $\mathcal P_t=\{a_i(t)\}$, $0\leqslant t\leqslant 1$.
There exists a moment $0<t_0<1$ at which the norms of $a_{k-1}(t)$ and $a_k(t)$ coincide.
As a consequence, the dynamic process of the singular locus $\Delta_t=\Delta_{\mathcal P_t}$ experiences a collision of two focus-focus singular points as $t\to t_0-$ and a subsequent birth of two new singular points as $t\to t_0+$. This process corresponds to the Dehn twist $\tau_k$ along the Lagrangian sphere $\mathcal L_k$, and the consequent deformation of the dual fibration $f=f_{\mathcal P}$ is expected to induce the twist functor on the mirror side.

The comprehensive theory remains in its nascent stage, especially with regard to non-archimedean geometric interpretations when moving $\mathcal P$. Perhaps we may start with investigating the dynamic process for the integral affine manifold with singularities $(B, \Delta_{\mathcal P_t})$ when $\mathcal P_t$ moves along a loop in $\mathscr C$.



\bibliographystyle{abbrv}
\addtocontents{toc}{\protect\setcounter{tocdepth}{0}}
\bibliography{mybib_An}		
\addtocontents{toc}{\protect\setcounter{tocdepth}{2}}
	
\end{document}